\documentclass[letterpaper]{amsart}
\usepackage{amsthm}
\usepackage{mathtools, amssymb, amsfonts,mathrsfs,graphicx}
\usepackage[usenames,dvipsnames]{xcolor}

\usepackage[mathscr]{eucal}

\usepackage[all]{xy}
\usepackage{tikz}
\usetikzlibrary{matrix,arrows,calc,decorations.pathmorphing}

\theoremstyle{plain}
\newtheorem{theorem}{Theorem}[section]
\newtheorem*{theorem*}{Theorem}

\newtheorem{thm}{Theorem}[section]
\newtheorem{proposition}[theorem]{Proposition}
\newtheorem{lemma}[theorem]{Lemma}

\newtheorem{cor}[theorem]{Corollary}
\newtheorem{claim}[theorem]{Claim}
\newtheorem{pdef}[theorem]{Proposition-Definition}

\theoremstyle{definition}
\newtheorem{definition}[theorem]{Definition}
\newtheorem{assumption}[theorem]{Assumption}
\newtheorem{convention}[theorem]{Convention}

\theoremstyle{remark}
\newtheorem{remark}[theorem]{Remark}
\newtheorem{example}[theorem]{Example}

\numberwithin{equation}{section}

\DeclareMathOperator{\pr}{pr}

\DeclareMathOperator{\coker}{coker}

\DeclareMathOperator{\Hom}{Hom}
\DeclareMathOperator{\Kan}{Kan}     %<-- Kan condition
\DeclareMathOperator{\Acyc}{Acyc}   %<-- Acyclic condition
\newcommand{\sse}{\subseteq}

\newcommand*{\nb}    {\nobreakdash}     %no word break after the following hyphen
\newcommand*{\defeq} {\mathrel{\vcentcolon=}} %used for definitions

\newcommand*{\Simp}[1]{\Delta^{#1}}
\newcommand*{\Horn}[2]{\Lambda^{#1}_{#2}}

\newcommand*{\id}{\textup{id}}%identity map

\newcommand{\N}{\ensuremath{\mathbb N}}
\newcommand{\Z}{\ensuremath{\mathbb Z}}
\newcommand{\C}{\ensuremath{\mathbb C}}
\newcommand{\R}{\ensuremath{\mathbb R}}
\newcommand{\pts}{\mathscr{P}}
\newcommand{\ppt}{\mathsf{p}}
\newcommand{\emb}{\hookrightarrow}
\newcommand{\Cat}{\mathsf M}% category M--->manifolds
\newcommand{\Mfd}{\mathsf{Mfd}}% category of smooth manifolds
\newcommand{\covers}{\mathcal T}% set of covers in a pretopology
\newcommand{\open}{\textup{open}}% open maps
\newcommand{\subm}{\textup{subm}}% submersions
\newcommand{\yon}{\mathbf{y}}% Yoneda embedding
\newcommand{\op}{\mathrm{op}}% opposite
\newcommand{\cO}{\mathcal{O}}

\def\smallint{{\begingroup\textstyle \int\endgroup}}

\makeatletter
\providecommand{\leftsquigarrow}{%
  \mathrel{\mathpalette\reflect@squig\relax}%
}
\newcommand{\reflect@squig}[2]{%
  \reflectbox{$\m@th#1\rightsquigarrow$}%
}
\makeatother

\newcommand{\sint}{\smallint}
\newcommand{\tlt}[1]{\tau_{< #1}}
\newcommand{\tleq}[1]{\tau_{\leq #1}}

\newcommand{\ti}[1]{\tilde{#1}}
\newcommand{\wht}[1]{\widehat{#1}}

\newcommand{\cinf}{C^{\infty}}

\newcommand{\Sh}{\mathsf{Sh}}
\newcommand{\PSh}{\mathsf{PSh}}

\newcommand{\maps}{\colon}
\newcommand{\tensor}{\otimes}
\renewcommand{\deg}[1]{\left \lvert #1 \right \rvert}

\renewcommand{\S}{\bar{S}}

\newcommand{\del}{\partial}

\newcommand{\codiag}{\nabla}                 
\newcommand{\bs}{\mathbf{s}}

\newcommand{\xto}[1]{\xrightarrow{#1}}

\newcommand{\CE}{\mathrm{CE}}

\newcommand{\Linf}{\mathsf{L}_{\infty}\mathsf{Alg}}

\newcommand{\sSh}{s\mathsf{Sh}}
\newcommand{\Set}{\mathsf{Set}}
\newcommand{\sSet}{s\mathsf{Set}}
\newcommand{\cU}{\mathcal{U}}
\newcommand{\cUp}{\tilde{\mathcal{U}}}

\renewcommand{\C}{\mathsf{C}}
\newcommand{\W}{\mathsf{W}}
\newcommand{\Wf}{\mathsf{W_f}}

\newcommand{\pset}[1]{\mathrm{2}^{#1}}
\newcommand{\g}{\mathfrak{g}}

\newcommand{\h}{\mathfrak{h}}

\newcommand{\pt}{\ast}

\newcommand{\fib}{\twoheadrightarrow}

\newcommand{\LnA}[1]{\mathsf{Lie}_{#1}\mathsf{Alg}}

\newcommand{\LnG}[1]{\mathsf{Lie}_{#1}\mathsf{Grp}}
\newcommand{\LnGpd}[1]{\mathsf{Lie}_{#1}\mathsf{Gpd}}
\newcommand{\lt}{<}
\newcommand{\OS}{\Omega(S)}
\newcommand{\Om}{\Omega}
\newcommand{\sOm}{\, \Omega}
\newcommand{\vphi}{\varphi}

\DeclareMathOperator{\spl}{\mathrm{spl}}

\DeclareMathOperator{\Ob}{\mathrm{Ob}}
\DeclareMathOperator{\Mor}{\mathrm{Mor}}

\DeclareMathOperator{\im}{\mathrm{im}}

\DeclareMathOperator{\ft}{\mathrm{fin}}
\DeclareMathOperator{\Ban}{\mathsf{Ban}}

\DeclareMathOperator{\Match}{\mathrm{Match}}

\DeclareMathOperator{\Map}{\mathrm{Map}}

\DeclareMathOperator{\cyl}{\mathrm{cyl}}

\DeclareMathOperator*{\colim}{\mathrm{colim}}
\DeclareMathOperator{\coeq}{\mathrm{coeq}}
\DeclareMathOperator{\cdga}{\mathsf{cdga}}

\DeclareMathOperator{\curv}{\mathrm{curv}}
\DeclareMathOperator{\MC}{\mathrm{MC}}

\DeclareMathOperator{\dR}{\mathrm{dR}}
\DeclareMathOperator{\cl}{\mathrm{cl}}
\DeclareMathOperator{\ddeg}{\mathrm{deg}}
\DeclareMathOperator{\sfib}{\mathrm{fib}}

\newcommand*{\emptycomment}[1]{}
\newcommand{\hypercover}{hypercover}

\newcommand{\gpd}[1]{{\mathsf{Gpd}_\infty[#1]}}

\newcommand{\lnaft}{\LnA{n}^{\ft}}

\newcommand{\el}{\ell}

\newcommand{\sgn}[1]{ (-1)^{\frac{#1(#1-1)}{2}}}

\usepackage[breaklinks=true,colorlinks=true,citecolor=blue,urlcolor=blue,linkcolor=black]{hyperref}

\begin{document}
\title[Homotopy theory for Lie $\infty$-groupoids \& application to
integration] {On the homotopy theory for Lie
  $\infty$-groupoids, with an application to integrating
  $L_\infty$-algebras}

\author{Christopher L.\  Rogers} 
\address{
 Department of Mathematics and Statistics, University of Nevada,
 Reno. \newline 1664 N. Virginia Street Reno, NV 89557-0084 USA} 
\email{chrisrogers@unr.edu, chris.rogers.math@gmail.com}

\author{Chenchang Zhu}
\address{Mathematisches Institut and Courant Research Centre ``Higher Order Structures''\\
  Georg-August-Universit\"at G\"ottingen\\
  Bunsenstra\ss e 3--5\\
  37073 G\"ottingen\\
  Germany}
\email{zhu@uni-math.gwdg.de, chenchang.zhu@gmail.com}
\subjclass[2000]{55U35; 18G30; 22A22; 17B55}
\keywords{simplicial manifold, Lie $\infty$-groupoid,
  $L_\infty$-algebra, category of fibrant objects, hypercover}

\begin{abstract}
Lie $\infty$-groupoids are simplicial Banach manifolds that
satisfy an analog of the Kan condition for
simplicial sets. An explicit construction of Henriques produces
certain Lie $\infty$-groupoids called ``Lie $\infty$-groups'' by
integrating  finite type Lie $n$-algebras. In order to study the compatibility
between this integration procedure and the homotopy theory of
Lie $n$-algebras introduced in the companion paper \cite{Rogers:2018}, 
we present a homotopy theory for Lie
$\infty$-groupoids. Unlike Kan simplicial sets and the higher
geometric groupoids of Behrend and Getzler, Lie $\infty$-groupoids do not
form a category of fibrant objects (CFO), since the category of manifolds
lacks pullbacks. Instead, we show that
Lie $\infty$-groupoids form  an ``incomplete category of
fibrant objects'' in which the weak equivalences correspond to
``stalkwise'' weak equivalences of simplicial sheaves. 
This homotopical structure enjoys many of the same properties as a  
CFO, such as having, in the presence of functorial path objects, a convenient realization of its simplicial localization.
We further prove that the acyclic fibrations are precisely the
hypercovers, which implies that many of Behrend and Getzler's results
also hold in this more general context. As an application, we show
that Henriques' integration functor is an exact functor with respect to a class of distinguished
fibrations which we call ``quasi-split fibrations''. Such fibrations include acyclic fibrations as well as fibrations that arise in string-like extensions. In particular, integration 
sends $L_\infty$ quasi-isomorphisms to weak equivalences, quasi-split fibrations to Kan fibrations, and preserves acyclic fibrations, as well as pullbacks of acyclic/quasi-split fibrations.
\end{abstract}

\maketitle

\setcounter{tocdepth}{2}
\tableofcontents

\section{Introduction}
Lie $\infty$-groupoids, introduced by Henriques \cite{Henriques:2008}, are simplicial
Banach manifolds that satisfy a certain 
diffeo-geometric analog of the ``horn filling'' condition
for Kan simplicial sets. A Lie $\infty$-groupoid for which all horns of
dimension $>n$ are filled uniquely is called a ``Lie
$n$-groupoid''. A Lie 0-groupoid is a just a Banach manifold,
while a Lie 1-groupoid is the nerve of a Lie groupoid. In general,
 Lie $n$-groupoids serve as models for differentiable $n$-stacks.

Important examples of Lie $n$-groupoids are ``Lie $n$-groups''.  These
are Lie $n$-groupoids that have a single 0-simplex (i.e.\ reduced Lie
$n$-groupoids). Lie $n$-groups have been used to construct
diffeo-geometric models for the higher stages of the Whitehead tower
of the orthogonal group. The most famous of these is the ``String Lie
2-group'': its geometric realization is a topological group whose
homotopy type is the 3-connected cover of the orthogonal group.
Initial interest in the String 2--group stemmed from its appearance in
string theory and possible applications to geometric models of
elliptic cohomology. (See Sec.\ 1.2 of \cite{Henriques:2008} for a
summary and also Sec.\ 7 of \cite{Wolfson:2016}.)

With these applications in mind, Henriques developed a smooth
analog of Sullivan's realization functor from rational homotopy theory
which produces Lie $n$-groups by ``integrating'' finite type Lie $n$-algebras.
Lie $n$-algebras are non-negatively graded chain
complexes concentrated in the first $n-1$ degrees, equipped with a
collection of multi-linear brackets which satisfy a coherent homotopy
analog of the Jacobi identity for differential graded Lie
algebras. (These are also known as $n$-term $L_\infty$-algebras, or
$n$-term homotopy Lie algebras \cite{Baez-Crans:2004, Lada-Markl:1995}.) 
A Lie 1-algebra is just a Lie algebra.

Finite type Lie $n$-algebras have a good notion of a homotopy theory
which can be modeled in a variety of ways \cite{DHR:2015,Rogers:2018,Vallette:2014}.
In these contexts, morphisms between Lie $n$-algebras are
significantly ``weaker'' than just linear maps which preserve the
brackets. But every morphism between Lie $n$-algebras induces a chain
map between the underlying complexes. A morphism between Lie
$n$-algebras is a weak equivalence when the induced chain map gives an
isomorphism on the corresponding homology groups. (Such morphisms are
also known as ``$L_\infty$-quasi-isomorphisms''.)

 What is still missing from this story is a full understanding of the
 relationship between the homotopy theory of Lie $n$-algebras and the
 homotopy theory of Lie $n$-groups. The situation is understood in
 some special cases, for example, for strict Lie 2-algebras and Lie
 2-groups \cite{Noohi:2013}.  But in general, one would hope
that   Henriques' integration functor would send a weak equivalence between
Lie $n$-algebras to a weak equivalence between Lie $n$-groups.

It turns out that presenting a user-friendly homotopy theory for Lie
$n$-groupoids is a bit subtle. This is mostly due to the
fact that the category of Banach manifolds lacks several desirable
properties, such as the existence of pullbacks. 
Recently, Behrend and Getzler \cite{Behrend-Getzler:2015} showed that higher
groupoids internal to certain geometric contexts called ``descent categories'' 
form a category of fibrant objects (CFO) for a homotopy theory, in the
sense of Brown \cite{Brown:1973}. Roughly, a descent category is a
category of ``spaces'' which has all finite limits 
equipped with a distinguished class of morphisms called
``covers'' satisfying some axioms. 
Examples include the category of schemes with surjective
\'{e}tale morphisms as covers, as well as the category of Banach
analytic spaces, with surjective submersions as covers.  The
fibrations in the Behrend--Getzler CFO structure for $n$-groupoids are
natural generalizations of Kan fibrations, while the 
acyclic fibrations are precisely the so-called ``hypercovers''. 
As Behrend and Getzler show, this data completely determines
the weak equivalences between geometric $n$-groupoids via a
very nice combinatorial characterization \cite[Thm.\ 5.1]{Behrend-Getzler:2015}. 

Unfortunately, the category of Banach manifolds---regardless of the
choice of covers---does not form a descent category, since it lacks
finite limits.  Hence, Behrend and Getzler's results do not apply
directly in this context. (Nor, unfortunately, does the related work
of Pridham \cite{Pridham:2013}.)  It is worth emphasizing that
\textit{finite-dimensional} Lie $n$-groupoids form a full subcategory
of the category of $n$-groupoids internal to the category of
$\cinf$-schemes (in the sense of Dubuc \cite{Dubuc:1981}). The latter
category, as Behrend and Getzler note, is a descent category. Kan
fibrations between Lie $n$-groupoids first appeared in the work of
Henriques \cite{Henriques:2008}, while hypercovers for Lie
$n$-groupoids are featured prominently in the work of the second
author \cite{Zhu:2009a} and Wolfson \cite{Wolfson:2016}.  In
particular, the second author defines two Lie $n$-groupoids to be
``Morita equivalent'' if they are connected by a span of
hypercovers. Wolfson's work on $n$-bundles is also quite relevant
here: he generalizes aspects of Behrend and Getzler's machinery to the
context of Lie $n$-groupoids, but he did not require an explicit
presentation of their homotopy theory.

\subsection{Summary of main results}
In Theorem \ref{thm:icfo_struct}, we show that Lie $n$-groupoids form
what we call---for lack of better terminology---an ``incomplete
category of fibrant objects'' (iCFO) for a homotopy theory.  
Just like
in a category of fibrant objects, an iCFO is equipped with two important
classes of morphisms: weak
equivalences and fibrations, and as usual, the morphisms which lie in
the intersection of the two are called acyclic fibrations.  The axioms
of an iCFO (Def.\ \ref{def:catfibobj}) are identical to those of a
CFO, except we do not assume the existence of pullbacks---even along 
fibrations. We define the weak equivalences of Lie $n$-groupoids to be
those smooth simplicial morphisms which correspond, via the Yoneda
embedding, to ``stalkwise weak equivalences'' between the associated
simplicial sheaves.  Stalkwise weak equivalences (Def.\
\ref{def:stalk_weq}) are a natural choice for weak equivalences
between simplicial sheaves, appearing in the work of Brown
\cite{Brown:1973}, and specifically in diffeo-geometric contexts in
the work of Dugger \cite{Dugger:1999}, Nikolaus, Schreiber, and
Stevenson \cite{NSS:2015}, and Freed and Hopkins
\cite{Freed-Hopkins:2013}

The fibrations in our iCFO structure for Lie $n$-groupoids are the ``Kan
fibrations'' (Def.\ \ref{def:Kan_arrow}) introduced by Henriques. The
acyclic fibrations are, by definition, those Kan fibrations which are
also stalkwise weak equivalences. However, we show in Prop.\
\ref{prop:hypercover} and Prop.\ \ref{cor:1} that any morphism which
is both a Kan fibration and a stalkwise weak equivalence is a
hypercover (Def.\ \ref{def:equivalence}).  This is because the
category of Banach manifolds can be given the structure of what we
call a ``locally stalkwise pretopology'' (Def.\
\ref{def:LSW_covers}). This is a diffeo-geometric result, and relies on
the fact that the inverse function theorem holds for Banach
manifolds. And so the acyclic fibrations in our iCFO structure are
precisely the hypercovers, just like in the work of
Behrend and Getzler. Hence, using their results, we
demonstrate in Sec.\ \ref{sec:more-on-w-eq}
 that the weak equivalences between Lie $n$-groupoids
can be characterized completely by combinatorial data that only
involves maps between Banach manifolds. No reference to simplicial
sheaf theory is actually needed.

One advantage of defining the weak equivalences 
to be stalkwise weak equivalences is that it
allows us to connect to the homotopy theory of Lie $n$-algebras, via
Henriques' integration functor, in a straightforward way. 
We can give a rather concise proof that the integration of 
a $L_\infty$ quasi-isomorphism is a stalkwise weak equivalence (Thm.\ \ref{thm:int_weq}). 

This is, in fact, only a part of a much stronger result.
The first author has shown in a companion paper \cite{Rogers:2018} that the category of finite type Lie $n$-algebras admits a CFO structure that, roughly speaking, lifts the projective model structure on non-negatively graded chain complexes. In particular, the 
weak equivalences are precisely the $L_\infty$ quasi-isomorphisms, and the fibrations are those $L_\infty$-morphisms whose linear term is surjective in all positive degrees. We summarize this result in Thm.\ \ref{thm:lnaft_cfo} of the present paper. Furthermore, in Remark \ref{rmk:lna_fibs}, we 
use some preliminary results concerning the \textit{differentiation} of Lie $\infty$-groups
to explain why this notion of fibration between Lie $n$-algebras is the ``correct'' one for our applications.

It turns out that not every fibration of Lie $n$-algebras integrates to a Kan fibration (Remark \ref{rmk:fib_not_int}). So the integration functor is not an exact functor, in the 
usual sense of homotopical algebra. However, in Def.\ \ref{def:exactfunctor}, we introduce the notion of a functor being exact with respect to a class of distinguished fibrations. This is a mild and controllable generalization, and it provides us with a useful way to compare iCFOs.

Indeed, Thm.\ \ref{thm:int_exact} demonstrates that integration is compatible with a substantial amount of the iCFO structure. More precisely, we show that Henriques' integration functor is exact with respect to the class of so-called ``quasi-split fibrations''. This class of fibrations includes all acyclic fibrations, as well as the fibrations which naturally arise in the construction of the string Lie 2-algebra. We interpret this result as step one of a larger project in progress
whose goal is to prove an analog of ``Lie's Second Theorem'' for Lie
$n$-groups and Lie $n$-algebras.

\subsection{Outline of paper}
Throughout, we try to write for a somewhat broader audience, which
could include, for example, readers from differential/Poisson geometry with interests in
both Lie groupoid theory and $L_\infty$-algebras. We attempt
a self-contained presentation, within reason, and use 
only a minimal amount of technical machinery. We recognize that
some of the auxiliary results presented here can reside in a more general
framework well known to experts in abstract homotopy theory. 

We begin in Section \ref{sec:iCFO}, where we give the axioms for an incomplete
category of fibrant objects (iCFO) for a homotopy theory. We show that
many of the nice properties which hold for CFOs also hold for
iCFOs. For example, we show that the mapping space between
two objects in the Dwyer-Kan simplicial localization of a small iCFO, equipped with functorial path objects and functorial pullbacks of acyclic fibrations, can be described as the nerve of a category of spans. 
We also introduce in this section the notion of 
an functor being exact with respect to a class of distinguished fibrations. Such a functor has
many of the same properties as an exact functor between CFOs. The main difference is that
the functor is only required to preserve those fibrations contained in a distinguished subclass of fibrations in the source category. 

In Section \ref{sec:higher_gpd}, we recall the definitions of an
$n$-groupoid object and a Kan fibration in a large category (such as
the category of Banach manifolds) equipped with a
``pretopology''. Many of the basic constructions in this section and
throughout the paper require taking limits in this category which a
priori do not exist. Hence, limits must be treated as limits of
sheaves, and then shown to be representable. Furthermore, we later on define weak equivalences between
$n$-groupoids in terms of the simplicial Yoneda embedding.  We
therefore need to deal with sheaves over large categories. To resolve
any set-theoretical problems, we employ the standard workaround by
passing to a larger Grothendieck universe.  In theory, this could
introduce a dependence on this ``enlargement'' (e.g., see \cite{Waterhouse:1975}).
So we show explicitly in Appendix \ref{sec:universe} that, for the
case of Lie $n$-groupoids, all of our results are independent of
choice of Grothendieck universe.

In Section \ref{sec:points}, we recall the notion of ``points'' for
categories of sheaves, which generalize the notion of stalks. 
In particular, we consider a collection of points for sheaves over the
category of Banach manifolds, which generalizes those
found in the literature (e.g.,  \cite{Dugger:1999,NSS:2015} ) for
finite-dimensional manifolds. We also collect in this section some useful results
regarding matching objects and various notions of epimorphisms and surjections for
pretopologies equipped with a collection of points. We use these
notions in Section \ref{sec:stalk_weqs} to define stalkwise weak
equivalences between $n$-groupoids. We also show in this section 
that any morphism between $n$-groups (reduced $n$-groupoids) which
induces an isomorphism of the corresponding simplicial homotopy group sheaves (in
the sense of Joyal \cite{Joyal:1984} and Henriques \cite{Henriques:2008}) 
is a stalkwise weak equivalence. We
use this result in Section \ref{sec:int_exact} where we consider the integration $L_\infty$ quasi-isomorphisms.

The definition of hypercover is recalled in Section
\ref{sec:hypercovers}, and we introduce the notion of a category
equipped with a ``locally stalkwise pretopology''. We show that the
category of Banach manifolds
equipped with the surjective submersion pretopology is an example.
In Section \ref{sec:gpd_icfo}, we prove $\infty$-groupoids in a
category equipped with a locally stalkwise pretopology form an iCFO.
We note that this proof can be easily refined to show that $n$-groupoids
for finite $n$ also have an iCFO structure. We also demonstrate in this
section that the weak equivalences for this iCFO structure can be
described without the need of simplicial sheaves, in analogy with a
result of Behrend and Getzler.

In Section \ref{sec:LnA},  we recall from \cite{Rogers:2018} the CFO structure on the category of finite type Lie $n$-algebras. We also summarize a number of technical results from \cite{Rogers:2018}
concerning: the ``strictification'' of fibrations between Lie $n$-algebras, the properties of Maurer-Cartan sets, and the decomposition of Postnikov towers.
We also introduce the notion of a quasi-split fibration between Lie $n$-algebras, and
we provide some motivating examples.

Finally, in Section \ref{sec:int_exact} we extend 
Henriques' results on the integration of Postnikov towers of Lie $n$-algebras.
We use these to show that the integration of a quasi-split fibration is a Kan fibration, and that the integration of a $L_\infty$ quasi-isomorphism is a stalkwise weak equivalence.
We conclude with Theorem \ref{thm:int_exact}, which is our main result: The integration functor is exact with respect to the class of quasi-split fibrations.

\subsection{Acknowledgments}
The authors gratefully thank Ezra Getzler for many helpful
conversations, and especially for sharing an early version of his
preprint with Kai Behrend. CLR thanks Dave Carchedi for helpful discussions. 
The authors also wish to thank an anonymous referee for their suggestions on 
improving the exposition and strengthening the results of this work.

Parts of this paper were written during CLR's stay
at the Mathematical Research Institute of the University of Melbourne (MATRIX). He thanks the institute and the organizers of the 2016 Higher Structures in Geometry and Physics program for their hospitality and support. CLR acknowledges support by an AMS-Simons Travel Grant, and
both authors acknowledge support by the DFG Individual Grant (ZH 274/1-1) ``Homotopy Lie Theory''.

\section{Incomplete categories of fibrant objects}
\label{sec:iCFO}
In this section, we introduce a slight generalization of Brown's
definition for a category of fibrant objects (CFO) for a homotopy theory
\cite[Sec.\ 1]{Brown:1973}.
In particular, we do not assume the existence of certain limits in the
underlying category, hence the term ``incomplete''. This is reasonable
given our applications to $\infty$-groupoid objects in diffeo-geometric
categories. We note that other variations on weakening Brown's axioms 
have already appeared in the literature, for example Horel's ``partial Brown categories'' \cite{Horel:2016}. (See Ex.\ \ref{ex:exactfunctor} below.)

\begin{definition} \label{def:catfibobj}
Let $\C$ be a category with finite products and terminal object $\ast
\in \C$ equipped with two classes of morphisms called
\textbf{weak equivalences} and \textbf{fibrations}. A morphism which
is both a weak equivalence and a fibration is called an
\textbf{acyclic fibration}. We say $\C$ is an
{\bf incomplete category of fibrant objects (iCFO)} iff:
\begin{enumerate}
\item{Every isomorphism in $\C$ is an acyclic fibration.}

\item{The class of weak equivalences satisfies  ``2 out of 3''. That is, if
    $f$ and $g$ are composable morphisms in $\C$ and any two of $f,g, g
    \circ f$ are weak equivalences, then so is the third.}

\item{The composition of two fibrations is a fibration.}

\item{If the pullback of a fibration exists, then it is a fibration.
That is, if $Y \xto{g} Z \xleftarrow{f} X$ is a diagram in $\C$ with $f$
    a fibration, and if $X \times_{Z} Y$ exists, then
   the induced projection $X \times_{Z} Y \to Y$ is a  fibration.}

\item{The pullback of an acyclic fibration exists, and is an acyclic fibration.
That is, if $Y \xto{g} Z \xleftarrow{f} X$ is a diagram in $\C$ with $f$
    an acyclic fibration, then the pullback $X \times_{Z} Y$ exists, and
   the induced projection $X \times_{Z} Y \to Y$ is an acyclic  fibration.}

\item{For any object $X \in \C$ there exists a (not necessarily
    functorial) \textbf{path object}, that is, an object
    $X^{I}$ equipped with morphisms
\[
X \xto{s} X^{I} \xto{(d_0,d_1)} X \times X,
\]
such that $s$ is a weak equivalence, $(d_0,d_1)$ is a fibration, and their
composite is the diagonal map.}

\item{All objects of $\C$ are \textbf{fibrant}. That is, for any $X \in \C$ the unique map 
$ X \to \ast$ is a fibration.}
\end{enumerate}
\end{definition}

\begin{remark}
The only difference between the above and the original definition of Brown is axiom (4). Brown requires that the pullback of a fibration always exists.
\end{remark}

\subsection{Factorization}
An important feature of a category of fibrant objects is Brown's
factorization lemma in \cite[Sec.\ 1]{Brown:1973}. 
The factorization lemma also holds for any iCFO.

Let $Y^I$ be a path object for an object $Y$ in 
an iCFO $\C$. Since $Y \to \ast$ is a fibration,
Def.\ \ref{def:catfibobj} implies that the projections $\pi_{i} \maps Y \times Y \to Y$
are fibrations. Hence, the morphisms $d_{i} \maps Y^I \to Y$  are also
fibrations. Moreover, since $d_i s = \id_Y$, and $s$ is a weak
equivalence, the maps $d_{i}$ are acyclic fibrations. Thus we have proven
\begin{lemma} \label{lem:di}
The projection $d_i: Y^I\to Y$ is an acyclic fibration for $i=0,1$.
\end{lemma}

%The above lemma can be used to prove the existence of a factorization.
The proof of the next lemma is identical to the analogous lemma for
categories of fibrant objects. Indeed, the proof does not require the
existence of pullbacks along arbitrary fibrations.

\begin{lemma} \label{lemma:fact} 
If $\C$ is an iCFO and $f\maps X\to Y$ is a morphism in $\C$, then $f$
can be factored as $f=p \circ i$, where $p$ is a fibration, and $i$ is a 
weak equivalence which is a section (right inverse) of an acyclic fibration. 
\end{lemma}

\begin{proof}
Let $Y^I$ be a path object for $Y$. 
Lemma \ref{lem:di} implies that the pullback $X \times_Y Y^I$ of
the diagram
\[
X \xto{f} Y \xleftarrow{d_0} Y^I
\] 
exists, and hence the projection $\pr_1 \maps X
\times_{Y}Y^{I} \to X$ is an acyclic fibration. Combining this fact
with the commutative diagram 
\[
\xymatrix{
X  \ar[d]_{\id} \ar[r]^{sf} & Y^I \ar[d]^{d_0}\\
X \ar[r]^{f} & Y \\
}
\]
implies that $\pr_1$ has a right inverse
\[
%\begin{equation} \label{eq:fac1}
i \maps X \to X \times_{Y} Y^I 
\]
%\end{equation}
which is necessarily a weak equivalence.  Moreover, if $p \maps X
\times_{Y} Y^I \to Y$  is the composition
%\begin{equation} \label{eq:fac2}
\[
X \times_{Y} Y^{I} \xto{\pr_2} Y^{I} \xto{d_1} Y.
\]
%\end{equation}
then 
%\begin{equation} \label{eq:fac3}
\[
f = p \circ i.
\]
%\end{equation}
To show that $p$ is a fibration, we observe that $X \times_Y Y^I$ is
also the pullback of the diagram
\[
\begin{tikzpicture}[descr/.style={fill=white,inner sep=2.5pt},baseline=(current  bounding  box.center)]
\matrix (m) [matrix of math nodes, row sep=2em,column sep=3em,
  ampersand replacement=\&]
  {  
X \times_{Y} Y^I \& Y^I \\
X \times Y \& Y \times Y \\
}; 
  \path[->,font=\scriptsize] 
   (m-1-1) edge node[auto] {$\pr_2$} (m-1-2)
   (m-1-1) edge node[auto,swap] {$(\pr_1,p)$} (m-2-1)
   (m-1-2) edge node[auto] {$(d_0,d_1)$} (m-2-2)
   (m-2-1) edge node[auto] {$(f,\id_Y)$} (m-2-2)
  ;

%begin pullback symbol%
  \begin{scope}[shift=($(m-1-1)!.4!(m-2-2)$)]
  \draw +(-0.25,0) -- +(0,0)  -- +(0,0.25);
  \end{scope}
  %end pullback symbol%

\end{tikzpicture}
\]
Since $(d_0,d_1)$ is a fibration, $(pr_1,p)$ is a fibration. Since the
projection $X \times Y \to Y$ is a fibration (indeed $Y$ is fibrant),
it follows that $p$ is also a fibration.

\end{proof}

One simple fact which follows from the factorization lemma is that
every weak equivalence in an iCFO
yields a span of acyclic fibrations. Also, just like in the case with
a CFO \cite[Lemma 1.3]{Behrend-Getzler:2015}, the weak equivalences in
an iCFO are determined by the acyclic fibrations.

\begin{proposition} \label{prop:catfib_morita_eq}
\mbox{}
\begin{enumerate}
\item  If $f \maps X \to Y$ is a weak equivalence in an iCFO
then there exists a span of acyclic fibrations $X \leftarrow Z \to Y$.

\item A morphism $f \maps X \to Y$ in an iCFO is a weak equivalence if
  and only if it factors as $f=p \circ i$, where $p$ is an acyclic
  fibration,  and  $i$ is a section of an acyclic fibration. 
\end{enumerate}
\end{proposition}

\begin{proof}
If $f$ is a weak equivalence, then factor  $f=p\circ
  i$ as in Lemma \ref{lemma:fact}. Since $f$ and $i$ are weak equivalences,
  $p$ is an acyclic fibration. Hence, the legs of the span
\[
X \xleftarrow{\pr_1}  X \times_{Y} Y^I  \xto{p} Y
\]
are acyclic fibrations. 
\end{proof}

\begin{remark}\label{rmk:morita}
Spans of acyclic fibrations are analogous to the notion of Morita equivalence between
Lie groupoids. Indeed, in \cite[Def.\ 2.12]{Zhu:2009a}, two Lie
$n$-groupoids are considered ``Morita equivalent'' iff they are
connected by a span of maps called hypercovers, which we consider in Sec.\ \ref{sec:hypercovers}. 
We will see in Sec.\ \ref{sec:gpd_icfo} that hypercovers are the
acyclic fibrations in the iCFO structure for Lie $n$-groupoids.
\end{remark}

Next, we provide a working definition for the notion of an exact functor between iCFOs, which is well suited for the applications we have in mind.  

\begin{definition} \label{def:distingfib}
Let $\C$ be an iCFO and denote by $\sfib(\C)$ the class of fibrations of $\C$.
A class of fibrations $S \sse \sfib(\C)$  is
\textbf{distinguished} iff $S$ contains all acyclic fibrations and all morphisms into the terminal object of $\C$.
\end{definition}

Every iCFO has two obvious classes of distinguished fibrations: the class of all fibrations
$S_{\mathrm{max}} = \sfib(\C)$, and the class $S_{\mathrm{min}}$ which only contains
the acyclic fibrations and the morphisms into the terminal object.

\begin{definition} \label{def:exactfunctor}
Let $\C$ be an iCFO and $S$ a subclass of distinguished fibrations of $\C$.  A functor $F \maps \C \to \C'$ between iCFOs is a \textbf{(left) exact functor with respect to $S$} iff 
\begin{enumerate}
\item $F$ preserves the terminal object and acyclic fibrations, 
\item $F$ maps every fibration in $S$ to a fibration in $\C'$, and
\item any pullback square in $\C$ of the form
\[
\begin{tikzpicture}[descr/.style={fill=white,inner sep=2.5pt},baseline=(current  bounding  box.center)]
\matrix (m) [matrix of math nodes, row sep=2em,column sep=3em,
  ampersand replacement=\&]
  {  
P \& X \\
Z \& Y \\
}
; 
  \path[->,font=\scriptsize] 
   (m-1-1) edge node[auto] {$$} (m-1-2)
   (m-1-1) edge node[auto,swap] {$$} (m-2-1)
   (m-1-2) edge node[auto] {$f$} (m-2-2)
   (m-2-1) edge node[auto] {$$} (m-2-2)
  ;

%begin pullback symbol%
  \begin{scope}[shift=($(m-1-1)!.4!(m-2-2)$)]
  \draw +(-0.25,0) -- +(0,0)  -- +(0,0.25);
  \end{scope}
  %end pullback symbol%
\end{tikzpicture}
\]
with $f$ a fibration in $S$, is mapped by $F$ to a pullback square in $\C'$.
\end{enumerate}
% it sends weak equivalences to weak equivalences, acyclic fibrations to acyclic fibrations, and pullbacks of acyclic fibrations to pullbacks.
\end{definition}

Before we give a few examples of such functors, we note the following simple lemma:
\begin{lemma} \label{lem:exactfunctor}
Let $F \maps \C \to \C'$ be an exact functor with respect to a class $S$.
Then $F$ preserves finite products and weak equivalences.
\end{lemma} 
\begin{proof}
It follows from Def.\ \ref{def:distingfib} and axiom 3 above that $F$ preserves finite products. If $f \maps X \to Y$ is a weak equivalence in $\C$, then Lemma \ref{lemma:fact} implies that $f=p \circ i$, where $p$ is an acyclic fibration and $i$ is a right inverse of an acyclic fibration. Since $F$ preserves acyclic fibrations, it follows that $F(f)$ is a weak equivalence.  
\end{proof}

\begin{example} \label{ex:exactfunctor}
\mbox{}
\begin{enumerate}
\item Let $F \maps \C \to \C'$ be a functor between categories of fibrant objects (CFOs).
Then $F$ is an exact functor between CFOs in the usual  sense (Def.\ 2.3.3 \cite{BHH:2017}) if and only if $F$ is exact with respect to the class $S_{\mathrm{max}}$. 

\item An iCFO equipped with a functorial path object (see Sec.\ \ref{subsubsec:func_path_obj}) has the structure of a ``partial Brown category'' (PBC), a notion  introduced by Horel in \cite[Rmk.\ 2.4, Prop.\ 2.6]{Horel:2016}. Following Horel's definition of an exact functor between PBCs \cite[Def.\ 2.5]{Horel:2016}, we say a functor between iCFOs is \textbf{partially exact} 
iff it sends weak equivalences to weak equivalences, acyclic fibrations to acyclic fibrations, and pullbacks of acyclic fibrations to pullbacks. In contrast with the first example, a product preserving functor $F \maps \C \to   \C'$  between iCFOs is partially exact if and only if $F$ is exact with respect to the class $S_{\mathrm{min}}$. 

\item In our main application, Thm. \ref{thm:int_exact}, we show that the integration functor from finite type Lie $n$-algebras to Lie $\infty$-groupoids is exact with respect to a class of distinguished fibrations called ``quasi-split fibrations'' (Def.\ \ref{def:split_fib}), which lies properly between $S_{\mathrm{min}}$ and $S_{\mathrm{max}}$. 
\end{enumerate}
\end{example}

\subsection{Simplicial localization} \label{sec:simp_loc}
In this section, we show that in certain cases the simplicial localization (or underlying
$\infty$-category) of a small incomplete category of fibrant objects
has a simple description in terms of the nerve of a category of spans. 
An example is a small category of $\infty$-groupoids equipped with the iCFO structure described in Sec.\ \ref{sec:gpd_icfo}. In particular, the simplicial localization of the category of Lie $n$-groupoids is discussed in Sec.\ \ref{sec:simp_loc_gpd}, and this will be useful for our future work concerning the higher category theory of Lie $\infty$-groupoids.

Recall that for any small category $\C$ and a wide subcategory $\W$ of weak
equivalences, one may associate to it a simplicial category $L_{\W}\C$
(i.e., a category enriched in simplicial sets) via the Dwyer-Kan
simplicial localization, or ``hammock localization'' 
\cite{Dwyer-Kan:1980}. The simplicial category $L_{\W}\C$ has the same
objects as $\C$ and is universal with respect to the property that
weak equivalences in $W$ are homotopy equivalences in $L_{\W}\C$. In
particular, $\pi_0 L_{\W}\C$, the category whose objects are those of
$\C$, and whose hom-sets are
\[
\pi_0\big( \mathrm{Map}_{L_{\W} \C}(X,Y) \big), \quad \forall X,Y \in \C,
\]
is equivalent to the usual localization $\C[\W^{-1}]$, the category
obtained by formally inverting the morphisms in $\W$.

The mapping space $\mathrm{Map}_{L_{\W} \C}(X,Y)$ is  
the direct limit of nerves of categories \cite[Prop.\ 5.5]{Dwyer-Kan:1980}. Dwyer and Kan showed that
when $\C$ is a model category with functorial factorizations,   
$\mathrm{Map}_{L_{\W} \C}(X,Y)$ can be described more simply, up to weak homotopy equivalence, as the nerve of a single category of spans (e.g., Prop.\ 8.2 in \cite{Dwyer-Kan:1980} and subsequent corollaries). Weiss \cite{Weiss:1999} later showed that Dwyer and Kan's proof extends to the case when $\C$ is a Waldhausen category (i.e.\ a category with cofibrations and weak equivalences) equipped with both a cylinder functor, and canonical pushouts of cofibrations. Roughly speaking, the ``opposite'' of Weiss's argument is spelled out in detail for CFOs in the work of Nikolaus, Schreiber, and Stevenson 
\cite[Thm.\ 3.61]{NSS:2015}. As we show below, an analogous result holds for iCFOs.

\subsubsection{Functorial path objects and pullbacks} \label{subsubsec:func_path_obj}
In this section, weak equivalences and fibrations will be denoted by $X \xto{\sim} Y$ and $X \fib Y$, respectively. Furthermore, in this section, $\C$ will denote a small incomplete category of fibrant objects equipped 
with:
\begin{itemize}

\item \textbf{functorial path objects}: An assignment of a path object  $X^I$ to each object $X \in \C$ and to each $f \maps X \to Y$, a morphism $f^I \maps X^I \to Y^I$ such that the following diagram commutes:
\[
\begin{tikzpicture}[descr/.style={fill=white,inner sep=2.5pt},baseline=(current  bounding  box.center)]
\matrix (m) [matrix of math nodes, row sep=2em,column sep=3em,
  ampersand replacement=\&]
  {  
X \& X^{I} \& X \times X \\
Y \& Y^{I} \& Y \times Y \\
}; 
  \path[->,font=\scriptsize] 
   (m-1-1) edge node[auto] {$s$} node[auto,below] {$\sim$} (m-1-2.185)
   (m-2-1) edge node[auto] {$s$} node[auto,below] {$\sim$} (m-2-2.185)
   (m-1-1) edge node[auto] {$f$} (m-2-1)
   (m-1-2) edge node[auto] {$f^I$} (m-2-2)
   (m-1-3) edge node[auto] {$(f,f)$} (m-2-3)
  ;
  \path[->>,font=\scriptsize] 
   (m-1-2) edge node[auto] {$(d_0,d_1)$} (m-1-3)
   (m-2-2) edge node[auto] {$(d_0,d_1)$} (m-2-3)
;
\end{tikzpicture}
\]

\item \textbf{functorial pullbacks of acyclic fibrations}: An
  assignment to each diagram of the form  $X \xto{f} Y \xleftarrow{g} Z$, in which $g$ is an acyclic
  fibration, a universal cone  $X \xleftarrow{ \tilde{g}} X \times_Y Z \xto{\tilde{f}} Z$ (which exists via the iCFO axioms, and in which $\tilde{g}$ is necessarily an acyclic fibration). The universal property implies that to 
each commutative diagram:
\[
\begin{tikzpicture}[descr/.style={fill=white,inner sep=2.5pt},baseline=(current  bounding  box.center)]
\matrix (m) [matrix of math nodes, row sep=2em,column sep=3em,
  ampersand replacement=\&]
  {  
X \& Y  \& Z  \\
X' \& Y' \& Z' \\
}; 
  \path[->,font=\scriptsize] 
   (m-1-1) edge node[auto] {$f$} (m-1-2)
   (m-2-1) edge node[auto] {$f'$} (m-2-2)
   (m-1-1) edge node[auto,swap] {$\alpha$} (m-2-1)
   (m-1-2) edge node[auto] {$\beta$} (m-2-2)
   (m-1-3) edge node[auto] {$\gamma$} (m-2-3)
   % (m-1-2) edge node[auto] {$f^I$} (m-2-2)
   % (m-1-3) edge node[auto] {$(f,f)$} (m-2-3)
  ;
  \path[->>,font=\scriptsize] 
  (m-1-3) edge node[auto,swap] {$g$} node[auto,below] {$\sim$} (m-1-2)
  (m-2-3) edge node[auto,swap] {$g'$} node[auto,below] {$\sim$} (m-2-2)

;

\end{tikzpicture}
\] 
in which the right horizontal morphisms are acyclic fibrations, we obtain a unique commutative  
diagram
\[
\begin{tikzpicture}[descr/.style={fill=white,inner sep=2.5pt},baseline=(current  bounding  box.center)]
\matrix (m) [matrix of math nodes, row sep=2em,column sep=3em,
  ampersand replacement=\&]
  {  
X \& X \times_Y Z  \& Z  \\
X' \& X' \times_{Y'} Z' \& Z' \\
}; 
  \path[->,font=\scriptsize] 
   (m-1-2) edge node[auto] {$\tilde{f}$} (m-1-3)
   (m-2-2) edge node[auto] {$\tilde{f}'$} (m-2-3)
   (m-1-1) edge node[auto,swap] {$\alpha$} (m-2-1)
   (m-1-2) edge node[auto] {$$} (m-2-2)
   (m-1-3) edge node[auto] {$\gamma$} (m-2-3)
   % (m-1-2) edge node[auto] {$f^I$} (m-2-2)
   % (m-1-3) edge node[auto] {$(f,f)$} (m-2-3)
  ;
  \path[->>,font=\scriptsize] 
   (m-1-2) edge node[auto,swap] {$\tilde{g}$} node[auto,below] {$\sim$} (m-1-1)
   (m-2-2) edge node[auto,swap] {$\tilde{g}'$} node[auto,below] {$\sim$} (m-2-1)

;

\end{tikzpicture}
\] 
in which the left horizontal morphisms are acyclic fibrations.
\end{itemize}

Functorial path objects and pullbacks as above provide an iCFO
with functorial factorizations in the sense of \cite[Def.\ 12.1.1]{Riehl:2014}. 
\begin{proposition} \label{prop:func_fact}
Let $\C$ be an iCFO with functorial path objects and functorial pullbacks of acyclic fibrations.
Then each morphism $f \maps X \to Y$ in $\C$ can be canonically factored as
\[
X \xto{i_f} X \times_Y Y^I \xto{p_f} Y
\]
where $p_f$ is a fibration and $i_f$ is a right inverse of an acyclic fibration. Moreover, this factorization is natural: Given a commutative diagram 
\[
\begin{tikzpicture}[descr/.style={fill=white,inner sep=2.5pt},baseline=(current  bounding  box.center)]
\matrix (m) [matrix of math nodes, row sep=2em,column sep=3em,
  ampersand replacement=\&]
  {  
X \& Y  \\
X' \& Y' \\
}; 
  \path[->,font=\scriptsize] 
   (m-1-1) edge node[auto] {$f$} (m-1-2)
   (m-2-1) edge node[auto] {$f'$} (m-2-2)
   (m-1-1) edge node[auto,swap] {$\alpha$} (m-2-1)
   (m-1-2) edge node[auto] {$\beta$} (m-2-2)
  ;
\end{tikzpicture}
\]

there exists a unique morphism $\gamma \maps X \times_Y Y^I \to X' \times_{Y'} {Y'}^I$ such that the diagram
\[
\begin{tikzpicture}[descr/.style={fill=white,inner sep=2.5pt},baseline=(current  bounding  box.center)]
\matrix (m) [matrix of math nodes, row sep=2em,column sep=3em,
  ampersand replacement=\&]
  {  
X \& X \times_Y Y^I \&  Y  \\
X' \& X' \times_{Y'} {Y'}^I \&  Y'  \\
}; 
  \path[->,font=\scriptsize] 
  (m-1-1) edge node[auto] {$i_f$} node[auto,below] {$\sim$} (m-1-2.182)
  (m-2-1) edge node[auto] {$i_{f'}$} node[auto,below] {$\sim$} (m-2-2.182)
  (m-1-1) edge node[auto] {$\alpha$} (m-2-1)
  (m-1-2) edge node[auto] {$\gamma$} (m-2-2)
  (m-1-3) edge node[auto] {$\beta$} (m-2-3)
   ;

 \path[->>,font=\scriptsize] 
  (m-1-2) edge node[auto] {$p_f$} (m-1-3)
  (m-2-2) edge node[auto] {$p_{f'}$} (m-2-3)
;
\end{tikzpicture}
\]
commutes.
\end{proposition}
\begin{proof}
One simply repeats the proof
of the factorization lemma for an iCFO (Lemma \ref{lemma:fact}) using the functorial path objects and pullbacks. 
\end{proof}

\subsubsection{Categories of spans in $\C$ }
Denote by $\Wf \subseteq \W \subseteq \C$ the subcategories of $\C$ consisting of acyclic fibrations and weak equivalences, respectively. For each pair of objects $X,Y \in \C$ we denote by  
$\C\W_f^{-1}(X,Y)$ (respectively, $\C\W^{-1}(X,Y)$) the category whose objects are spans in $\C$ of the form
\[
X \leftarrow C \rightarrow Y
\] 
in which the left arrow is an acyclic fibration (respectively, weak equivalence), and whose morphisms are commutative diagrams of the form
\[
\begin{tikzpicture}[descr/.style={fill=white,inner sep=2.5pt},baseline=(current  bounding  box.center)]
\matrix (m) [matrix of math nodes, row sep=1em,column sep=3em,
  ampersand replacement=\&]
  {  
 \& C \&    \\
X \&  \&  Y  \\
 \& C' \&   \\
}; 
  \path[->,font=\scriptsize] 
  (m-1-2) edge node[auto,swap] {$f$} (m-2-1)
  (m-3-2) edge node[auto] {$f'$} (m-2-1)
  (m-1-2) edge node[auto] {$g$} (m-2-3)
  (m-3-2) edge node[auto,swap] {$g'$} (m-2-3)
  (m-1-2) edge node[above,sloped] {$\sim$} (m-3-2)
   ;
\end{tikzpicture}
\]
in which the vertical arrow is a weak equivalence.

We will also need generalizations of the above categories of spans. In what follows, we use Dwyer and Kan's notation for ``hammock graphs'' \cite[Sec.\ 5.1]{Dwyer-Kan:1980}. Let $\mathbf{w}$ be a word of length $n \geq 1$ consisting of letters $\{\C,\W,\W^{-1},\Wf^{-1}\}$. For each pair of objects $X,Y \in \C$ we denote by $\mathbf{w}(X,Y)$ the category whose objects are diagrams in $\C$ of the form
\[
\begin{tikzpicture}[descr/.style={fill=white,inner sep=2.5pt},baseline=(current  bounding  box.center)]
\matrix (m) [matrix of math nodes, row sep=2em,column sep=3em,
  ampersand replacement=\&]
  {  
X \& C_1 \& C_2 \& \cdots \& C_{n-1} \& Y\\
}; 
  \path[font=\scriptsize] 
(m-1-1) edge node[auto] {$f_0$} (m-1-2)
(m-1-2) edge node[auto] {$f_1$} (m-1-3)
(m-1-3) edge node[auto] {$f_2$} (m-1-4)
(m-1-4) edge node[auto] {$f_{n-2}$} (m-1-5)
(m-1-5) edge node[auto] {$f_{n-1}$} (m-1-6)
;
\end{tikzpicture}
\]
in which the morphism $f_i$ goes to the right and is in $\C$ (resp.\ $\W$)
iff the $(n-i)$th letter in $\mathbf{w}$ is $\C$ (resp.\ $\W$). Otherwise, 
the morphism $f_i$ goes to the left and is in $\W$ (resp.\ $\Wf$)
iff the $(n-i)$th letter in $\mathbf{w}$ is $\W^{-1}$ (resp.\ $\Wf^{-1}$). Morphisms in $\mathbf{w}(X,Y)$ are commuting diagrams in which all vertical arrows are weak equivalences.

\begin{proposition}\label{prop:simploc1}
Let $X,Y$ be objects of $\C$.  The inclusion functors
\[
\C \Wf^{-1}(X,Y) \stackrel{\iota}{\emb} \C \W^{-1} (X,Y) 
\]
\[
\W \Wf^{-1}(X,Y) \stackrel{\iota}{\emb} \W \W^{-1} (X,Y) 
\]
induce simplicial homotopy equivalences between the corresponding nerves
\[
N\C \Wf^{-1}(X,Y) \xto{\simeq} N \C \W^{-1} (X,Y), \quad  
N\W \Wf^{-1}(X,Y) \xto{\simeq} N \W \W^{-1} (X,Y) 
\]
\end{proposition}
\begin{proof}
We use the fact that a natural transformation between functors induces a homotopy between the corresponding simplicial maps between nerves. Denote by $F \maps \C \W^{-1} (X,Y) \to 
\C \Wf^{-1}(X,Y)$ the functor which assigns to a span of the form
\[
\begin{tikzpicture}[descr/.style={fill=white,inner sep=2.5pt},baseline=(current  bounding  box.center)]
\matrix (m) [matrix of math nodes, row sep=2em,column sep=3em,
  ampersand replacement=\&]
  {  
X \& C \&  Y\\
}; 
  \path[->,font=\scriptsize] 
(m-1-2) edge node[auto,swap]{$f$} node[auto]{$\sim$} (m-1-1)
(m-1-2) edge node[auto] {$g$} (m-1-3)
;
\end{tikzpicture}
\] 
an object in $\C \Wf^{-1}(X,Y)$ via the following. First, apply the functorial factorization
(Prop.\ \ref{prop:func_fact}) to the morphism $(f,g) \maps C \to X \times Y$ to obtain
\[
\begin{tikzpicture}[descr/.style={fill=white,inner sep=2.5pt},baseline=(current  bounding  box.center)]
\matrix (m) [matrix of math nodes, row sep=2em,column sep=3em,
  ampersand replacement=\&]
  {  
C \& C' \&  X \times Y\\
}; 
  \path[->,font=\scriptsize] 
(m-1-1) edge node[auto]{$i$} node[auto,swap]{$\sim$} (m-1-2)
;
 \path[->>,font=\scriptsize] 
(m-1-2) edge node[auto] {$p$} (m-1-3)
;
\end{tikzpicture}
\] 
Then composing the fibration $p$ with the projections gives a span of fibrations: 
\begin{equation} \label{eq:simploc1}
\begin{tikzpicture}[descr/.style={fill=white,inner sep=2.5pt},baseline=(current  bounding  box.center)]
\matrix (m) [matrix of math nodes, row sep=2em,column sep=3em,
  ampersand replacement=\&]
  {  
X \& C' \&  Y\\
}; 
  \path[->>,font=\scriptsize] 
(m-1-2) edge node[auto,swap]{$f'$}(m-1-1)
(m-1-2) edge node[auto] {$g'$} (m-1-3)
;
\end{tikzpicture}
\end{equation}
along with a commutative diagram
\begin{equation}\label{eq:simploc2}
\begin{tikzpicture}[descr/.style={fill=white,inner sep=2.5pt},baseline=(current  bounding  box.center)]
\matrix (m) [matrix of math nodes, row sep=1em,column sep=3em,
  ampersand replacement=\&]
  {  
 \& C \&    \\
X \&  \&  Y  \\
 \& C' \&   \\
}; 
  \path[->,font=\scriptsize] 
  (m-1-2) edge node[auto,swap] {$f$} node[sloped,below] {$\sim$} (m-2-1)
  (m-1-2) edge node[auto] {$g$} (m-2-3)
  (m-1-2) edge node[auto] {$i$} node[below,sloped] {$\sim$} (m-3-2)
   ;
  \path[->>,font=\scriptsize] 
  (m-3-2) edge node[auto] {$f'$} (m-2-1)
  (m-3-2) edge node[auto,swap] {$g'$} (m-2-3)

;
\end{tikzpicture}
\end{equation}
which, combined with the ``2 out of 3'' axiom implies that $f'$ is an acyclic fibration.
Hence, the diagram \eqref{eq:simploc1} is an object of $\C \Wf^{-1}(X,Y)$.
It is easy to see that this assignment is indeed functorial, due to the use of functorial factorizations. Moreover, the weak equivalence $i$ in the diagram \eqref{eq:simploc2}
gives natural transformations
\[
\id_{\C \W^{-1}(X,Y)} \to \iota \circ F, \quad \id_{\C \Wf^{-1}(X,Y)} \to F \circ \iota
\]
Hence, $N\iota \maps N\C \Wf^{-1}(X,Y) \to N \C \W^{-1} (X,Y)$ is a homotopy equivalence.
To show $N \W \Wf^{-1}(X,Y) \xto{\simeq} N \W \W^{-1} (X,Y)$ is a homotopy equivalence, we observe that the restriction of the functor $F$ to the subcategory $\W \W^{-1}(X,Y)$ 
has as its target $\W \Wf^{-1}(X,Y)$, thanks to the commutative diagram \eqref{eq:simploc2} and the ``2 out of 3'' axiom. 
\end{proof}

\subsubsection{$\C$ admits a homotopy calculus of right fractions} 
In the terminology of \cite[Sec.\ 5.1]{Dwyer-Kan:1980}, 
the $k$-simplices in $N\C \W^{-1} (X,Y)$ are hammocks between $X$ and $Y$
of width $k$ and type $\C \W^{-1}$. There is simplicial map (i.e., the reduction map)
%\begin{equation} \label{eq:red_map}
\[
r \maps N\C \W^{-1} (X,Y) \to \Map_{L_W  \C}(X, Y)
\]
%\end{equation}
sending such a hammock to a reduced hammock, in the sense of \cite[Sec.\ 2.1]{Dwyer-Kan:1980}. The main theorem of this section is:

\begin{theorem}\label{thm:simploc}
Let $\C$ be a small incomplete category of fibrant objects with functorial path objects and functorial pullbacks of acyclic fibrations. Then for all objects $X,Y$ of $\C$, the maps
\begin{equation} \label{eq:thm:simploc}
N\C \Wf^{-1}(X,Y) \xto{N\iota } N\C \W^{-1} (X,Y) \xto{r} \Map_{L_W  \C}(X, Y)
\end{equation}
are weak homotopy equivalences of simplicial sets.
\end{theorem} 

Prop.\ \ref{prop:simploc1} implies, of course, that the first map in \eqref{eq:thm:simploc} is a weak homotopy equivalence. To prove Thm. \ref{thm:simploc}, we just need to show that the reduction map is a weak equivalence. We do this by showing that $\C$ admits a homotopy calculus of right fractions. Then our result will follow from a result of Dwyer and Kan \cite[Prop.\ 6.2]{Dwyer-Kan:1980}.  

Let $i,j \geq 0$ be integers. Given objects $X,Y \in \C$ there is functor 
\begin{equation}\label{eq:simploc_incl}
\jmath \maps \C^{i+j}\W^{-1}(X,Y) \to \C^{i}\W^{-1}\C^{j}\W^{-1}(X,Y)
\end{equation}
which sends a diagram of the form
\[
\begin{tikzpicture}[descr/.style={fill=white,inner sep=2.5pt},baseline=(current  bounding  box.center)]
\matrix (m) [matrix of math nodes, row sep=2em,column sep=2em,
  ampersand replacement=\&]
  {  
X \& C_1 \& C_2 \& ~ \cdots ~   C_j \& C_{j+1} \& ~\cdots ~ C_{i+j} \& Y\\
}; 
  \path[->,font=\scriptsize] 
(m-1-2) edge node[auto] {$f_1$} (m-1-3)
(m-1-3) edge node[auto] {$f_2$} (m-1-4)
(m-1-4) edge node[auto] {$f_{j}$} (m-1-5)
(m-1-5) edge node[auto] {$f_{j+1}$} (m-1-6)
(m-1-6) edge node[auto] {$f_{i+j}$} (m-1-7)
;
  \path[->,font=\scriptsize] 
(m-1-2) edge node[above] {$f_0$} node[below,sloped] {$\sim$} (m-1-1)
;
\end{tikzpicture}
\]
to the diagram 
\[
\begin{tikzpicture}[descr/.style={fill=white,inner sep=2.5pt},baseline=(current  bounding  box.center)]
\matrix (m) [matrix of math nodes, row sep=2em,column sep=2em,
  ampersand replacement=\&]
  {  
X \& C_1 \& C_2 \& ~ \cdots ~   C_j \& C_{j+1} \& C_{j+1} \& ~\cdots ~ C_{i+j} \& Y\\
}; 
  \path[->,font=\scriptsize] 
(m-1-2) edge node[auto] {$f_1$} (m-1-3)
(m-1-3) edge node[auto] {$f_2$} (m-1-4)
(m-1-4) edge node[auto] {$f_{j}$} (m-1-5)
(m-1-6) edge node[auto] {$f_{j+1}$} (m-1-7)
(m-1-7) edge node[auto] {$f_{i+j}$} (m-1-8)
;
  \path[->,font=\scriptsize] 
(m-1-2) edge node[above,swap] {$f_0$} node[below,sloped] {$\sim$} (m-1-1)
(m-1-6) edge node[above]  {$\id$} node[below,sloped] {$\sim$} (m-1-5)
;
\end{tikzpicture}
\]
We abuse notation and denote by $\jmath \maps \W^{i+j}\W^{-1}(X,Y) \to \W^{i}\W^{-1}\W^{j}\W^{-1}(X,Y)$
the restriction of \eqref{eq:simploc_incl} to the subcategory $\W^{i+j}\W^{-1}(X,Y)$.
We recall \cite[Sec.\ 6.1]{Dwyer-Kan:1980} that the pair $(\C,\W)$ admits a \textbf{homotopy calculus of right fractions} iff the induced maps on nerves
\[
\begin{split}
N \C^{i+j}\W^{-1}(X,Y) \xto{N \jmath} N\C^{i}\W^{-1}\C^{j}\W^{-1}(X,Y), \\
N \W^{i+j}\W^{-1}(X,Y) \xto{N \jmath} N\W^{i}\W^{-1}\C^{j}\W^{-1}(X,Y)
\end{split}
\]
are weak homotopy equivalences for all $i,j \geq 0$ and objects $X,Y \in \C$.

\begin{proof}[Proof of Thm.\ \ref{thm:simploc}]
We show $(\C,\W)$ admits a homotopy calculus of right fractions by adopting the strategy 
used by Nikolaus, Schreiber, and Stevenson to prove their Thm.\ 3.61 in \cite{NSS:2015}. 
First, we observe that the proof of Prop.\ \ref{prop:simploc1} can be easily generalized 
to show that the inclusions of subcategories
\begin{equation} \label{eq:simploc_pf1}
\begin{split}
\C^{i} \Wf^{-1} \stackrel{\iota}{\emb} \C^{i} \W^{-1}, \quad \W^{i} \Wf^{-1} \stackrel{\iota}{\emb} \W^{i} \W^{-1} \\
\end{split}
\end{equation}
and
\begin{equation} \label{eq:simploc_pf2}
\begin{split}
\C^{i} \Wf^{-1} \C^{j} \Wf^{-1} \stackrel{\iota}{\emb} \C^{i} \W^{-1} \C^{j} \W^{-1}, \quad  
\W^{i} \Wf^{-1} \W^{j} \Wf^{-1} \stackrel{\iota}{\emb} \W^{i} \W^{-1} \W^{j} \W^{-1}
\end{split}
\end{equation}
induce homotopy equivalences on the corresponding nerves. Next, we consider
the restriction of the functor \eqref{eq:simploc_incl} $\jmath \maps \C^{i+j}\W^{-1}(X,Y) \to \C^{i}\W^{-1}\C^{j}\W^{-1}(X,Y)$  to the following subcategories:
\begin{equation} \label{eq:simploc_pf3a}
\C^{i+j}\Wf^{-1}(X,Y) \xto{\jmath} \C^{i}\Wf^{-1}\C^{j}\Wf^{-1}(X,Y), 
\end{equation}
and
\begin{equation}\label{eq:simploc_pf3b}
\W^{i+j}\Wf^{-1}(X,Y) \xto{\jmath} \W^{i}\Wf^{-1}\W^{j}\Wf^{-1}(X,Y).
\end{equation}
Let $F \maps \C^{i}\Wf^{-1}\C^{j}\Wf^{-1}(X,Y) \to \C^{i+j}\Wf^{-1}(X,Y)$ be the functor that assigns to the diagram
\[
\begin{tikzpicture}[descr/.style={fill=white,inner sep=2.5pt},baseline=(current  bounding  box.center)]
\matrix (m) [matrix of math nodes, row sep=2em,column sep=2em,
  ampersand replacement=\&]
  {  
X \& C_1 \& C_2 \& ~ \cdots ~   C_j \& C_{j+1} \& C_{j+2} \& ~\cdots ~ C_{i+j+1} \& Y\\
}; 
\path[->,font=\scriptsize] 
(m-1-2) edge node[auto] {$f_1$} (m-1-3)
(m-1-3) edge node[auto] {$f_2$} (m-1-4)
(m-1-4) edge node[auto] {$f_{j}$} (m-1-5)
(m-1-6) edge node[auto] {$f_{j+2}$} (m-1-7)
(m-1-7) edge node[auto] {$f_{i+j+1}$} (m-1-8)
;
  \path[->>,font=\scriptsize] 
(m-1-2) edge node[above] {$f_0$} node[below,sloped] {$\sim$} (m-1-1)
(m-1-6) edge node[above]  {$f_{j+1}$} node[below,sloped] {$\sim$} (m-1-5)
;
\end{tikzpicture}
\]
a diagram
 \[
\begin{tikzpicture}[descr/.style={fill=white,inner sep=2.5pt},
every node/.style={outer sep=2.5pt},
baseline=(current  bounding  box.center)]
\matrix (m) [matrix of math nodes, row sep=2em,column sep=2em,
  ampersand replacement=\&,text height=1.5ex, text depth=0.25ex]
  {  
X \& D_1 \& D_2 \& \cdots  \&  D_{i+j} \& Y\\
}; 
\path[->,font=\scriptsize] 
(m-1-2) edge node[auto] {$g_1$} (m-1-3)
(m-1-3) edge node[auto] {$g_2$} (m-1-4)
(m-1-4) edge node[auto] {$$} (m-1-5)
(m-1-5) edge node[auto] {$g_{i+j}$} (m-1-6)
;

  \path[->>,font=\scriptsize] 
(m-1-2) edge node[above] {$g_0$} node[below,sloped] {$\sim$} (m-1-1)
%(m-1-6) edge node[above]  {$f_{j+1}$} node[below,sloped] {$\sim$} (m-1-5)
;
\end{tikzpicture}
\]
obtained by taking the iterated pullback of the acyclic fibration $f_{j+1}$:

\begin{equation}\label{diag:simploc1}
\begin{tikzpicture}[descr/.style={fill=white,inner sep=2.5pt},
every node/.style={outer sep=2.5pt},
baseline=(current  bounding  box.center)]
\matrix (m) [matrix of math nodes, row sep=2em,column sep=2em,
  ampersand replacement=\&,text height=1.5ex, text depth=0.25ex]
   {  
 \& C'_2 \& C'_3 \&  \cdots \&   C'_{j-1} \& C'_{j} \& C'_{j+1} \& C_{j+2} \& \cdots  \\
X \& C_1 \& C_2 \&  \cdots \&   C_{j-2} \& C_{j-1} \& C_{j}  \& C_{j+1} \&   \\
 }; 

\path[->>,font=\scriptsize] 
(m-1-2) edge node[above,sloped] {$f_0 \circ \ti{f}_1$} (m-2-1)
;

\path[->,font=\scriptsize] 
(m-1-2) edge node[auto] {$f'_{2}$} (m-1-3)
(m-1-3) edge node[auto] {$f'_{3}$} (m-1-4)
(m-1-4) edge node[auto] {$$} (m-1-5)
(m-1-5) edge node[auto] {$f'_{j-1}$} (m-1-6)
(m-1-6) edge node[auto] {$f'_{j}$} (m-1-7)
(m-1-7) edge node[auto] {$f'_{j+1}$} (m-1-8)
(m-1-8) edge node[auto] {$f_{j+2}$} (m-1-9)
% (m-1-9) edge node[auto] {$$} (m-1-10)
% (m-1-10) edge node[auto] {$f_{i+j+1}$} (m-1-11)

(m-2-2) edge node[auto] {$f_1$} (m-2-3)
(m-2-3) edge node[auto] {$f_2$} (m-2-4)
(m-2-4) edge node[auto] {$$} (m-2-5)
(m-2-5) edge node[auto] {$f_{j-2}$} (m-2-6)
(m-2-6) edge node[auto] {$f_{j-1}$} (m-2-7)
(m-2-7) edge node[auto] {$f_j$} (m-2-8)
;

\path[->>,font=\scriptsize] 
(m-2-2) edge node[above] {$f_0$}node[below,sloped] {$\sim$} (m-2-1)
(m-1-8) edge node[auto] {$f_{j+1}$} node[below,sloped] {$\sim$} (m-2-8)
(m-1-7) edge node[auto] {$\ti{f}_{j}$} node[below,sloped] {$\sim$} (m-2-7)
(m-1-6) edge node[auto] {$\ti{f}_{j-1}$} node[below,sloped] {$\sim$} (m-2-6)
(m-1-5) edge node[auto] {$\ti{f}_{j-2}$} node[below,sloped] {$\sim$} (m-2-5)
(m-1-3) edge node[auto] {$\ti{f}_{2}$} node[below,sloped] {$\sim$} (m-2-3)
(m-1-2) edge node[auto] {$\ti{f}_{1}$} node[below,sloped] {$\sim$} (m-2-2)
;
\end{tikzpicture}
\end{equation}

Hence, $D_k := C'_{k+1}$ if $k \leq j$, otherwise $D_k :=C_{k+1}$. And $g_0:= f_0 \circ \ti{f}_1$, 
$g_k:= f'_{k+1}$ if $1 \leq k \leq j$, otherwise $g_k := f_{k+1}$.
Note $F$ is indeed a functor, since all the pullbacks in \eqref{diag:simploc1} are functorial.
Furthermore, if all the $f_k$ are morphisms in $\W$, then so are the $g_k$ by the 2 out of 3 axiom. Hence $F$ restricts to a functor $\W^{i}\Wf^{-1}\W^{j}\Wf^{-1}(X,Y) \to \W^{i+j}\Wf^{-1}(X,Y)$

There is a natural transformation $F \circ \jmath \to \id_{\C^{i+j}\Wf^{-1}(X,Y)}$, where $\jmath$ is the functor \eqref{eq:simploc_pf3a}, whose components are all identity morphisms. Indeed, if $f_{j+1}=\id_{C_{j+1}}$ in the diagram \eqref{diag:simploc1}, then $\ti{f}_{k}=\id_{C_{k-1}}$ for all $k \geq 1$. There is also a natural transformation $\jmath \circ F \to \id_{\C^{i}\Wf^{-1}\C^{j}\Wf^{-1}(X,Y)}$ whose components are the vertical maps in the following diagram:
\begin{equation}\label{diag:simploc2}
\begin{tikzpicture}[descr/.style={fill=white,inner sep=2.5pt},
every node/.style={outer sep=2.5pt},
baseline=(current  bounding  box.center)]
\matrix (m) [matrix of math nodes, row sep=2em,column sep=2em,
  ampersand replacement=\&,text height=1.5ex, text depth=0.25ex]
   {  
 \& C'_2 \& C'_3 \&  \cdots \&  C'_{j+1} \& C_{j+2} \& C_{j+2} \&C_{j+3} \& \cdots \\
X \& C_1 \& C_2 \&  \cdots \&   C_{j}  \& C_{j+1} \& C_{j+2} \& C_{j+3} \& \cdots  \\
 }; 

\path[->,font=\scriptsize] 
(m-1-2) edge node[auto] {$f'_{2}$} (m-1-3)
(m-1-3) edge node[auto] {$f'_{3}$} (m-1-4)
(m-1-4) edge node[auto] {$$} (m-1-5)
(m-1-5) edge node[auto] {$f'_{j+1}$} (m-1-6)
(m-1-7) edge node[auto] {$f_{j+2}$} (m-1-8)
(m-1-8) edge node[auto] {$$} (m-1-9)
%(m-1-9) edge node[auto] {$$} (m-1-10)
%(m-1-10) edge node[auto] {$f_{i+j+1}$} (m-1-11)

(m-2-2) edge node[auto] {$f_{1}$} (m-2-3)
(m-2-3) edge node[auto] {$f_{2}$} (m-2-4)
(m-2-4) edge node[auto] {$$} (m-2-5)
(m-2-5) edge node[auto] {$f_{j}$} (m-2-6)
(m-2-7) edge node[auto] {$f_{j+2}$} (m-2-8)
(m-2-8) edge node[auto] {$$} (m-2-9)
%(m-2-9) edge node[auto] {$$} (m-2-10)
%(m-2-10) edge node[below] {$f_{i+j+1}$} (m-1-11)
;

 \path[->>,font=\scriptsize] 
 (m-1-2) edge node[above,sloped] {$f_0 \circ \ti{f}_1  $} (m-2-1)
 (m-1-7) edge node[above,sloped] {$\id$} node[below] {$\sim$} (m-1-6)
 (m-2-7) edge node[above] {$f_{j+1}$} node[below] {$\sim$} (m-2-6)
 (m-2-2) edge node[above] {$f_0$} node[below] {$\sim$}(m-2-1)
 ;

 \path[->>,font=\scriptsize] 
(m-1-2) edge node[auto] {$\ti{f}_{1}$} node[below,sloped] {$\sim$} (m-2-2) 
(m-1-3) edge node[auto] {$\ti{f}_{2}$} node[below,sloped] {$\sim$} (m-2-3)
(m-1-5) edge node[auto] {$\ti{f}_{j}$} node[below,sloped] {$\sim$} (m-2-5)
(m-1-6) edge node[auto] {$f_{j+1}$} node[below,sloped] {$\sim$} (m-2-6)
;

% (m-2-2) edge node[above] {$f_0$}node[below,sloped] {$\sim$} (m-2-1)
%  (m-1-8) edge node[auto] {$f_{j+1}$} node[below,sloped] {$\sim$} (m-2-8)

%  (m-1-6) edge node[auto] {$\ti{f}_{j-1}$} node[below,sloped] {$\sim$} (m-2-6)
%  (m-1-5) edge node[auto] {$\ti{f}_{j-2}$} node[below,sloped] {$\sim$} (m-2-5)

 \path[->,font=\scriptsize] 
(m-1-7) edge node[auto] {$\id$} (m-2-7) 
(m-1-8) edge node[auto] {$\id$} (m-2-8) 
%(m-1-10) edge node[auto] {$\id$} (m-2-10) 
 ;
\end{tikzpicture}
\end{equation}

The existence of these natural transformations implies that the functors 
\eqref{eq:simploc_pf3a} and \eqref{eq:simploc_pf3b} induce homotopy equivalences on the corresponding nerves. Combining these with functors \eqref{eq:simploc_pf1}, \eqref{eq:simploc_pf2} and then taking the nerve, gives us commutative diagrams of simplicial sets.
\[
\begin{tikzpicture}[descr/.style={fill=white,inner sep=2.5pt},
every node/.style={outer sep=2.5pt},
baseline=(current  bounding  box.center)]
\matrix (m) [matrix of math nodes, row sep=2em,column sep=2em,
  ampersand replacement=\&,text height=1.5ex, text depth=0.25ex]
   {  
N\C^{i+j} \Wf^{-1} \& N\C^{i} \Wf^{-1} \C^{j} \Wf^{-1}\\
N\C^{i+j} \W^{-1} \& N\C^{i} \W^{-1} \C^{j} \W^{-1}\\
 }; 

  \path[->,font=\scriptsize] 
 (m-1-1) edge node[auto] {$N\jmath$} node[below] {$\sim$}(m-1-2) 
 (m-1-1) edge node[auto]  {$N \iota $} node[below,sloped] {$\sim$} (m-2-1) 
 (m-1-2) edge node[auto] {$N \iota $} node[below,sloped] {$\sim$} (m-2-2) 
 (m-2-1) edge node[auto] {$N \jmath$} (m-2-2) 
  ;
\end{tikzpicture}
\qquad
\begin{tikzpicture}[descr/.style={fill=white,inner sep=2.5pt},
every node/.style={outer sep=2.5pt},
baseline=(current  bounding  box.center)]
\matrix (m) [matrix of math nodes, row sep=2em,column sep=2em,
  ampersand replacement=\&,text height=1.5ex, text depth=0.25ex]
   {  
N\W^{i+j} \Wf^{-1} \& N\W^{i} \Wf^{-1} \W^{j} \Wf^{-1}\\
N\W^{i+j} \W^{-1} \& N\W^{i} \W^{-1} \W^{j} \W^{-1}\\
 }; 

  \path[->,font=\scriptsize] 
 (m-1-1) edge node[auto] {$N\jmath$} node[below] {$\sim$}(m-1-2) 
 (m-1-1) edge node[auto]  {$N \iota $} node[below,sloped] {$\sim$} (m-2-1) 
 (m-1-2) edge node[auto] {$N \iota $} node[below,sloped] {$\sim$} (m-2-2) 
 (m-2-1) edge node[auto] {$N \jmath$} (m-2-2) 
  ;
\end{tikzpicture}
\] 
By 2 out of 3, the bottom horizontal morphisms in these diagrams are weak equivalences of simplicial sets. Therefore, $(\C,\W)$ admits a homotopy calculus of right fractions.
\end{proof}

\section{Higher groupoids and Kan fibrations} \label{sec:higher_gpd}
In this section, we recall Henriques' definition \cite{Henriques:2008} of a higher groupoid
object in a (large) category $\Cat$ equipped with a pretopology $\covers$.

\subsection{Preliminaries and notation}
If $\Cat$ is a small category, we denote by $\PSh(\Cat)$ the category
of  presheaves on $\Cat$.
The functor $\yon \maps \Cat \to \PSh(\Cat)$, $X \mapsto \yon X = \hom_{\Cat}(-,X)$
denotes the  Yoneda embedding, which identifies objects in $\Cat$ with the
representable presheaves. 

We denote by $s\Cat$ the category of simplicial objects in $\Cat$,
i.e.\ the category of contravariant functors $\Delta \to \Cat$, where
$\Delta$ is the category of finite ordinals
\[
[0]=\{0\}, \qquad [1]=\{0, 1\},\quad \dotsc,\quad
[n]=\{0, 1,\dotsc, n\},\quad\dotsc,
\]
with order-preserving maps. In particular, $\sSet$ is the category of
simplicial sets. For $m \geq 0$,
the simplicial sets $\Delta^m$ and $\partial \Delta^m$
are the simplicial $m$-simplex and its boundary,
respectively:
\begin{align*}
  (\Simp{m})_n &= \{f\colon (0,1,\dotsc,n) \to (0,1,\dotsc, m)
  \mid f(i)\le f(j) \text{ for all }i \le j\},\\
  (\partial\Simp{m})_n &=
  \bigl\{f\in (\Simp{m})_n \bigm| \{0,\dotsc,m\}
  \nsubseteq \{f(0),\dotsc, f(n)\} \bigr\}.
\end{align*}
For $m>0$ and $0 \leq j \leq m$, 
the horn \(\Horn{m}{j}\) is the simplicial set obtained from the 
\(m\)\nb-simplex~\(\Simp{m}\) by taking away its interior and
its \(j\)th face:
\[
  (\Horn{m}{j})_n =
  \bigl\{f\in (\Simp{m})_n\bigm| \{0,\dotsc,j-1,j+1,\dotsc,m\}
  \nsubseteq \{f(0),\dotsc, f(n)\} \bigr\}.
\]

\subsection{Pretopologies}
An $n$-groupoid in $\Cat$ is special kind of simplicial object in
$\Cat$. The precise definition requires us to equip $\Cat$ with extra
structure, which also allows us
to define sheaves on $\Cat$.
\begin{definition}
  \label{def:pretopology}
  Let~\(\Cat\) be a category with coproducts and a terminal
  object $\ast$.  A
  {\bf pretopology}  on $\Cat$ is a collection $\covers$ of
  arrows, called \textbf{covers}, with the following properties:
  \begin{enumerate}
  \item isomorphisms are covers;
  \item the composite of two covers is a cover;
  \item pullbacks of covers are covers; more precisely, for a
    cover \(U\to X\) and an arrow \(Y\to X\), the pull-back
    \(Y\times_X U\) exists in $\Cat$ and the canonical map
    \(Y\times_X U \to Y\) is a cover;

  \item for any object \(X\in \Cat\), the map \(X\to *\) is a cover.
  \end{enumerate} 
\end{definition}
What we call a pretopology is called
a ``singleton Grothendieck pretopology'' in \cite{Zhu:2009a}, and was
first defined in \cite[Def.\ 2.1]{Henriques:2008}. Every pretopology in our sense gives a
Grothendieck pretopology in the classical sense. 

Let $(\Cat,\covers)$ be a category equipped with a pretopology. A
presheaf $F \in \PSh(\Cat)$ is a sheaf if and only if for every cover
$U \to X$, $F(X)$ is the equalizer of the diagram
\[
F(U) \rightrightarrows F(U \times_X U).
\]
We denote by $\Sh(\Cat) \subseteq \PSh(\Cat)$ the full subcategory of
sheaves on $(\Cat,\covers)$. A pretopology is \textbf{subcanonical}
iff every representable presheaf is a sheaf. In this paper, all
pretopologies are assumed to be subcanonical.

Also, we will never assume
$\Cat$ has limits. Therefore, we take limits of
diagrams in $\Cat$ by first showing that the limit of the
corresponding diagram in $\Sh(\Cat)$ of
representable presheaves is representable, and then using the fact that the
functor $\yon$ preserves limits.

\subsection{Pretopologies for Banach
  manifolds} \label{subsubsec:Banach} We denote by $\Mfd$ the category
whose objects are Banach manifolds, in the sense of \cite[Ch
2.1]{Lang:95}, and whose morphisms are smooth maps. (We could also
consider $C^r$ maps as well.) A morphism $f \maps X\to Y$ between
manifolds is a {\bf submersion} iff for all $x \in X$, there exists an
open neighborhood $U_x$ of $x$, an open neighborhood $V_{f(x)}$ of $f(x)$, and a local
section $\sigma \maps V_{f(x)}\to U_x$ of $f$ at $x$. That is, $\sigma$ is
a morphism in $\Mfd$ such that $f \circ \sigma =\id$ and
$\sigma(f(x))= x$. Note that we may always take $U_x$ to be the
connected component of of $f^{-1} (V_{f(x)})$ containing $x$.

It is a result of Henriques (\cite[Cor.\ 4.4]{Henriques:2008}), 
that the collection $\covers_{\subm}$ of surjective submersions is a
subcanonical pretopology for the category $\Mfd$. 
\begin{remark} \label{rmk:open} Another example of a subcanonical
  pretopology on $\Mfd$ is the pretopology of open covers
  $\covers_{\open}$. Since every cover in the surjective submersion
  pretopology can be refined by a cover in the pretopology of open
  covers (see Example \ref{ex:banach_small}), every sheaf on
  $(\Mfd,\covers_\open)$ is also a sheaf on $(\Mfd,\covers_{ss})$ and
  vice versa. See also \cite [Prop.\ 2.17]{NS:2011}.
\end{remark}

Examples of subcanonical pretopologies for other categories of
geometric interest can be found in Table 1 of \cite{Zhu:2009a}.

\subsubsection{Technicalities involving large categories} \label{subsec:set_theory}
Strictly speaking, the categories $\Sh(\Cat)$ and  $\PSh(\Cat)$ are not well-defined if $\Cat$ 
is not small. In our main example of interest, $\Cat$ will be the large category
of Banach manifolds, so we will need a good theory of sheaves over a large
category.  The set-theoretic technicalities involved with
sheaves over large categories can be subtle.
For example, the sheafification functor may not be well-defined for
presheaves over a large site since a priori it requires taking
colimits over proper classes. The standard workaround is to use
Grothendieck universes, and in particular, to appeal to the Universe
Axiom, that allows one to take colimits in an ambient larger universe in which
classes are sets. However, this larger universe is by no means
canonical, and the resulting colimit may very well depend on the
choice of larger universe. See \cite{Waterhouse:1975} for such an example 
involving sheaves for the fpqc topology in algebraic geometry.

We show in Appendix \ref{sec:universe} that all results in this paper
are independent of choice of universe, provided our pretopology
$\covers$ on the large category $\Cat$ admits what we call a ``small
refinement'' (Def.\ \ref{def:small}).  Indeed, an example of such a pretopology is the
surjective submersion pretopology on the category of Banach
manifolds. (See Example \ref{ex:banach_small}.) 

From here on, we always assume that the pretopology being considered admits a small
refinement. This allows the reader to ignore all set-theoretic
issues, and treat $(\Cat,\covers)$ as if it was a pretopology on a
small category.

\subsection{Kan fibrations and higher groupoids in $(\Cat, \covers)$}
Here we recall Henriques' definition of Kan fibration and
$n$-groupoid, which is
based on the work of Duskin \cite{Duskin:1977} and Glenn \cite{Glenn:1982}.
Let $(\Cat,\covers)$ be a
category equipped with a pretopology.
In what follows, if $K$ is a simplicial set and $X$ is a simplicial
object in $\Cat$, then we denote by  $\Hom(K,X)$ the sheaf
\begin{equation} \label{eq:matchobj_andre0}
 \Hom(K,X)(U): = \hom_{s\Cat}(K \tensor U, X)
 \end{equation}
 where $K \tensor U$ is the simplicial object $(K \tensor U)_n :=
 \coprod_{K_n} U$.  (See Prop.\ \ref{prop:matchobj} for other
 characterizations of this sheaf.)  Note that $\Hom(K,X)$ may not be
 representable. 

More generally, suppose $i\colon A\to B$ is a map
 between simplicial sets and $f\colon X\to Y$ is a morphism of simplicial
 objects in $\Cat$.  Denote by
\begin{equation}
  \label{eq:commuting_square_space}
  \Hom(A \xto{i} B, X\xto{f} Y) \defeq
  \Hom(A, X) \times_{\Hom(A, Y)} \Hom(B, Y);
\end{equation}
the sheaf which assigns to an object $U \in \Cat$ the 
set of commuting squares in $s\Cat$ of the form
 \[
\begin{tikzpicture}[descr/.style={fill=white,inner sep=2.5pt},baseline=(current  bounding  box.center)]
\matrix (m) [matrix of math nodes, row sep=2em,column sep=2em,
  ampersand replacement=\&]
{
     A \tensor U  \& X\\
     B \tensor U \& Y\\
   };
  \path[->,font=\scriptsize] 
   (m-1-1) edge node[auto] {$$} (m-1-2)
   (m-1-1) edge node[auto,swap] {$i$} (m-2-1)
   (m-1-2) edge node[auto] {$f$} (m-2-2)
   (m-2-1) edge node[auto] {$$} (m-2-2)
  ;
 \end{tikzpicture}
 \]
There is a canonical map 
\[
\Hom(B,X) \xto{(i^\ast, f_\ast)} \Hom(A\to B, X\to Y)
\]
induced by pre and post composition with $i \maps A \to B$ and 
$f \maps X \to Y$, respectively.

\begin{definition}[Def.\ 2.3 \cite{Henriques:2008}]
  \label{def:Kan_arrow}
 A morphism $f\colon X\to Y$ of simplicial objects in a category equipped
 with a pretopology $(\Cat, \covers)$ satisfies the
  \textbf{Kan condition} $\Kan(m,j)$ iff the sheaf
  $\Hom(\Horn{m}{j}\to\Simp{m}, X\to Y)$ is representable and the
  canonical map (i.e., the horn projection)
  \begin{equation}
    \label{eq:Kan_arrow}
    X_m = \Hom(\Simp{m},X)\xrightarrow{(\iota^\ast_{m,j},f_\ast)} 
\Hom(\Horn{m}{j}\xto{\iota_{m,j}} \Simp{m}, X\xto{f} Y)
  \end{equation}
  is a cover. The morphism $f \maps X  \to Y$  satisfies the
  \textbf{unique Kan condition} $\Kan!(m, j)$ iff the canonical map
  in \eqref{eq:Kan_arrow} is an isomorphism.
  We say $f \maps X \to Y$ is a \textbf{Kan fibration} iff it satisfies $\Kan(m,j)$
  for all $m\ge 1$, $0\le j\le m$.
\end{definition}

\begin{definition}[Def.\ 2.3 \cite{Henriques:2008}]
  \label{def:ngpd}
A simplicial object $X \in s\Cat$ is
a \textbf{higher groupoid} in $(\Cat,\covers)$, or more precisely,
  an \textbf{$n$-groupoid} object in
  $(\Cat, \covers)$ for $n \in\N \cup \{\infty \}$ iff the unique morphism
\[
X \to \ast
\]
satisfies the Kan condition $\Kan(m,j)$ for $1\leq m \leq n$, $0 \leq
j \leq m$, and the unique Kan condition $\Kan!(m,j)$ for all $m >n$,
$0 \leq j \leq m$. An \textbf{$n$-group object} in $(\Cat, \covers)$
is an $n$-groupoid object $X$, such that $X_0=\ast$, 
where $\ast$ is the terminal object in $\Cat$.
\end{definition}

In other words, $X$ is an
$n$-groupoid if the sheaf $\Hom(\Horn{m}{j}, X)$ is representable and
the restriction map
\[
  \Hom(\Simp{m},X)\to\Hom(\Horn{m}{j},X)
\]
is a cover for all $1 \leq m \leq n$, $0 \leq j
\leq m$ and an isomorphism for all $m >n$, $0 \leq j
\leq m$. 

\subsection{Representability results}
Now we record some useful tools for proving
representability. Similar results can be found in the work of Behrend
and Getzler \cite{Behrend-Getzler:2015}, Wolfson \cite{Wolfson:2016}, and Zhu \cite{Zhu:2009a}.
What follows is reminiscent of the use of anodyne extensions in
simplicial sets.   

\begin{definition}
  \label{def:collapsible}
  The inclusion $\iota \maps  S \hookrightarrow T$ of a simplicial
  subset $S$ into a finitely generated simplicial set $T$ is a
  \textbf{collapsible extension} iff it is the composition of inclusions
  of simplicial subsets
  \[
  S= S_0 \hookrightarrow S_1 \hookrightarrow \dotsb \hookrightarrow S_l=T
  \]
  where each $S_i$ is obtained from $S_{i-1}$ by filling a horn. That
  is, for each $i=1,\ldots,l$, there is a horn $\Horn{m}{j}$ and a map
  $\Horn{m}{j}\to S_{i-1}$ such that $S_i = S_{i-1}
  \sqcup_{\Horn{m}{j}} \Simp{m}$.  If the inclusion of a point into a
  finitely-generated simplicial set $T$ is a collapsible extension,
  then we say $T$ is \textbf{collapsible}.  Similarly, we say $\iota
  \maps S \hookrightarrow T$ is a \textbf{boundary extension} iff it
  is the composition of inclusions of simplicial subsets
  \[
  S= S_0 \hookrightarrow S_1 \hookrightarrow \dotsb \hookrightarrow S_l=T
  \]
  where each $S_i$ is obtained from $S_{i-1}$ by filling a
  boundary. That is, for each $i=1,\ldots,l$, there is a $m \geq 0$
  and a map $\partial \Delta^m \to S_{i-1}$ such that $S_i =
  S_{i-1} \sqcup_{\partial \Delta^m} \Simp{m}$.
  \end{definition}
Collapsible extensions are called ``expansions'' in \cite{Wolfson:2016}.
We follow the terminology that appears in \cite[Sec.\ 2.6]{Li:2015}. 
(A more detailed study of these morphisms can be found there.) 
Other examples of collapsible extensions are the 
``$m$-expansions'' in \cite[Def.\ 3.7]{Behrend-Getzler:2015}.

Clearly a collapsible extension is a boundary extension since the natural
inclusion $\Horn{n}{j} \to \Simp{n}$ can be decomposed into two
boundary extensions $\Horn{n}{j} \to \partial \Simp{n} \to
\Simp{n}$. Also,  if $S\hookrightarrow T$ and $T\hookrightarrow U$ are
both collapsible extensions (or boundary extensions), then so is their composition
$S\hookrightarrow U$. 

We have the following fact:
\begin{lemma}[Lemma 2.44, \cite{Li:2015}]
  \label{lem:example_collapsible}
  The inclusion of any face \(\Simp{k} \to \Simp{n}\) is a collapsible
  extension for \(0\le k<n\).
\end{lemma}

We will use the following two lemmas to solve most of the
representability issues in this paper.  They are similar to Lemma 2.4
in \cite{Henriques:2008}.

\begin{lemma} 
  \label{lem:covers_in_groupoid}
  Let~\(S\hookrightarrow T\) be a collapsible extension and let~\(X\)
  be a higher groupoid in
  \((\Cat,\covers)\).  If \(\Hom(S,X)\) is representable,  then
  \(\Hom(T,X)\) is representable as well and the restriction map
  \(\Hom(T,X) \to \Hom(S,X)\) is a cover.
\end{lemma}
\begin{proof}
  Let \(S=S_0\subset S_1\subset \dotsb \subset S_l=T\) be a
  filtration as in Definition~\ref{def:collapsible}.  Since
  composites of covers are again covers, we may assume without loss
  of generality that \(l=1\), i.e., \(T=S\cup_{\Horn{n}{j}} \Simp{n}\) for
  some \(n,j\).  Note that the functor $\Hom(-,X)$ \eqref{eq:matchobj_andre0}
  sends colimits of simplicial sets to limits of sheaves. Hence, we
  have
\[
  \Hom(T,X)
  = \Hom(S,X) \times_{\Hom(\Horn{n}{j},X)} \Hom(\Simp{n},X).
\]
Since~\(X\) is a higher groupoid, $\Hom(\Simp{n},X) \to
\Hom(\Horn{n}{j},X)$ is a cover between representable sheaves, and
hence, the axioms of a pretopology imply that
 $\Hom(T,X)$ is representable and that
$\Hom(T,X)\to\Hom(S,X)$ is a cover.
\end{proof}
\begin{remark} \label{rem:n-gpd}
In particular, if $X$ is a $k$-groupoid for $k < \infty$, and
$T=S\cup_{\Horn{n}{j}} \Simp{n}$ with $n >k$, and $\Hom(S,X)$ is
representable, then the proof of Lemma \ref{lem:covers_in_groupoid}
implies that $\Hom(T,X) \to \Hom(S,X)$ is not just a cover, but an isomorphism.
\end{remark}

The next lemma concerns the representability of the sheaf \eqref{eq:commuting_square_space}.
\begin{lemma}
  \label{lem:representable}
  Let $S$ be a collapsible simplicial subset of $\Simp{k}$, $X$ a
  simplicial object in $\Cat$, and $Y$ a higher groupoid in $\Cat$.
  If $f\colon X\to Y$ is a morphism 
 which satisfies $\Kan(m,j)$ for all $m<k$ and $0\le j\le m$, 
then the sheaf \(\Hom(S \hookrightarrow \Simp{k}, X \xto{f} Y)\) is representable.
\end{lemma}

\begin{proof}
  First note that the statement is identical to that of \cite[Lemma
  2.4]{Henriques:2008} except that we do not require
  $X_0=Y_0=*$. So, we can proceed exactly as in the proof of
  \cite[Lemma 2.4]{Henriques:2008}, but with a different
  verification of the base case for the induction.  For this, we
  consider $S_0=\ast \hookrightarrow \Delta^k$ and observe that the sheaf
  $\Hom(*\hookrightarrow \Simp{k}, X \xto{f} Y )$ is represented by the pullback
  $X_0\times_{Y_0}Y_k$ in $\Cat$, which exists in $\Cat$ since Lemmas
  \ref{lem:example_collapsible} and \ref{lem:covers_in_groupoid}.
  imply that $Y_k\to Y_0$ is a cover.
  \end{proof}

  \begin{remark} \label{rem:rep-kan-map} 
The horn $\Horn{n}{j} \subset\Delta^{n}$ is collapsible. Hence, if $Y$ is a
higher groupoid and $f \maps X \to Y$ is a
    simplicial morphism satisfying the Kan condition
    $\Kan(m,j)$ for $1 \leq m< n$, then Lemma \ref{lem:representable}
    implies that $\Hom(\Horn{n}{j}\to\Simp{n} , X\to Y)$ is
    automatically representable. 
    Similarly, if $X$ is a simplicial object and if $X \to \ast$ satisfies the Kan condition
    $\Kan(m,j)$ for $1 \leq m< n$, then Lemma \ref{lem:representable}
    implies that $\Hom(\Horn{n}{j}, X)$ is  automatically representable. 
\end{remark}

\section{Points for categories with pretopologies} \label{sec:points}
We begin this section by considering certain functors called ``points'' for
a category equipped with a pretopology. This will allow us to make comparisons
between the homotopy theory of higher groupoid objects and that of
simplicial sets. A point can be thought of
as a generalization of the functor which sends a sheaf on a space to its stalk
at a particular point. The notion originates in 
topos theory. See, for example, \cite[C.2.2, p.\ 555]{Johnstone:2002} and \cite[VII.5]{MM:1994}.

We will not need all of the topos theory formalism here, so our presentation is
quite abbreviated and self-contained. Our
motivation stems from the use of points in the homotopy theory of
simplicial sheaves over the category of finite-dimensional manifolds
(e.g.\, \cite{Dugger:1999,NSS:2015}).

\begin{definition} \label{def:points}
Let $(\Cat,  \covers)$ be a category equipped with a pretopology.
\begin{enumerate}
\item{ A {\bf point} of $(\Cat,\covers)$ is a functor
\[
\ppt \maps \Sh(\Cat) \to \Set
\]
which preserves finite limits and small colimits.
}

\item{
A collection of points $\pts$ of $(\Cat,\covers)$ is \textbf{jointly
  conservative} iff a morphism $\phi \maps F
\to G$ in $\Sh(\Cat)$ is an isomorphism if and only if for all $\ppt \in \pts$
\[
\ppt(\phi) \maps \ppt(F) \to \ppt(G) 
\]
is an isomorphism of sets.
}
\end{enumerate}
\end{definition}
If $X$ is a simplicial object in $(\Cat,\covers)$ and $\ppt \maps
\Sh(\Cat) \to \Set$ is a point, we denote by $\ppt X$ the simplicial set
%\begin{equation} \label{eq:point_sset}
\[
\ppt X_n:=\ppt (\yon X_n)
\]
%\end{equation}

\subsection{Points for Banach manifolds}
Our main example, the category of Banach manifolds
equipped with the pretopology of surjective submersions, admits
a jointly conservative collection of points. If we were only considering
finite-dimensional manifolds, then these points would be the same as those
used in \cite[Def.\ 3.4.6]{Dugger:1999}.

Let $V \in \Ban$ be a Banach space, and denote by $B_V(r)$ the open ball of
radius $r$ about the origin in $V$. Let $(\Mfd,\covers_{\subm})$
denote the category of Banach manifolds equipped with the surjective submersion pretopology,
and denote by  $\ppt_V \maps \Sh(\Mfd) \to \Set$
the functor
\begin{equation} \label{eq:ban_point}
\ppt_V(F) = \colim_{r \to 0} F( B_{V}(r) ).
\end{equation}

\begin{proposition} \label{prop:points}
\mbox{}
\begin{enumerate}
\item For every Banach space $V$, the functor 
$\ppt_V \maps \Sh(\Mfd) \to \Set$ preserves finite limits and small colimits.

\item The collection of points $\pts_{\Ban} :=\{\ppt_V ~ \vert ~  V \in \Ban \}$ is jointly conservative.
\end{enumerate}
\end{proposition}

\begin{proof} 
Since $\ppt_V$ is a filtered colimit, it commutes with finite
limits and small colimits. 
Now let $\phi \maps F \to G$ be a morphism of sheaves such that for all $V
\in \Ban$
\[
\ppt_V(\phi) \maps \ppt_V F \to  \ppt_V G
\]
is bijection. In particular, injectivity implies that  if $x \in F(B_V(r_x))$ and $y \in
F(B_V(r_y))$ such that $\ppt_V(\phi)(\bar{x}) = \ppt_V(\phi)(\bar{y})$,
then there exists $r \leq r_x$ and $r \leq r_y$ such that $ i_x^*x = i_y^*
y$, where $i_x \maps B_V(r) \to B_V(r_x)$ and $i_y \maps B_V(r) \to
B_V(r_y)$ are the inclusions.

We  show $\phi \maps F \to G$ is injective.
Let $M \in \Mfd$ and $x,y \in F(M)$ such that
$\phi_M(x)=\phi_M(y)$. By the usual arguments, for each $m \in M$
there exists an open neighborhood $U_m$ of $m$, a Banach space $V_m$, a
radius $r_m >0$ and a diffeomorphism
$\psi_m \maps B_{V_m}(r_m) \xto{\cong} U_m$ such that $\psi_m(0)=m$.
Hence, we have a cover
\begin{equation} \label{eq:cover}
\coprod_{m \in M} B_{V_m}(r_m) \xto{(i_m)} M,
\end{equation}
where $i_m \maps B_{V_m}(r_m) \to M$ is the composition of $\psi_m$ with the inclusion.
So for each $m\in M$
\[
\phi_{B_{V_m}(r_m) }(i^*_{m} x )  = \phi_{B_{V_m}(r_m) }(i^*_{m} y )
\]
which implies
\[
p^{\ast}_{V_m}(\phi_{B_{V_m}(r_m) })(\overline{i^*_{m} x} )  = p^{\ast}_{V_m}(\phi_{B_{V_m}(r_m) })(\overline{i^*_{m} y} ).
\]
Therefore, there exists a smaller ball $B_{V_m}(r'_m)
\subseteq B_{V_m}(r_m) $ such that the restrictions of $x$ and $y$  to
$B_{V_m}(r'_m)$ are equal. Since 
\[
\coprod_{m \in M} B_{V_m}(r_m) \xto{(i_m)} M
\]
is a cover, and $F$ is a sheaf, we conclude $x=y$.

Now we show $\phi \maps F \to G$ is surjective. Let $M \in \Mfd$ and $y \in
G(M)$. We use the cover \eqref{eq:cover} as before, so that $i^{\ast}_m y \in
G(B_{V_{m}}(r_m))$. Since $\ppt_{V_{m}}\phi \maps \ppt_{V_{m}}F \to
\ppt_{V_{m}}G$ is onto, there exists $r'_m \leq r_m$ and $x_m \in F(B_{V_{m}}(r_m))$
such that $\phi(x_m)= j^{*}_m y$ where $j_m \maps B_{V_{m}}(r'_m)  \to M$ is
the composition of $i_m$ with the inclusion $B_{V_{m}}(r'_m) \hookrightarrow B_{V_{m}}(r_m)$.

Consider the pullback $W \times_M W
\overset{p_1}{\underset{p_2}{\rightrightarrows}} W$, where $W$ is the cover $\coprod_{m \in
  M} B_{V_{m}}(r'_m)$. We claim $p^{*}_1 \bigl( \{ x_m \} \bigr) =
p^{*}_2 \bigl( \{ x_m \} \bigr)$. Indeed, since $y$ is a global
section over $M$, we have the equalities
\[
\phi(p^{*}_1 \bigl( \{x_m \} \bigr)) = 
p^{*}_1 \phi(\bigl( \{ x_m \} \bigr)) = p^{*}_1( \{j^{*}_my
\})=p^{*}_2( \{ j^{*}_my \})= \phi(p^{*}_2 \bigl( \{ x_m \} \bigr)).
\]
Therefore, since $\phi$ is one to one, we conclude 
$p^{*}_1 \bigl( \{ x_m \} \bigr) = p^{*}_2 \bigl( \{ x_m
\} \bigr)$. And since $F$ is a sheaf, there exists a section $x \in
F(M)$ such that $j^{*}_{m} \phi (x) = j^{*}_m y$, and hence $\phi(x)=y$
.
\end{proof}

\begin{remark}\label{rmk:ss_vs_open}
It follows from Remark \ref{rmk:open} that
Prop.\ \ref{prop:points} also implies that the collection of functors
\eqref{eq:ban_point} is jointly conservative for $(\Mfd,\covers_{\open})$.

\end{remark}

\subsection{Matching objects and stalkwise Kan fibrations} \label{subsec:matching}
We next recall how points of $(\Cat,\covers)$ take ``matching objects'' for
simplicial sheaves to matching objects for simplicial sets. We also
describe the relationship between the sheaf $\Hom(K,X)$, defined in
\eqref{eq:matchobj_andre0}, for a simplicial object $X \in s\Cat$ and 
the corresponding matching object for the representable simplicial
sheaf $\yon X$.

\begin{definition} \label{def:matching}
Given a simplicial set $K \in \sSet$ and a simplicial sheaf $F \in
\sSh(\Cat)$, their \textbf{matching object} $M_K(F)$ is the sheaf
\[
M_{K}(F)(U):= \hom_{\sSet}(K ,F(U) ).
\] 
\end{definition}

If $F \in \sSh(\Cat)$ is a simplicial sheaf and $K \in \sSet$ is a
simplicial set, then we
denote by $F_m^{K_{n}}$ the sheaf 
\[
U \mapsto F_{m}^{K_n}(U):= \prod_{K_n} F_m(U).
\]
Hence, for each $m$, $n$, and element $x \in K_n$ we have the
canonical projection $\pi^{n}_m(x) \maps F_m^{K_n} \to F_m$.
For any such $F$ and $K$, there are two maps of sheaves:
\[
\alpha_F, \alpha_K \maps \prod_{m \geq 0} F_{m}^{K_m} \to \prod_{\theta_{mn}} F_m^{K_n},
\]
where the latter product is taken over all morphisms $\theta_{mn} \maps [m] \to [n]$
in the category $\Delta$.
The maps $\alpha_F$ and $\alpha_K$ are defined in the following way:
If $U \in \Cat$, then
$F_m^{K_n}(U)=\hom_{\Set}(K_n,F_m(U))$. So let $\bar{f}=(f_m) \in
\prod_m F_m^{K_m}(U)$. Then, the projections of $\alpha_F(\bar{f})$
and $\alpha_{K}(\bar{f})$ to the factor $F_m^{K_n}(U)$ labeled by a morphism 
$\theta_{mn} \maps [m] \to [n]$ are $F(\theta_{mn})(f_n)$ and 
$ f_m \circ K(\theta_{mn})$, respectively.

We now record some basic facts about the matching object $M_{K}(F)$.
\begin{proposition} \label{prop:matchobj}
\mbox{}
\begin{enumerate}

\item  For any simplicial sheaf $F$ and 
  morphism of simplicial sets  $\gamma \maps K \to L$, there is a natural morphism of
  sheaves $M_\gamma \maps M_L(F) \to M_{K}(F)$.\\

\item  For any simplicial sheaf $F$ and
  simplicial set $K$, $M_{K}(F)$ is the equalizer of the diagram
 \begin{equation} \label{eq:matchobj_eql1}
 \xymatrix{
 \displaystyle{\prod_{m \geq 0}} F_{m}^{K_m}  \ar^*++{\alpha_F}@<-.7ex>[r]
 \ar@<+.7ex>[r]_*++{\alpha_K} & \displaystyle{\prod_{\theta \maps [m] \to [n]}} F_m^{K_n}.
 }
\end{equation}

Moreover, if $K$ is a finitely generated
  simplicial set of dimension $M$, then $M_K(F)$ is the equalizer of
  the diagram
 \begin{equation} \label{eq:matchobj_eql2}
 \xymatrix{
 \displaystyle{\prod_{0 \leq m \leq M}} F_{m}^{K_m}  \ar^*++{\alpha_F}@<-.7ex>[r]
 \ar@<+.7ex>[r]_*++{\alpha_K} & \displaystyle{\prod_{\substack{\theta
       \maps [m] \to [n] \\ 0 \leq m,n \leq M}  }} F_m^{K_n}.
 }
\end{equation}

\item  If $X$ is a simplicial object in $\Cat$, and $yX$ is the representable
  simplicial sheaf $yX(U)_n = \hom_{\Cat}(U,X_n)$, then there
 is a unique natural isomorphism of sheaves $M_K(yX) \cong \Hom(K,X)$, where
  $\Hom(K,X)$ is the sheaf
 \begin{equation} \label{eq:matchobj_andre}
 \Hom(K,X)(U): = \hom_{s\Cat}(K \tensor U, X)
 \end{equation}
previously introduced in \eqref{eq:matchobj_andre0},
 and $K \tensor U$ is the simplicial object $(K \tensor U)_n := \coprod_{K_n} U$.
\end{enumerate}
\end{proposition}

\begin{proof}
Statement (1) is  obvious, as is the proof that
$M_{K}(F)$ is the equalizer of \eqref{eq:matchobj_eql1}. If $K$ is
finitely generated, then one shows $M_K(F)$ is the equalizer of \eqref{eq:matchobj_eql2}
by using the fact that  every simplex in $K$ can be uniquely
written as a non-degenerate simplex composed with a degeneracy map
(the ``Eilenberg--Zilber Lemma''). Finally,  (3) follows by showing
that $\Hom(K,X)$ is the equalizer of \eqref{eq:matchobj_eql1} when
$F=\yon X$.
\end{proof}

\begin{cor} \label{cor:matching}
  %Let $(\Cat,\covers,\pts)$ be a category equipped
  %with a pretopology and a collection of jointly conservative points.
  Let $X$ be a higher groupoid object in $(\Cat, \covers)$ and $K$ a finitely generated
  simplicial set.  Then for each point $\ppt: \Sh(\Cat) \to \Set$ %$\ppt \in \pts$, 
there is an unique  natural isomorphism of sets
\[
\ppt \Hom(K,X) \cong \hom_{\sSet}(K,\ppt X)
\]
\end{cor}

\begin{proof}
For any simplicial set $K$, Prop.\ \ref{prop:matchobj} implies that
\[
\ppt \Hom(K,X) \cong \ppt M_{K}(\yon X).
\]
If $K$ is finitely generated, then for any $n,m \geq 0$, $\yon X^{K_n}_{m}$ is a
finite product of sheaves, and so Prop.\ \ref{prop:matchobj} implies
that $M_{K}(\yon X)$ is an equalizer of finite limits of sheaves. By
definition, the functor $\ppt$ preserves finite limits. Hence,
\begin{equation} \label{eq:corr_match_2}
 \ppt M_{K}(\yon X) \cong  \mathrm{eq} \Bigl (
\xymatrix{
 \displaystyle{\prod_{0 \leq m \leq M}} (\ppt X_{m})^{K_m}  \ar^*++{\ppt \alpha_{yX}}@<-.7ex>[r]
 \ar@<+.7ex>[r]_*++{\ppt \alpha_K} & \displaystyle{\prod_{\substack{\theta
       \maps [m] \to [n] \\ 0 \leq m,n \leq M}  }} (\ppt X_m)^{K_n}
 }
\Bigr).
\end{equation}
A direct computation shows that the equalizer on the right-hand side of 
\eqref{eq:corr_match_2} is simply $\hom_{\sSet}(K,\ppt X)$.
\end{proof}

A jointly conservative collection of points allows us to compare
Kan fibrations of higher groupoids in $(\Cat,\covers)$ to the usual Kan
fibrations of simplicial sets. We first recall a few different notions
of ``surjection'' for sheaves.

\mbox{}
\begin{definition} \label{def:surj}
Let $(\Cat,\covers)$ be a category equipped with a pretopology. 
\begin{enumerate}

\item A morphism of sheaves $\phi \maps F \to G$ is a \textbf{local surjection} iff
for every object $C \in \Cat$ and every element $y \in G(C)$, there exists a
cover $U \xto{f} C$ such that the element $f^{\ast}y \in G(U)$ is in the image
of $\phi_U \maps F(U) \to G(U)$. 

\item If $\pts$ is a collection of jointly conservative points for
  $(\Cat,\covers)$, then a morphism of sheaves $\phi \maps F \to
G$ is a \textbf{stalkwise surjection} with respect to $\pts$ iff for
all $\ppt \in \pts$ the function
\[
\ppt(\phi) \maps \ppt F \to  \ppt G 
\]
is a surjection.
\end{enumerate}
\end{definition}

\begin{lemma}\label{prop:locsurj_epi} \label{prop:cover-sw-surj}
Let $(\Cat,\covers)$ be a category equipped with a pretopology.

\begin{enumerate}
\item If $\phi \maps F \to G$ is a local surjection of sheaves, then $\phi$ is an
epimorphism of sheaves.

\item If $\phi \maps F \to G$ is an epimorphism of sheaves, then $\phi$ is a
stalkwise surjection with respect to any collection of jointly
conservative points $\pts$ of $(\Cat,\covers)$.

\item Let $f \maps U \to C$ be a cover in $(\Cat,\covers)$. The induced morphism of sheaves
$f_\ast \maps \yon (U) \to \yon(C)$ is a stalkwise surjection with
respect to any collection of jointly conservative points.

\end{enumerate}

\end{lemma}

\begin{proof}
(1) Suppose $\alpha,\beta \maps G \to H$ are morphisms of sheaves such that
$\alpha \circ \phi = \beta \circ \phi$. We wish to show $\alpha=\beta$.
Let $C \in \Cat$ and $ y \in G(C)$. Let $U \xto{f} C$ be a cover such that there
exists $x \in F(U)$ such that $\phi_U(x) = f^{\ast}(y) \in G(U)$. Since
$\alpha_U \circ \phi_U = \beta_U \circ \phi_U$, we have 
\[
\alpha_U(f^{\ast}y) = \beta_U(f^{\ast}y) \in H(U).
\]
Hence, $f^{\ast}\alpha_C(y)=f^{\ast}\beta_C(y)$. Since $H$ is a sheaf, and $U$
is a cover, we conclude $\alpha_C(y)=\beta_C(y)$.

(2) A morphism $\phi \maps F \to G$ is an epimorphism if and only if the diagram
\[
\xymatrix{
F \ar[d]_{\phi} \ar[r]^{\phi} & G \ar[d]^{\id} \\
G \ar[r]^{\id} & G \\
}
\]
is a pushout. By definition, a point $\ppt \maps \Sh(\Cat) \to \Set$  preserves
small colimits. Hence, the proof follows.

(3) We will show $f_\ast$ is a local surjection of sheaves. Then (1)
and (2) will imply the result. Let $A \in \Cat$ and $g \in \yon(
C)(A)=\hom(A,C)$. By axioms of a pretopology, the pullback
\[
\xymatrix{
A \times_C U \ar[d]_{\pr_1} \ar[r]^{\pr_2} & U \ar[d]^{f} \\
A \ar[r]^{g} & C
}
\]
is a cover of $A$. Since $\pr_2 \in \yon (U)(A \times_C U)$, we can apply
$f_\ast$:
\[
f_\ast(\pr_2)=f \circ \pr_2 = g \circ \pr_1 = \pr_1^{\ast}(g) \in \yon(C)(A
\times_C U).
\]
Hence, $f_\ast$ is a local surjection.
\end{proof}

Now we will start connecting the homotopy theory of higher groupoids
in $(\Cat,\covers)$ with that of Kan simplicial sets.
The next proposition says that a Kan fibration of higher groupoids is
a ``stalkwise Kan fibration''  with respect to any collection of jointly conservative points.

\begin{proposition}\label{prop:kan}
If $f \maps X \to Y$ is a Kan fibration of higher groupoids in $(\Cat,
\covers)$ and $\pts$ is a collection of jointly
conservative points of $(\Cat,\covers)$, then for all $\ppt \in \pts$ the
map $\ppt  f \maps \ppt X \to \ppt Y$ is a Kan fibration of simplicial sets.
 \end{proposition}
\begin{proof}
Since $f \maps X \to Y$ is a Kan fibration, the horn projections
 \[ 
\Hom(\Simp{n}, X) \xrightarrow{(\iota^\ast_{n,j},f_\ast)} \Hom(\Horn{n}{j} \to
 \Simp{n}, X\to Y) 
\] 
are covers for all $n\ge 1$ and $0\le j \le
 n$. Hence, Prop.\ \ref{prop:locsurj_epi} implies that
for each point $\ppt \in \pts$
\begin{equation} \label{eq:prop_kan1}
\ppt \Hom(\Simp{n}, X) \xrightarrow{\ppt (\iota^\ast_{n,j},f_\ast)}
\ppt \Hom(\Horn{n}{j} \to \Simp{n}, X\to Y)
 \end{equation}
 is a surjection.  Corollary
 \ref{cor:matching} implies that $\ppt (\Hom(\Simp{n}, X)) \cong
 \hom_{\sSet}(\Simp{n}, \ppt X)$, and since each $\ppt$ preserves finite
 limits, we have
\[
\ppt ( \Hom(\Horn{n}{j} \to \Simp{n},
 X\to Y)) \cong 
  \hom_{\sSet}(\Horn{n}{j}, \ppt X) \times_{\hom_{\sSet}(\Horn{n}{j}, \ppt Y)} \hom_{\sSet}(\Simp{n}, \ppt Y).
\]
Combining these natural isomorphisms with \eqref{eq:prop_kan1} implies
that $\ppt X \xrightarrow{\ppt f} \ppt Y$ is a Kan fibration of
simplicial sets.
\end{proof}

\begin{cor}\label{cor:kan}
If $X$ is a higher groupoid in $(\Cat,
\covers)$ and $\pts$ is a collection of jointly
conservative points for $(\Cat,\covers)$, then the simplicial set $\ppt X$
is a Kan complex for each $\ppt \in \pts$.
\end{cor}

\section{Stalkwise weak equivalences} \label{sec:stalk_weqs} In this
section, we introduce the morphisms between higher groupoids which will
eventually become the weak equivalences in an incomplete category of
fibrant objects. These stalkwise weak equivalences are a natural choice for weak equivalences
between simplicial sheaves in diffeo-geometric contexts, e.g.\
\cite[Def.\ 3.4.6]{Dugger:1999}.
 
\begin{definition} \label{def:stalk_weq}
Let $(\Cat,\covers)$ be a category equipped with a pretopology and a
collection of jointly conservative points $\pts$. A morphism $f \maps X
\to Y$ of higher groupoids in $(\Cat,\covers)$ is a \textbf{stalkwise
  weak equivalence} iff $\ppt f \maps \ppt X \to \ppt Y$ is a weak
homotopy equivalence of simplicial sets for all $\ppt \in \pts$.
\end{definition}

\subsection{Simplicial homotopy groups for higher groups} 
\label{sec:simplicial-homotopy-groups}
Following Joyal \cite{Joyal:1984}, Henriques gave a definition of simplicial homotopy group
sheaves for $n$-groups  in a category
equipped with a pretopology. (See also related work of Jardine \cite{Jardine:1983}.) 
We show here that a morphism $f \maps X
\to Y$ of $n$-groups which induces an isomorphism between the
simplicial homotopy groups is a stalkwise weak equivalence. We will
need this result for integrating $L_\infty$ quasi-isomorphisms in Section \ref{sec:int_exact}.

Let $S^n:=\Delta[n] / \del \Delta[n]$ be the simplicial
$n$-sphere, and let $\cyl(S^n)$ denote the simplicial set
\begin{equation} \label{eq:cyl}
\cyl (S^n) := \Delta[n+1] \bigg/ \Bigl( \bigcup_{i=2}^{n+1} F^i \cup
( F^0 \cap F^1) \Bigr),
\end{equation}
where $F^i$ is the simplicial set generated by the $i$th face of
$\Delta[n+1]$. 
There are two inclusions $i_0,i_1 \maps S^n \to \cyl(S^n)$ induced by the maps
$\Delta[n] \to F^0$ and $\Delta[n] \to F^1$, which are homotopy
equivalences as well as homotopic to one another. 
Let $X$ be a reduced Kan simplicial set, i.e.\ $X_0=\ast$. Since we
have a unique basepoint, for $n \geq 1$ 
we can define the \textbf{$n$th-homotopy group} of
$X$ as the coequalizer
\begin{equation}\label{eq:pi_n-sSet}
\pi_n(X) = \coeq \Bigl ( \hom_{\sSet}(\cyl(S^n), X) \rightrightarrows
\hom_{\sSet}(S^n,X) \Bigr).
\end{equation}
In \cite{Henriques:2008}, Eq.\ \eqref{eq:pi_n-sSet} is used to provide
an analogous definition of simplicial homotopy groups for 
higher groups in $(\Cat,\covers)$.  The suitable analog of 
$\hom_{\sSet}(K, -)$ in \eqref{eq:pi_n-sSet} is the matching object
described in Sec. \ref{subsec:matching}.

\begin{definition}[Def.\ 3.1 \cite{Henriques:2008}] \label{def:andre_pi_n}
Let $X$ be an $k$-group in $(\Cat,\covers)$. Let $n \geq1$ and let $S^n$ be
the simplicial $n$-sphere, and $\cyl(S^n)$ the simplicial set
\eqref{eq:cyl}.
The \textbf{simplicial homotopy groups} $\pi^{\spl}_n(X)$ are the
sheaves 
%\begin{equation} \label{eq:andre_pi_n}
\[
\pi^{\spl}_n(X) := \coeq \Bigl ( \Hom(\cyl(S^n), X) \rightrightarrows
\Hom(S^n,X) \Bigr),
\]
%\end{equation}
where $\Hom(-,X)$ is the sheaf \eqref{eq:matchobj_andre}. 
\end{definition}

\begin{proposition}\label{thm:homotopy_grps}
 % Let $(\Cat,\covers,\pts)$ be a category equipped
 % with a pretopology and a collection of jointly conservative points. 
  Let $X$ be an $n$-group in $(\Cat,\covers)$. Then for all points
  $\ppt: \Sh(\Cat) \to \Set$, there is a unique natural isomorphism of groups
%\begin{equation} \label{eq:homotopy_grps_iso}
\[
\phi_X \maps \ppt \pi^{\spl}_{\ast}(X) \xto{\cong} \pi_{\ast}(\ppt X)
\]
%\end{equation}
\end{proposition}

\begin{proof}
Since $\ppt$ preserves colimits, we have a natural isomorphism
%\begin{equation} \label{eq:homotopy_grps1}
\[
\ppt \pi^{\spl}_{k}(X) \cong \coeq \Bigl ( \ppt \Hom(\cyl(S^k), X) \rightrightarrows
\ppt \Hom(S^k,X) \Bigr).
\]
%\end{equation}
Since $S^n$ and $\cyl(S^n)$ are finitely generated, Cor.\
\ref{cor:matching} implies that there are natural isomorphisms
\begin{align*}
\ppt \Hom(S^n,X) & \cong \hom_{\sSet}(S^n,\ppt X), \\
 \ppt \Hom(\cyl(S^n),X) & \cong  \hom_{\sSet}(\cyl(S^n),\ppt X).
\end{align*}
Combining these natural isomorphisms and using the fact that
$\Hom(-,X)$ is functorial (Prop.\ \ref{prop:matchobj}), we obtain
a natural isomorphism
\[
\phi_X \maps \ppt \pi^{\spl}_{k}(X) \xto{\cong} \coeq \Bigl (
\hom_{\sSet}(\cyl(S^k), \ppt X) \rightrightarrows \hom_{\sSet}(S^k, \ppt X) \Bigr)
=\pi_k(\ppt X).
\]
\end{proof}

\subsection{Stalkwise acyclic fibrations}
Recall that an acyclic fibration of simplicial sets is a morphism of
simplicial sets which is both a weak equivalence and a
fibration. Equivalently, $X \to Y$ is an acyclic fibration of
simplicial sets iff the boundary projections
\begin{equation} \label{eq:acyc_sset}
\hom_{\sSet}(\Simp{n}, X) \to \hom_{\sSet}(\partial \Simp{n}, X)
\times_{\hom_{\sSet}(\partial \Simp{n}, Y)} \hom_{\sSet}(\Simp{n}, Y)
\end{equation}
are surjective for all $n \geq 0$. In Prop.\
\ref{prop:kan} we showed that a Kan fibration $X \to Y$ between higher groupoids
in $(\Cat,\covers)$ equipped with a jointly conservative collection of
points $\pts$ is a stalkwise Kan fibration, i.e.\ induces a Kan
fibration $\ppt X \to \ppt Y$ for all $\ppt \in \pts$. So clearly a Kan fibration which is also a
stalkwise weak equivalence is obviously a ``stalkwise acyclic
fibration''. In fact, more is true, as the following proposition
shows:

\begin{proposition}\label{prop:explicit-acyclic}
 Let $(\Cat, \covers, \pts)$ be a category equipped with a
  pretopology and a collection of jointly conservative points. A morphism $f\maps X\to Y$ of
  higher groupoids in $(\Cat, \covers)$ is both a stalkwise Kan
  fibration and a stalkwise weak equivalence (i.e., a stalkwise acyclic fibration)
if and only if the boundary projections
\[
\Hom(\Simp{n}, X) \xrightarrow{(\jmath^\ast_n,f_\ast)} 
\Hom(\partial\Simp{n}\xto{\jmath_n} \Simp{n},X\xto{f} Y)
\]
are stalkwise surjections for $n \geq 0$.
\end{proposition}
\begin{proof}
Recall that the sheaf $\Hom(\partial\Simp{n}\to \Simp{n},X\to Y)$ 
\eqref{eq:commuting_square_space} is the pullback
\[
\Hom(\partial \Simp{n}, X)
\times_{\Hom(\partial \Simp{n}, Y)} \Hom(\Simp{n}, Y).
\]
Corollary \ref{cor:matching} implies that $\ppt (\Hom(\Simp{n}, X)) \cong
 \hom_{\sSet}(\Simp{n}, \ppt X)$ for all $\ppt \in \pts$, and since each $\ppt$ preserves finite
 limits, we have
\[
\begin{split}
\ppt \bigl( \Hom(\partial \Simp{n}, X) & \times_{\Hom(\partial \Simp{n},
  Y)}  \Hom(\Simp{n}, Y) \bigr)  \\
& \cong \hom_{\sSet}(\partial \Simp{n}, \ppt X) \times_{\hom_{\sSet}(\partial \Simp{n},
  \ppt Y)} \hom_{\sSet}(\Simp{n}, \ppt Y).
\end{split}
\]
Hence, by comparison with \eqref{eq:acyc_sset}, we see that $\ppt f
\maps \ppt X \to \ppt Y$ is an acyclic fibration of simplicial sets if and
only if $\ppt  (\jmath^\ast_n,f_\ast)$ is surjective, i.e.\ $(\jmath^\ast_n,f_\ast)$ is a stalkwise surjection.
\end{proof}

\section{Hypercovers} \label{sec:hypercovers}
Hypercovers were introduced in \cite{AGV:1972} and subsequently have been used
throughout the homotopy theory of simplicial sheaves, e.g.\ \cite{Brown:1973}. 
Hypercovers for Lie $n$-groupoids play an important role in the work
of Zhu \cite{Zhu:2009a} and Wolfson \cite{Wolfson:2016}.
Also, the acyclic fibrations in the Behrend--Getzler
CFO structure for $n$-groupoids objects in a descent category
are hypercovers.

\begin{definition}
  \label{def:equivalence}
Given a category and pretopology \((\Cat, \covers)\), a morphism
  \(f\colon X\to Y\) of simplicial objects in \(\Cat\) is a  \textbf{\hypercover} iff it satisfies
  the condition \(\Acyc(m)\) for all \(0\le m\),  which means the sheaf 
  \(\Hom(\partial\Simp{m}\to \Simp{m},X\to Y)\) is representable and
  the canonical boundary projection: 
  \begin{equation}
    \label{eq:acyclic-fibration}
    \Hom(\Simp{m}, X) \xrightarrow{q_m:=(\jmath^\ast_m,f_\ast)}
    \Hom(\partial\Simp{m}\xto{\jmath_m} \Simp{m},X \xto{f} Y),
  \end{equation}
  is a cover  in \(\covers\).  For \(m=0\),
  \(\Hom(\partial\Simp{0}\to \Simp{0},X\to Y)\defeq Y_0\) by definition.
\end{definition}

\begin{remark}
\mbox{}
\begin{enumerate}\label{rmk:hypercover}
\item   As shown in \cite[Lemma 2.4]{Zhu:2009a}, if $Y$ is a
higher groupoid in \((\Cat, \covers)\),  and $f \colon X \to Y$
satisfies \eqref{eq:acyclic-fibration} for $l<m$, then
$\Hom(\partial\Simp{m}\xto{\jmath_m} \Simp{m},X \xto{f} Y) $ is
    automatically representable.                           %initial condition of an induction to
                                %be held, the collapsible set in the target
                                %need to be representable.
\item A hypercover is a Kan fibration
  automatically, since maps in \(\{ \Horn{m}{j} \to
  \Simp{m}\mid m > 0 , 0\le j\le m \}\) can be reconstructed
  as pushouts of ones in \(\{ \partial \Simp{m} \to
  \Simp{m}\mid m \geq 0 \}\).
\end{enumerate}
\end{remark}

The $\Acyc(m)$ condition is obviously analogous to the previously discussed
characterization \eqref{eq:acyc_sset} of acyclic fibrations of
simplicial sets. Indeed, every hypercover of higher groupoids is a
stalkwise acyclic fibration in the sense of Prop.\
\ref{prop:explicit-acyclic}:

\begin{cor} \label{cor:hyper_sw_acyc}
If a morphism of higher groupoids $f \maps X \to Y$  in
$(\Cat,\covers,\pts)$ is a hypercover, then it is Kan fibration and
a stalkwise weak equivalence.
\end{cor}
\begin{proof}
Remark \ref{rmk:hypercover} implies that $f \maps X \to Y$ is a Kan fibration.
Since the boundary projections \eqref {eq:acyclic-fibration} are covers,
Prop.\ \ref{prop:locsurj_epi} implies that they are stalkwise surjections with respect
to the points $\pts$. Hence, Prop.\ \ref{prop:explicit-acyclic} implies that $f
\maps X \to Y$ is a stalkwise acyclic fibration, so in particular, it
is a stalkwise weak equivalence.
\end{proof}

The natural question to ask is whether the converse of Cor.\
\ref{cor:hyper_sw_acyc} is true. That is, is every morphism of higher groupoids
which is both a Kan fibration and a stalkwise weak equivalence
necessarily a hypercover? The answer will be yes, if our pretopology
satisfies some additional conditions. Our main example,
the category $\Mfd$ of Banach manifolds 
equipped with the surjective submersion pretopology $\covers_{\subm}$
and the collection of points $\pts_{\Ban}$, satisfies these additional conditions.

\subsection{Locally stalkwise pretopologies}

\begin{definition}\label{defi:cover-from-sheaf}
Let $(\Cat, \covers, \pts)$ be a category equipped with a pretopology
and a jointly conservative collection of points. Let $X \in \Cat$.
A morphism $F \xto{g} \yon X$ of sheaves is a \textbf{local
stalkwise cover} iff there exists an object $Y \in \Cat$ and a stalkwise
surjection (Def.\ \ref{def:surj}) $\yon Y \xto{f} F$ such that the
composition $g \circ f \maps \yon Y \to \yon X$ is a cover in $(\Cat,\covers)$.
\end{definition}

\begin{definition} \label{assump:cover} \label{def:LSW_covers}
A pretopology $\covers$ on a category $\Cat$ is a \textbf{locally
  stalkwise pretopology} iff there exists a jointly conservative
collection of points $\pts$ of $(\Cat,\covers)$ such that:
\begin{enumerate}
\item (2-out-of-3) If $U\xrightarrow{f} V \xrightarrow{g} W$ are
  morphisms in $\Cat$, and $g\circ f$ is a cover and $\yon(f)$ is a stalkwise surjection in
  $\Sh(\Cat)$ with respect to $\pts$, then $g$ is also a cover,
\\  
\item (locality of covers) If
  $W\xrightarrow{q} V$ and $U\xrightarrow{p} V$ are morphisms in
  $\Cat$, $\yon(q)$ is a stalkwise surjection with respect to
  $\pts$, and the base change $U\times_{V} W
  \xrightarrow{\tilde{p}} W$ is a local stalkwise cover,
  then $p$ is a cover, and therefore $U\times_{V} W$ is representable.
\end{enumerate}
\end{definition}

We have several remarks about Def.\ \ref{def:LSW_covers}:

\begin{remark} \label{rk:2-out-of-3}
\mbox{}
\begin{enumerate}

\item It is not really precise to call the first condition in Def.\ \ref{def:LSW_covers}:
``2-out-of-3'', because
$g\circ f$ being a cover and $g$ being a cover will
not imply anything (just like surjective maps for sets). 
We note that our ``2-out-of-3'' property holds
automatically for stalkwise surjections in $\Sh(\Cat)$. 

\item If $(M,\covers,\pts)$ is a category equipped with a locally
  stalkwise pretopology then a local stalkwise cover $F \xto{g} \yon
  X$ of $X$ is a cover in $\covers$ whenever $F$ is representable.
\end{enumerate}
\end{remark}

We now can give a converse to Cor.\ \ref{cor:hyper_sw_acyc}. In fact,
we have:

\begin{proposition} \label{prop:hypercover}
Let $(\Cat,\covers)$ be a category equipped with a locally stalkwise
pretopology with respect to a jointly conservative collection of points
$\pts$, and let $f \maps X \to Y$ be a morphism of higher
groupoids in $(\Cat,\covers)$. The following are equivalent:
\begin{enumerate}
\item $f \maps X \to Y$ is a Kan fibration and a stalkwise weak equivalence with
  respect to the collection of points $\pts$.
\item $f \maps X \to Y$ is a Kan fibration and a stalkwise weak equivalence with
  respect to any jointly conservative collection of points of $(\Cat,\covers)$.
\item $f \maps X \to Y$ is a hypercover.
\end{enumerate}
\end{proposition}
\begin{proof}
(3)$\Rightarrow$(2) is the content of Cor.\ \ref{cor:hyper_sw_acyc}, and
(2) $\Rightarrow$(1) is obvious. So we focus on  (1)$\Rightarrow$(3).

First, let $k \geq 1$ and consider the following pullback diagram in $\Sh(\Cat)$
%$\tilde{\Cat}$
\begin{equation}\label{eq:bd-horn}
\xymatrix{
\Hom(\Simp{k}, X) \ar[r]^-{q_k} &\Hom(\partial \Simp{k} \to \Simp{k}, X\to Y) \ar[d]^{pr_0} \ar[r]^{\tilde{q}_{k-1}} &
\Hom(\Horn{k}{0} \to \Simp{k}, X\to Y)  \ar[d]^{d_0^Y \times
  (d_0^X)^{\times k}} \\
&\Hom(\Simp{k-1}, X) \ar[r]^-{q_{k-1}} & \Hom(\partial \Simp{k-1} \to \Simp{k-1},
X\to Y). }
\end{equation} 
The projection $\Hom(\partial \Simp{k} \to \Simp{k}, X\to Y) \to
\Hom(\Simp{k-1}, X)=X_{k-1}$ is induced by $d^X_0$. 
The composition $\tilde{q}_{k-1} \circ q_k: \Hom(\Simp{k}, X) \to \Hom(\Horn{k}{0}
\to \Simp{k}, X\to Y)$ is the horn projection, and therefore a cover,
since $f \maps X \to Y$ is a Kan fibration. Since $f$ is also a
stalkwise weak equivalence, Prop.\ \ref{prop:explicit-acyclic} implies
that $q_k$ is a stalkwise surjection. Hence, $\tilde{q}_{k-1}$ is a
local stalkwise cover (Def.\ \ref{defi:cover-from-sheaf}).
 
Next, we observe that the composition 
$pr_0\circ q_k: \Hom(\Simp{k}, X) \to \Hom(\Simp{k-1}, X)$ is simply
the face map $d_0 \maps X_k \to X_{k-1}$, which is a cover (Lemmas 
\ref{lem:example_collapsible} and   \ref{lem:covers_in_groupoid}),
hence a stalkwise surjection (Lemma \ref{prop:locsurj_epi}). 
Combining this with the fact that $q_k$ is a stalkwise surjection
implies that $pr_0$ is a local stalkwise cover, and hence also a stalkwise surjection.
Thus $d_0^Y\times(d_0^X)^k \circ \tilde{q}_{k-1} =
q_{k-1} \circ pr_0$ is stalkwise surjective, hence $d_0^Y\times(d_0^X)^k$
is stalkwise surjective.

Now, we show via induction that the boundary maps $\Hom(\Simp{k},X)
\xto{q_k} \Hom(\partial \Simp{k} \to \Simp{k}, X\to Y)$ are covers for
all $k \geq 0$.  For the base case, let $k=1$.  Then Def.\
\ref{def:equivalence} implies that $\Hom(\partial \Simp{k-1} \to
\Simp{k-1}, X\to Y)=Y_0$ is representable, and since $f$ is a Kan
fibration, $\Hom(\Horn{1}{0} \to \Simp{1}, X\to Y)$ is also representable.
The results of the previous paragraphs imply that $\tilde{q}_0$ is a local stalkwise cover and
$d_0^Y\times(d_0^X)^k$ is a stalkwise surjection.  Since $\covers$ is
a locally stalkwise pretopology, Def.\ \ref{assump:cover} (2) implies
that $q_{0}$ is a cover and $\Hom(\partial \Simp{1} \to \Simp{1}, X\to
Y)$ is representable.

Finally, suppose $k \geq 2$, $\Hom(\Simp{k-2}, X) \xto{q_{k-2}}
\Hom(\partial \Simp{k-2} \to \Simp{k-2}, X\to Y)$ is a cover, and 
$\Hom(\partial \Simp{k-1} \to \Simp{k-1}, X\to Y)$ is representable.
As previously shown, the pullback of $q_{k-1}$ in diagram
\eqref{eq:bd-horn} along the stalkwise surjection $d_0^Y\times(d_0^X)^k$
is a local stalkwise cover. Hence, Def.\ \ref{assump:cover}(2) implies
that  $q_{k-1} \maps \Hom(\Simp{k-2}, X) \to
\Hom(\partial \Simp{k-1} \to \Simp{k-1}, X\to Y)$ is a cover, and 
that $\Hom(\partial \Simp{k} \to \Simp{k}, X\to Y)$ is
representable. This completes the proof.
\end{proof}

\subsection{Locally stalkwise pretopologies for Banach manifolds}
\label{sec:local_sw_BanMfd}
Here we show that the pretopology $\covers_{\subm}$ of surjective submersions equipped
with the collection of jointly conservative points $\pts_{\Ban}$
\eqref{eq:ban_point} is a locally stalkwise pretopology for the
category of Banach manifolds. We first give an explicit description of
stalkwise surjective maps in $(\Mfd,\covers_{\subm},\pts_{\Ban})$.

\begin{lemma} \label{lem:surj_sub_2of3}
Let $X\xrightarrow{f} Y
\xrightarrow{g} Z$ be composable morphisms in  $\Mfd$. If $f$ is surjective and $g\circ f$ is a
surjective submersion, then $g$ is also a surjective submersion. 
\end{lemma}
\begin{proof}
Clearly, $g$ is surjective. Let $y \in Y$, $z=g(y)$, and $x \in
f^{-1}(y)$. Since $g \circ f$ is a surjective submersion, there exists
open neighborhoods $U$ and $V$ of $z$ and $x$, respectively, and a
section $\sigma_{XZ} \maps U \to V$ such that $\sigma_{XZ}(z)=x$, and
$g \circ f \circ \sigma_{XZ} = \id_{U}$. Let $W=g^{-1}(U)$. Then
$\sigma := f\circ \sigma_{XZ} \maps U \to W$ is a section of $g$ such
that $\sigma(z)=y$. Hence $g$ is a surjective submersion.
\end{proof}

\begin{lemma}\label{lemma:stalkwise-surj}
Let $\phi \maps X \to Y$ be a morphism in $(\Mfd,\covers_{\subm},\pts_{\Ban})$. 
Then the following are equivalent:
\begin{enumerate}
\item \label{item-surj-sub} 
The morphism of sheaves $\yon(\phi) \maps \yon X \to \yon Y$ 
is stalkwise surjective.

\item \label{item-glue}  
For any morphism $f \maps  U \to Y$ in $\Mfd$, there exists an open cover
$\{U_i \}_{i \in \mathcal{I}}$ of $U$ and morphisms $f_i \maps U_i \to X$ such that
$\varphi \circ f_i = f \vert_{U_i}$ for each $i \in \mathcal{I}$.

\item  \label{item-section} 
For each $y \in Y$, there exists: a preimage $x \in \phi^{-1}(y)$, an open
neighborhood of $V$ of $y$, an open
neighborhood of $W$ of $x$, and a morphism $\sigma \maps V \to W$ such
that $\sigma(y)=x$ and $\phi \circ \sigma = \id_{V}$.
\end{enumerate}
\end{lemma}
\begin{proof}

\mbox{}

(\ref{item-surj-sub} $\Rightarrow$ \ref{item-glue}): Let $f \maps U
\to Y$ be a morphism of Banach manifolds.  Since $\phi$ is stalkwise
surjective, for each $z \in U$, we can find a Banach space $V_z$, an
$r_z >0$ and a local diffeomorphism $i_z \maps B_{V_z}(r_z) \to U$
with $i_z(0)=z$ such that there exists a map $f_z \maps B_{V_z}(r_z)
\to X$ with the property that  $\phi \circ f_z=f\circ i_z$. We then use
collection of open balls $\{B_{V_z}(r_z) \}_{z \in U}$ to obtain the
desired open cover of $U$, and the collection
$\{ f_z \}_{z \in U}$ as our desired collection of maps.

(\ref{item-glue}  $\Rightarrow$\ref{item-surj-sub}):
Let $V$ be a Banach space and $\bar{g} \in p_V\yon Y$. This class is
represented by a map $g \maps B_V(r) \to Y$.  By hypothesis, there
exists an open cover $\{U_i\}$ of $B_V(r)$ and maps $h_i \maps U_i \to
X$ such that $\varphi \circ h_i = g \vert_{U_{i}}$. Let $U_{i_{0}}$ be
an open subset containing $0$, and $r'>0$ such that  $B_V(r') \subseteq U_{i_0}$.
Then $g \vert_{B_V(r')}$ is the composition of $\varphi$ with $f=h_{i_0}
\vert_{B_V(r')}$, and hence $p_V\varphi(\bar{f})=\bar{g}$.

(\ref{item-glue} $\Rightarrow$ \ref{item-section}): 
Let $y \in Y$ and consider the identity map $\id_Y \maps Y \to Y$.
There there exists an open cover $\{U_i\}$ of $Y$ and maps $f_i \maps
U_i \to Y$ such that $\varphi \circ f_i = \id_{U_i}$. Let $U_{i'}$
be an element of the cover containing $y$, and let  $x=f_{i'}(y)$. Then 
$\sigma:=f_{i'} \maps U_{i'} \to X$ is a desired local section of
$\varphi$ which maps $y$ to $x$.

(\ref{item-section} $\Rightarrow$ \ref{item-glue}): 
Let $f \maps U \to Y$ be a map. For each $y \in Y$, there exists an $x
\in \phi^{-1}(y)$, an open subset $V_y$ containing $y$, and open
subset $W_y$ containing $x$ and a local section $\sigma_y \maps V_y
\to W_y$ mapping $y$ to $x$. The collection $\{U_y:=f^{-1}(V_y)\}_{y \in
  Y}$ is an open cover of $U$. For each $y \in Y$, let $f_y \maps U_y
\to X$ be the composition $f_y=\sigma_y \circ f \vert_{U_y}$. Then, by
construction, $\varphi \circ f_{y} = f \vert_{U_y}$.
\end{proof}

\begin{remark}
Note that Lemma \ref{lemma:stalkwise-surj} implies that 
stalkwise surjective maps are weaker than surjective submersions. For
example, given a point $x\in X$, it is easy to see, by item (2) above, the natural map $X\sqcup\{x\} \to X$ is a
stalkwise surjective map but not a surjective submersion. 
Also note that Remark \ref{rmk:ss_vs_open} implies that
Lemma \ref{lemma:stalkwise-surj} also holds if we replace 
$\covers_{ss}$ with $\covers_{\open}$.

\end{remark}

\begin{lemma} \label{lem:stalk_surj_top}
Let $W\xrightarrow{q}V \xleftarrow{p} U$ be morphisms in $\Mfd$ such
that $\yon(q)$ is stalkwise surjective, and
consider the pullback diagram in the category of topological spaces
\[
\begin{tikzpicture}[descr/.style={fill=white,inner sep=2.5pt},baseline=(current  bounding  box.center)]
\matrix (m) [matrix of math nodes, row sep=2em,column sep=3em,
  ampersand replacement=\&]
  {  
U \times_V W \& W \\
U  \& V \\
}; 
  \path[->,font=\scriptsize] 
   (m-1-1) edge node[auto] {$\tilde{p}$} (m-1-2)
   (m-1-1) edge node[auto,swap] {$\tilde{q}$} (m-2-1)
   (m-1-2) edge node[auto,swap] {$q$} (m-2-2)
   (m-2-1) edge node[auto] {$p$} (m-2-2)
  ;
%begin pullback symbol%
  \begin{scope}[shift=($(m-1-1)!.4!(m-2-2)$)]
  \draw +(-0.25,0) -- +(0,0)  -- +(0,0.25);
  \end{scope}
  %end pullback symbol%

\end{tikzpicture}
\]
Suppose that there exists a manifold $X \in \Mfd$ and a continuous
surjective map $X\xrightarrow{\pi} U\times_{V}W$ such that $
\tilde{q}\circ \pi \maps X \to U$ is smooth and the composition 
$\tilde{p}\circ \pi \maps X \to W$ is a smooth surjective submersion.
Then $p \maps U \to V $ is a surjective submersion and therefore
$U\times_{V} W$ is representable in $\Mfd$.
\end{lemma}
\begin{proof}
  The surjectivity of $p$ follows from an easy set-theoretical
  argument, so we will show it is a surjective submersion.  Let $u \in
  U$ and $v=p(u)$.  Since $\yon(q)$ is stalkwise surjective, Lemma
  \ref{lemma:stalkwise-surj} implies there exists 
 an open neighborhood $O_v$ of $v$, an element $w\in q^{-1}(v) \sse W$, and
 a smooth section
\[
\sigma_{WV}: O_v \to O_w:= q^{-1}(O_v),
\]
such  that $\sigma_{WV}(v)=w, \quad q \circ \sigma_{WV}=\id_{O_v}$.
Hence the restriction $q \vert_{O_w}$ is a submersion at $w$ (i.e.,
$T_w q$ is surjective and its kernel splits). The inverse function
theorem \cite[Cor.\ I.5.2s]{Lang:95} then implies that, by taking $O_v$
to be a small enough neighborhood, we can express $q \vert_{O_w}$ as a
projection. Therefore, the following pullback diagram exists in the
category $\Mfd$:
\[
\begin{tikzpicture}[descr/.style={fill=white,inner sep=2.5pt},baseline=(current  bounding  box.center)]
\matrix (m) [matrix of math nodes, row sep=2em,column sep=3em,
  ampersand replacement=\&]
  {  
O_u \times_{O_v} O_w \& O_w \\
O_u  \& O_v \\
}; 
  \path[->,font=\scriptsize] 
   (m-1-1) edge node[auto] {$\tilde{p}$} (m-1-2)
   (m-1-1) edge node[auto,swap] {$\tilde{q}$} (m-2-1)
   (m-1-2) edge node[auto,swap] {$q$} (m-2-2)
   (m-2-1) edge node[auto] {$p$} (m-2-2)
  ;
%begin pullback symbol%
  \begin{scope}[shift=($(m-1-1)!.4!(m-2-2)$)]
  \draw +(-0.25,0) -- +(0,0)  -- +(0,0.25);
  \end{scope}
  %end pullback symbol%

\end{tikzpicture}
\]
where $O_u:=p^{-1}(O_v)$, and we suppress restrictions of the morphisms.

Now, by hypothesis, we have a surjective continuous map $\pi \maps X
\to U \times_V W$ such that $\pi \circ \tilde{q}$ is smooth and 
 $\pi \circ \tilde{p}$ is a surjective
submersion. Let $x \in X$ such that $\pi(x)=(u,w)$.  
We may shrink $O_v$ further if necessary, so that we have a smooth
section 
\[
\sigma_{XW} \maps O_w \to O_x:=(\tilde{p}\circ \pi)^{-1}(O_w)
\]
such that  $\bigl (\tilde{p} \circ \pi \bigr) \circ \sigma_{XW} =
\id_{O_w}$ and $\sigma_{XW}(w)=x$. It is not difficult to see that we
have the following equalities of open sets:
\[
\tilde{p}^{-1}(O_w)= O_u \times_{O_v} O_w, \quad \pi^{-1}( O_u \times_{O_v}
O_w)= O_x.
\]
Since $\tilde{p} \circ \pi \circ \sigma_{XW}=\id_{O_w}$, we see that 
\[
\pi\circ \sigma_{XW} \maps O_w\to O_u \times_{O_v} O_w
\]
is a continuous section of $\tilde{p} \maps O_u \times_{O_v}
O_w \to O_w$. In fact, $\pi\circ \sigma_{XW}$ is smooth since the factors
$\tilde{q} \circ \pi\circ \sigma_{XW}$ and $\tilde{p} \circ \pi\circ
\sigma_{XW}$ are compositions of smooth maps. 

The commutativity of the pushout diagram gives us
$p \circ \tilde{q} \circ \pi \circ \sigma_{XW} = q$.
Hence,  
\[
\tilde{q} \circ \pi \circ \sigma_{XW} \circ \sigma_{WV} \maps O_v \to O_u
\]
is our desired smooth section of $p$. Therefore $p$ is a  submersion. 
\end{proof}

%\ccomment{made changes below}
\begin{proposition}\label{cor:1}
The category of Banach manifolds $\Mfd$ equipped with the surjective
submersion pretopology $\covers_{\subm}$, and the collection of jointly conservative points 
$\pts_{\Ban}$ is a locally stalkwise pretopology. 
\end{proposition}

\begin{proof}
  First, we show the ``2-out-of-3'' property of Def.\
  \ref{def:LSW_covers} is satisfied. Indeed, this follows immediately
  from Lemma \ref{lem:surj_sub_2of3}, since if $\yon(f) \maps \yon(X)\to
  \yon(Y)$ is stalkwise surjective with respect to $\pts_{\Ban}$ then $f
  \maps X \to Y$ is surjective.

Next, we show that the locality property of Def.\ \ref{def:LSW_covers} is
satisfied. Let $W\xrightarrow{q}V \xleftarrow{p} U$ be morphisms
in $\Mfd$ and suppose $\yon q \maps \yon W \to \yon V$ is a stalkwise surjection. 
Furthermore, suppose we have a diagram of sheaves
\[
\begin{tikzpicture}[descr/.style={fill=white,inner sep=2.5pt},baseline=(current  bounding  box.center)]
\matrix (m) [matrix of math nodes, row sep=2em,column sep=3em,
  ampersand replacement=\&]
  {  
\yon X \& \yon U \times_{\yon V} \yon W \&  \yon W \\
\& \yon U  \& \yon V \\
}; 
  \path[->,font=\scriptsize] 
  (m-1-1) edge node[auto] {$f$} (m-1-2)
   (m-1-2) edge node[auto] {$\widetilde{\yon p}$} (m-1-3)
   (m-1-2) edge node[auto,swap] {$\widetilde{\yon q}$} (m-2-2)
   (m-1-3) edge node[auto,swap] {$\yon q$} (m-2-3)
   (m-2-2) edge node[auto] {$\yon p$} (m-2-3)
  ;
%begin pullback symbol%
  \begin{scope}[shift=($(m-1-2)!.4!(m-2-3)$)]
  \draw +(-0.25,0) -- +(0,0)  -- +(0,0.25);
  \end{scope}
  %end pullback symbol%

\end{tikzpicture}
\]
which exhibits the pullback $\widetilde{\yon p}$ as a local stalkwise cover, i.e.\ 
$f$ is a stalkwise surjection, and $\widetilde{\yon p} \circ f$ is represented by a surjective submersion $g \maps X \to W$.
Note that the Yoneda lemma implies that there exists a smooth map $h \maps X \to U$ representing
the composition $\widetilde{\yon q} \circ f$.
Moreover, since $f$ is stalkwise surjective, it is represented in the category of topological spaces by a surjective map $\pi=(h,g) \maps X \to U \times_{V} W$.
This observation, combined with the fact  that $g$ is a surjective submersion, implies that the hypotheses of
Lemma \ref{lem:stalk_surj_top} are satisfied, and therefore we conclude that $p
\maps U \to V$ is a cover.

% in which $\yon q$ and $f$ are stalkwise surjections, and
% $\yon g:=\widetilde{\yon p} \circ f$ is represented by a surjective
% submersion. More precisely, 
% the Yoneda lemma implies that there exists smooth maps
% $g \maps X \to W$ and $h \maps X \to U$ such that $\yon g = \widetilde{\yon p}
% \circ f$ and $\yon h = \widetilde{\yon q} \circ f$. Since $f$ is
% stalkwise surjective, $f$ is surjective as a map between underlining topological
%  spaces and $g$ is a surjective submersion. We now have satisfied all the hypoth
% eses of
% Lemma \ref{lem:stalk_surj_top}, and therefore we can conclude that $p
% \maps U \to V$ is a cover. 

\end{proof}

\section{Higher groupoids as an incomplete category
  of fibrant objects} \label{sec:gpd_icfo}
We now fix a category $(\Cat,\covers,\pts)$ equipped
with a locally stalkwise pretopology and prove that the corresponding category
of $\infty$-groupoids admits an iCFO structure.
The results in Sec.\ \ref{sec:local_sw_BanMfd} then imply that we obtain an
iCFO structure for the category of Lie $\infty$-groupoids
as a special case.

After discussing some aspects of the simplicial localization of 
category of Lie $\infty$-groupoids in Sec.\ \ref{sec:simp_loc_gpd}, 
we analyze in Sec.\ \ref{sec:more-on-w-eq} the weak equivalences for
this iCFO structure in more detail. In particular, we recall that the
weak equivalences are completely characterized by the acyclic
fibrations, which in this case are, respectively, stalkwise weak equivalences and hypercovers.
Thanks to a result of Behrend and Getzler \cite{Behrend-Getzler:2015}, we obtain a sheaf-theoretic independent characterization of weak equivalences, without any mention of the
collection of points $\pts$.

\subsection{Path object}
We first construct a candidate for a path object. This will require
several steps. Later in Sec.\ \ref{sec:iCFO_struct}, we verify that
this gives a path object as part of the iCFO structure for higher
groupoids.  Our construction, in particular the proof of Prop.\
\ref{prop:pobj_rep}, is essentially identical to that of Behrend and
Getzler \cite[Thm.\ 3.21]{Behrend-Getzler:2015}.  However, as
previously mentioned, we do not assume the existence of finite
limits in our category $\Cat$. Moreover, our definition of weak
equivalences, necessary for our application in Sec.\
\ref{sec:int_exact} is different than the one given in
\cite{Behrend-Getzler:2015}. Hence, it is necessary to present a
verification that the Behrend--Getzler construction works in our context.

Given a simplicial sheaf $F \in \sSh(\Cat)$ and 
a simplicial set $K \in \sSet$, denote by $F^K$ the simplicial sheaf
\begin{equation}
\label{eq:pobj_sheaf}
 F^{K}(U)_n:=\hom_{\sSet}\bigl(\Delta^n \times K,F(U)
 \bigr)=M_{\Delta^n \times K}(F),
\end{equation}
where $M_{\Delta^n \times K}(F)$ is the aforementioned
matching object (Def.\ \ref{def:matching}).
A natural path object
for $F$ is the simplicial sheaf $F^{\Delta^1}$.
The inclusion of simplicial
sets
\[
\partial \Delta^1 \cong \Delta^0 \cup \Delta^0 \xto{d^0 \cup d^1} \Delta^1
\]
induces a morphism of simplicial sheaves
\begin{equation} \label{eq:pobj_d0d1}
F^{\Delta^1} \xto{(d^{0 \ast},d^{1 \ast})} F \times F \cong F^{\Delta^0 \cup \Delta^0}.
\end{equation}
Also, the constant map $s^0 \maps \Delta^1 \to \Delta^0$ gives us a
map of simplicial sheaves 
\begin{equation} \label{eq:pobj_s0}
F \cong F^{\Delta^0} \xto{s^{0 \ast}} F^{\Delta^1}.
\end{equation}

We now make the following observation which justifies our
consideration of $F^{\Delta^1}$:
\begin{proposition} \label{prop:stalkwise_pobject}
Let $X$ be a higher groupoid in $(\Cat,\covers,\pts)$. In the diagram
\[
\yon X \xto{s^{0 \ast}} (\yon X)^{\Delta^1} \xto{(d^{0 \ast},d^{1
    \ast})} \yon X \times \yon X. 
\]
 the map $s^{0 \ast}$ is a stalkwise weak equivalence and the map $(d^{0 \ast},d^{1\ast})$ is a stalkwise Kan fibration.
 \end{proposition}
 \begin{proof}
Recall that Corollary \ref{cor:kan} implies that $\ppt  X:=\ppt  \yon X$ is
a Kan complex for any point $\ppt \in \pts$.
Equation \ref{eq:pobj_sheaf} and Prop.\ \ref{prop:matchobj} imply that the
image of the above diagram under $\ppt $ is naturally
isomorphic to the diagram of simplicial sets
\[
\ppt X \xto{s^{0 \ast}} (\ppt  X)^{\Delta^1} \xto{(d^{0 \ast},d^{1
    \ast})} \ppt  X \times \ppt  X. 
\]
This is the usual diagram which exhibits the Kan complex
$(\ppt  X)^{\Delta{^1}}=\Map_{\sSet}(\Delta^1,\ppt X)$ as the path object of
$\ppt X$. Hence, $s^{0 \ast}$ is a weak equivalence and 
$(\ppt  d^{0 \ast},\ppt  d^{1 \ast})$ is a fibration of simplicial sets
for any point $\ppt $.
\end{proof}

The above proposition suggests that, to construct a path
object for an $n$-groupoid $X$ in $\Cat$, we should show that the
simplicial sheaf $(\yon X)^{\Delta{^1}}$ is representable by a higher
groupoid $X^{\Delta^1}$, and that the map of higher groupoids
$ X^{\Delta^1} \to X \times X$ is not just a stalkwise Kan fibration,
but a Kan fibration in the sense of Def.\ \ref{def:Kan_arrow}.

As a first step, we prove the following Proposition:

\begin{proposition} \label{prop:pobj_rep}
If $X$ is an $\infty$-groupoid in $(\Cat, \covers)$, then there is a canonical
$\infty$-groupoid  $X^{\Delta^1}$ in $(\Cat, \covers)$  representing the simplicial sheaf
$(\yon X)^{\Delta^1}$.
\end{proposition}
\emptycomment{
Before we prove this, let us first recall that if $F \in ss
\Sh(\Cat)$ is a bisimplicial sheaf, then $\codiag F$ can be expressed
explicitly in terms of iterated pullbacks. Denote by $d^h_i \maps
F_{p,q} \to F_{p-1,q}$ and $d^v_{i} \maps F_{p,q} \to F_{p,q-1}$ the
horizontal and vertical face maps of $F$. Then we have
\begin{equation}  \label{eq:artin-mazur}
(\codiag F)_n = F_{0,n} ~ \prescript{}{d^v_0}\times_{d^h_1} ~ F_{1,n-1}
~ \prescript{}{d^v_0}\times_{d^h_2} \cdots
\prescript{}{d^v_0}\times_{d^h_{i}} ~ F_{i,n-i} ~ \prescript{}{d^v_0}
\times_{d^h_{i+1}} \cdots \prescript{}{d^v_0} \times_{d^h_n} ~ F_{n,0}
\end{equation}
The face and degeneracy maps 
$d_i \maps (\codiag F)(U)_n \to (\codiag F)(U)_{n-1}$, $s_i \maps (\codiag
F)(U)_n \to (\codiag F)(U)_{n+1}$ can be written as \cite{Stevenson:2012}
\begin{equation}\label{eq:di_si} 
\begin{split}
 d_i&= (d^v_i p_0, d^v_{i-1} y_1, \dots, d^v_1
p_{i-1}, d_i^h p_{i+1}, \dots, d^h_ip_n),  \\
s_i &= (s^v_i p_0, s^v_{i-1} p_1, \dots, s^v_0p_i,
s^h_i p_i, s^h_i p_{i+1}, \dots, s^h_i p_n) 
\end{split}
\end{equation}
where $p_i \maps \codiag(F)_n \to F_{i,n-i}$ is the canonical projection.}

We need the
following Lemma whose proof we will postpone for the moment.
\begin{lemma} \label{lem:path-kan}
The  inclusion 
\[ \Horn{n}{ j} \times \Simp{1} \to \Simp{n} \times \Simp{1}, \]
is a collapsible extension (Def.\ \ref{def:collapsible}).
\end{lemma}

\begin{proof}[Proof of Prop.\ \ref{prop:pobj_rep}]
 Let $X$ be an $\infty$-groupoid in $\Cat$. Proposition \ref{prop:matchobj}
and \eqref{eq:pobj_sheaf} imply that
% , the simplicial sheaf
 % $(\yon X)^{\Delta^1}$ satisfies 
\[
(\yon
 X)^{\Delta^1}_n \cong \Hom(\Delta^n\times \Delta^1, X),
\] 
on each level $n$. %We will show $\Hom(\Delta^n\times \Delta^1, X)$ is representable.
It follows from the presentation of $\Delta^n\times \Delta^1$ (e.g., see Appendix \ref{sec:lem_proof})
that $\Hom(\Delta^n\times \Delta^1, X)$ is represented by
\[
X_{n+1} ~ \prescript{}{d_1}\times_{d_1} ~ X_{n+1}
~ \prescript{}{d_2}\times_{d_2} \cdots
\prescript{}{d_i}\times_{d_{i}} ~ X_{n+1} ~ \prescript{}{d_{i+1}}
\times_{d_{i+1}} \cdots \prescript{}{d_n} \times_{d_n} ~ X_{n+1}.
\] 
Lemmas \ref{lem:example_collapsible} and \ref{lem:covers_in_groupoid}
imply that the face maps of $X$ are covers. Hence, the above pullbacks exist in $\Cat$. 
Therefore, $(\yon X)^{\Delta^1}$ is represented by a
simplicial object $X^{\Delta ^1}$.

Next we prove that $X^{\Delta^{1}}$ is an $\infty$-groupoid in
$\Cat$. By definition, we need to show the condition $\Kan(n,j)$
holds, i.e.,
$\Hom(\Horn{n}{j},X^{\Delta^1})$ is representable and the map induced by the inclusion
 \[
 \Hom(\Delta^{n},X^{\Delta^1}) \to \Hom(\Horn{n}{j},X^{\Delta^1})
 \]
is a cover, for all $n \geq 1$ and $0 \leq j \leq n$.
Note that Prop.\ \ref{prop:matchobj} implies that  we have a natural isomorphism, 
\[
\Hom(\Horn{n}{j},X^{\Delta^1}) \cong \Hom(\Horn{n}{j}  \times \Delta^1,X).
\]

We proceed by induction. For the $n=1$ case, we have
\[
\Hom(\Delta^1 \times \Delta^1,X) \to \Hom(\Horn{1}{j}  \times
\Delta^1,X) \cong \Hom(\Delta^1,X) \cong X_{1}.
\]
Since $X$ is a higher groupoid, Lemma \ref{lem:path-kan} combined with
Lemma \ref{lem:covers_in_groupoid} imply that  $\Kan(1,j)$ is satisfied. 

Now assume $n>1$ and that $\Kan(m,j)$ holds for all $m < n$ and $0 \leq
j \leq m$. This plus the fact that $\Horn{n}{j}$ is a collapsible subset of $\Delta^n$,
allows us to apply Lemma \ref{lem:representable} to conclude that 
$\Hom(\Horn{n}{j},X^{\Delta^1})$ is representable. Hence,
Lemma \ref{lem:path-kan} again combined with
Lemma \ref{lem:covers_in_groupoid} imply that  $\Kan(n,j)$ is satisfied.
This completes the proof.
\end{proof}

It remains to prove Lemma \ref{lem:path-kan}. We start with the
following auxiliary Lemma:

\begin{lemma} \label{lem:path-kan-aux}
Let $f\maps \Horn{1}{i} \to \Simp{1}$, for either $i= 0$ or $i= 1$ be
the usual inclusion. For every boundary extension $\iota: S\to T$
the induced map
\[ 
(S \times \Simp{1}) \sqcup_{S \times \Horn{1}{i} } (T \times \Horn{1}{i}) \to T \times \Simp{1},
\] 
is a collapsible extension.
\end{lemma}

\begin{proof}
The case when $\iota: S \to T$ is the standard inclusion $\partial \Simp{n}
\to \Simp{n}$ is proven in \cite[Lemma 3.3.3]{Hovey:1999}. 
We next observe the following 
following fact: 
If $F \maps \sSet \times \sSet \to
\sSet$ is a co-continuous functor and
\[
\begin{tikzpicture}[descr/.style={fill=white,inner sep=2.5pt},baseline=(current  bounding  box.center)]
\matrix (m) [matrix of math nodes, row sep=2em,column sep=3em,
  ampersand replacement=\&]
  {  
A  \&  B  \\
S \&  T  \\
}; 
\path[->,font=\scriptsize] 
 (m-1-1) edge node[auto] {$$} (m-1-2)
 (m-1-1) edge node[auto,swap] {$$} (m-2-1)
 (m-1-2) edge node[auto,swap] {$$} (m-2-2)
 (m-2-1) edge node[auto] {$$} (m-2-2)
;

%begin pushout symbol%
 \begin{scope}[shift=($(m-1-1)!.65!(m-2-2)$)]
 \draw +(.25,0) -- +(0,0)  -- +(0,.-.25);
 \end{scope}
 %end pushout symbol%
\end{tikzpicture}
\]
is a pushout square of simplicial sets, then the diagram
\[
\begin{tikzpicture}[descr/.style={fill=white,inner sep=2.5pt},baseline=(current  bounding  box.center)]
\matrix (m) [matrix of math nodes, row sep=2em,column sep=3em,
  ampersand replacement=\&]
  {  
F(A, \Simp{1}) \sqcup_{F(A, \Horn{1}{i})} F(B,\Horn{1}{i}) \& F(B, \Simp{1}) \\ 
F(S, \Simp{1}) \sqcup_{F(S, \Horn{1}{i})} F(T, \Horn{1}{i}) \& F(T, \Simp{1})\\
}; 
\path[->,font=\scriptsize] 
 (m-1-1) edge node[auto] {$$} (m-1-2)
 (m-1-1) edge node[auto,swap] {$$} (m-2-1)
 (m-1-2) edge node[auto,swap] {$$} (m-2-2)
 (m-2-1) edge node[auto] {$$} (m-2-2)
;

%begin pushout symbol%
 \begin{scope}[shift=($(m-1-1)!.65!(m-2-2)$)]
 \draw +(.25,0) -- +(0,0)  -- +(0,.-.25);
 \end{scope}
 %end pushout symbol%
\end{tikzpicture}
\]
 is again a pushout square.
(See \cite[Lemma 2.42]{Li:2015}, where $F$ is taken to
be the join functor, and observe that only the co-continuity of the
join is used in the proof.)

Hence, to prove the statement, we take $F$ above to be the product
functor and proceed by induction.
\end{proof}

Since any collapsible extension is a boundary extension, we have the
following corollary:
\begin{cor}\label{cor:c-to-c}
Let $f\maps \Horn{1}{i} \to \Simp{1}$, for either $i= 0$ or $i= 1$ be
the usual inclusion. For every collapsible extension $\iota: S\to T$
the induced map
\[ 
(S \times \Simp{1}) \sqcup_{S \times \Horn{1}{i} } (T \times \Horn{1}{i}) \to T \times \Simp{1},
\] 
is a collapsible extension.
\end{cor}

\begin{proof} [Proof of Lemma \ref{lem:path-kan}]
We observe that  the morphism $\Horn{n}{ j} \times \Simp{1} \to \Simp{n}
\times \Simp{1}$ is a composition of the following maps \[\Horn{n}{j} \times \Simp{1} \to  (\Horn{n}{j} \times \Simp{1}) \sqcup_{\Horn{n}{j} \times \Horn{1}{
    i} } (\Simp{n} \times \Horn{1}{i}) \to \Simp{n}
\times \Simp{1}. \] 
The first map is clearly a collapsible extension (note
that $\Horn{1}{i}=\Delta^0$). The second map is also a collapsible extension by
Cor. \ref{cor:c-to-c}.
\end{proof}

\begin{remark}\label{rem:pobj_rep}
If $X$ in Prop.\ \ref{prop:pobj_rep} is actually a $k$-groupoid for
$k< \infty$, then the above proof can be refined to show that
 $X^{\Delta^1}$ is also an $k$-groupoid. We just need to verify that
the cover
\[
\Hom(\Delta^n \times \Delta^1, X) \to  \Hom(\Horn{n}{j} \times \Delta^1, X)
\]
is an isomorphism for $n>k$. By Remark \ref{rem:n-gpd}, this will be
true provided we can show that the collapsible extension
\begin{equation} \label{eq:rem:pobj_rep}
\Horn{n}{j} \times \Delta^1 \to \Delta^n \times \Delta^1
\end{equation}
is the pushout of maps of the form $\Horn{m}{l} \to \Delta^m$ with $m >k$.
Indeed, this is the case. The collapsible extension $\Horn{n}{j} \to
\Delta^n$ is the composition of two boundary extensions
$\Horn{n}{j} \to \partial \Delta^n \to \Delta^n$. We obtain $\partial
\Delta^n$ from $\Horn{n}{j}$ by attaching $\Delta^{n-1}$ along
$\partial \Delta^{n-1} \to \Horn{n}{j}$. Therefore, the collapsible
extension obtained from  $S=\Horn{n}{j} \to \Delta^n=T$ from the auxiliary
Lemma \ref{lem:path-kan-aux} above involves 
pushouts of maps of the form $\Horn{n}{l} \to \Delta^n$ and
$\Horn{n+1}{l} \to \Delta^{n+1}$. (See the proof of Lemma 3.3.3 in \cite{Hovey:1999}.)
From the proof of Lemma \ref{lem:path-kan}, we can then conclude that
the collapsible extension \eqref{eq:rem:pobj_rep} only involves pushouts
along inclusions of horns which have dimension $>k$.
\end{remark}

\subsection{The iCFO structure} \label{sec:iCFO_struct}
We arrive at our first main result:

\begin{thm} \label{thm:icfo_struct}
Let $(\Cat,\covers,\pts)$ be a category equipped with a
locally stalkwise pretopology with respect to a jointly conservative
collection of points.  The category $\gpd{\Cat, \covers}$, whose objects are 
$\infty$-groupoids in $(\Cat,\covers)$, and whose morphisms are simplicial maps
is an incomplete category of fibrant objects in which:
\begin{itemize}
\item the weak equivalences are the stalkwise weak equivalences (Def.\
  \ref{def:stalk_weq}),

\item the fibrations are the Kan fibrations (Def.\ \ref{def:Kan_arrow}),

\item the acyclic fibrations are hypercovers (Def.\ \ref{def:equivalence}).
\end{itemize}
In particular, the category of Lie $\infty$-groupoids
\[
\LnGpd{\infty}:=\gpd{\Mfd, \covers_{\subm}, \pts_{\Ban}}
\]
is an incomplete category of fibrant objects in this way.
\end{thm}
Let us make a number of important remarks before we proceed to the proof.

\begin{remark} \label{rmk:iCFO_struct1} 
\mbox{}
\begin{enumerate}
\item The acyclic fibrations are
  obviously determined by the weak equivalences and fibrations, i.e.,
  those Kan fibrations that are also stalkwise weak
  equivalences. Since we are working with a locally stalkwise
  pretopology, it follows from Prop.\ \ref{prop:hypercover} that
  acyclic fibrations are precisely the hypercovers. 

\item In particular, if $\Cat$ is small and has all finite limits, then
  $(\Cat,\covers)$ is a ``descent category'' in the sense of
  Behrend--Getzler. (See proof of Cor.\ \ref{cor:w-eq-combinatoric} for more details.)
In this case, the iCFO structure on $\gpd{\Cat,
    \covers}$ agrees with the CFO structure of Thm.\ 3.6 in \cite{Behrend-Getzler:2015}.

\item In the proof of Thm\ \ref{thm:icfo_struct}, we use the
  assumption that $\covers$ is a locally stalkwise pretopology only in
  the proof of Prop.\ \ref{prop:axiom5}, where we show that 
  pullbacks of acyclic fibrations always exist.

\item The fact that the category of Lie $\infty$-groupoids is
an example of such an iCFO follows immediately from Prop.\ \ref{cor:1}.

\item The proof below of Thm.\ \ref{thm:icfo_struct} can be enhanced to show that
the category $\mathsf{Gpd}_{n}[\Cat,\covers]$ of $n$-groupoids in
$(\Cat,\covers)$ for $n < \infty$ also forms an iCFO. It
follows from Remark \ref{rem:pobj_rep} that the path object
$X^{\Delta^1}$ associated to an $n$-groupoid $X$ will also be an
$n$-groupoid. 
The only other modifications needed are in the proof of Prop.\ \ref{prop:axiom4} below.
(See Remark \ref{rmk:axiom4}.)
\end{enumerate}
\end{remark}

The following collection of propositions proves Thm.\ \ref{thm:icfo_struct} by
directly verifying the axioms of Def.\ \ref{def:catfibobj} 
We begin with the easiest axioms to verify:
\begin{proposition}[Axioms 1, 2, 7] \label{prop:ax127}
\mbox{}
\begin{enumerate}
\item Every isomorphism in $\gpd{\Cat, \covers}$ is a stalkwise weak
  equivalence and a Kan fibration.
\item If $f$ and $g$ are composable morphisms in $\gpd{\Cat,
    \covers}$, and any two of $f$, $g$, or $g\circ f$ are stalkwise weak
  equivalences, then so is the third.
\item If $X$ is an $\infty$-groupoid in $(\Cat, \covers)$, then $X \to
  \ast$ is a Kan fibration.
\end{enumerate}
\end{proposition}

\begin{proof}
(1) Obvious. (2) Follows from the fact that weak equivalences of
simplicial sets satisfy the analogous 2 out of 3 axiom. (3) 
Follows immediately from Def.\ \ref{def:ngpd}.
\end{proof}

This next proposition implies that the composition of two Kan
fibrations is again a Kan fibration. A similar result appears
in \cite[Lemma 2.7]{Henriques:2008} and also in \cite[Theorem 2.17 (3)]{Wolfson:2016}, where Kan fibrations are called ``$n$-stacks''.

\begin{proposition}[Axiom 3] \label{prop:axiom3}
  \label{pro:comp}
  Let $f \maps X\to Y$ and $g\colon Y\to Z$ be morphisms of
  simplicial objects in $\Cat$.  If $f$ and $g$ satisfy
  $\Kan(n,j)$ and the sheaf  
\[
\Hom(\Horn{n}{j}\xto{\iota} \Simp{n}, X \xto{g\circ f} Z)
\]
 is representable, then $g\circ f$ also satisfies $\Kan(n,j)$.
 Similarly, if  $f$ and $g$ satisfy $\Kan!(n,j)$, then $g\circ f$ also
  satisfies $\Kan!(n,j)$.
\end{proposition}

\begin{proof}
The proposition, in fact, follows 
from Wolfson's Theorem 2.17 (3) in \cite{Wolfson:2016}. Indeed, axiom 1 in our Def.\ \ref{def:LSW_covers} implies axiom 3 in \cite[Sec.\ 2]{Wolfson:2016} for a category equipped with a subcategory of covers. For the reader's convenience we recall the proof 
from \cite{Wolfson:2016} using our notation.

We have the following diagram containing three pullback squares:
\[
\begin{tikzpicture}[descr/.style={fill=white,inner sep=2.5pt},baseline=(current  bounding  box.center)]
\matrix (m) [matrix of math nodes, row sep=2em,column sep=1em,
  ampersand replacement=\&]
  {  
\Hom(\Delta^n,X) \& \Hom(\Horn{n}{j}\xto{\iota} \Simp{n}, X\xto{f} Y) \&
\Hom(\Horn{n}{j}\xto{\iota} \Simp{n}, X\xto{g\circ f} Z) \&
\Hom(\Horn{n}{j},X) \\
\& \Hom(\Delta^n,Y) \& \Hom(\Horn{n}{j}\xto{\iota} \Simp{n}, Y\xto{g} Z) \&
\Hom(\Horn{n}{j},Y) \\
\& \& \Hom(\Delta^n,Z) \& \Hom(\Horn{n}{j},Z)\\
}; 
 \path[->,font=\scriptsize] 
  (m-1-1) edge node[auto] {$(\iota^\ast,f_\ast)$} (m-1-2)
  (m-1-2) edge node[auto] {$g_\ast$} (m-1-3)
  (m-1-3) edge node[auto] {$\pr$} (m-1-4)
  (m-2-2) edge node[auto] {$(\iota^\ast,g_\ast)$} (m-2-3)
  (m-2-3) edge node[auto] {$\pr$} (m-2-4)
  (m-1-2) edge node[auto] {$\pr$} (m-2-2)
  (m-1-3) edge node[auto] {$f_\ast$} (m-2-3)
  (m-1-4) edge node[auto] {$f_\ast$} (m-2-4)
  (m-2-3) edge node[auto] {$\pr$} (m-3-3)
  (m-2-4) edge node[auto] {$g_\ast$} (m-3-4)
  (m-3-3) edge node[auto] {$\iota^\ast$} (m-3-4)
 ;

%begin pullback symbol%
 \begin{scope}[shift=($(m-1-2)!.4!(m-2-3)$)]
 \draw +(-0.25,0) -- +(0,0)  -- +(0,0.25);
 \end{scope}
 %end pullback symbol%

%begin pullback symbol%
 \begin{scope}[shift=($(m-1-3)!.4!(m-2-4)$)]
 \draw +(-0.25,0) -- +(0,0)  -- +(0,0.25);
 \end{scope}
 %end pullback symbol%

%begin pullback symbol%
 \begin{scope}[shift=($(m-2-3)!.4!(m-3-4)$)]
 \draw +(-0.25,0) -- +(0,0)  -- +(0,0.25);
 \end{scope}
 %end pullback symbol%
\end{tikzpicture}
\]
If $g$ satisfies $\Kan(n,j)$, then
$\Hom(\Delta^n,Y) \xto{(\iota^*,g_*)} \Hom(\Horn{n}{j}\xto{\iota} \Simp{n}, Y\xto{g} Z)$ 
is a cover. Hence, $\Hom(\Horn{n}{j}\xto{\iota} \Simp{n}, X\xto{f} Y) \xto{g_\ast}
\Hom(\Horn{n}{j}\xto{\iota} \Simp{n}, X\xto{g\circ f} Z)$ is a cover.
If $f$ satisfies $\Kan(n,j)$, then 
$\Hom(\Delta^n,X) \xto{(\iota^*,f_*)} \Hom(\Horn{n}{j}\xto{\iota} \Simp{n}, X\xto{f} Y)$
is a cover. Hence, the composition 
\[
g_\ast \circ (\iota^\ast,f_\ast)=(\iota^\ast, (g \circ f)_\ast) \maps
\Hom(\Delta^n,X) \to \Hom(\Horn{n}{j}\xto{\iota} \Simp{n},
X\xto{g\circ f} Z)
\]
 is a cover, and so
$g \circ f$ satisfies $\Kan(n,j)$. The same argument \textit{mutatis mutandis}
shows that if $f$ and $g$ satisfy $\Kan!(n,j)$, then so does $g
\circ f$.
\end{proof}

We next show that the pullback of a Kan fibration is a Kan fibration,
provided the pullback exists in $\gpd{\Cat, \covers}$. As previously
mentioned, this is the only difference between the axioms for an iCFO
and Brown's axioms for a CFO: we do not require the pullback  along
a fibration to exist in general. It turns out, for $\infty$-groupoids, this generalization is in fact
quite mild. The pullback of along a Kan fibration will always
exist provided the corresponding pullback of the 0-simplices exists
in $\Cat$. (The same result appears as Thm.\ 2.17 (4) in
\cite{Wolfson:2016}.)
% where Kan fibrations are called ``$n$-stacks''.)

\begin{proposition}[Axiom 4] \label{prop:axiom4}
Let  $f \maps X \to Y$ be a Kan fibration in $\gpd{\Cat, \covers}$
and $g \maps Z\to Y$ a morphism in $\gpd{\Cat, \covers}$.
If the pullback $Z_0\times_{Y_0} X_0$ exists in $\Cat$, then:
\begin{enumerate}
\item the pullback $Z_n \times_{Y_n} X_n$ exists in $\Cat$ for all $n
  \geq 0$,
\item the morphism $Z\times_{Y} X \xto{p_f} Z$ induced by pulling back $f$ along $g$ 
is a Kan fibration between simplicial objects in $\Cat$,

\item the pullback $Z\times_{Y} X $ is an object of $\gpd{\Cat,\covers}$.
\end{enumerate}
\end{proposition}

\begin{proof}
For convenience, we use the following notation below: If $K$ is a
simplicial set and $W$ is a simplicial object in $\Cat$, then $K(W)$
is the sheaf $\Hom(K,W)$. Also, we denote by $\Hom(\iota,p_f)$
the sheaf $\Hom(\Horn{n}{j} \xto{\iota} \Delta^n, 
Z\times_{Y} X \xto{p_f} Z)$. Finally, we do not distinguish between a
simplicial object $W$ in $\Cat$, and the representable sheaf $\yon W$.

We shall prove statements (1) and (2) simultaneously:
For all $n \geq 0$, and $0 \leq j \leq n$, we wish to show that
morphism of sheaves
\begin{equation} \label{eq:axiom4.1}
Z_n \times_{Y_n} X_n \xto{(\iota^*,p_{f \ast})} \Hom(\iota,p_f)
\end{equation}
is represented by a cover in $\Cat$. It follows from the definition
Def.\ \ref{eq:matchobj_andre0}, that for a fixed simplicial set $K$, 
the functor $\Hom(K,-) \maps s\Cat \to \Sh(\Cat)$ preserves limits.
Therefore, we have an isomorphism of sheaves
\[
\begin{split}
\Hom(\iota,p_f) &:= \Horn{n}{j} \bigl( Z \times_Y X \bigr)
\times_{\Horn{n}{j}(Z)} Z_n\\
&\cong \bigl( \Horn{n}{j}(Z) \times_{\Horn{n}{j}(Y)} \Horn{n}{j}(X)
\bigr)  \times_{\Horn{n}{j}(Z)} Z_n,
\end{split}
\]
and a commuting diagram of pullback squares
\begin{equation} \label{diag:pbdiag1}
\begin{tikzpicture}[descr/.style={fill=white,inner sep=2.5pt},baseline=(current  bounding  box.center)]
\matrix (m) [matrix of math nodes, row sep=1.5em,column sep=3em,
  ampersand replacement=\&]
  {  
\Hom(\iota,p_f) \& Z_n \\
\Horn{n}{j}(Z) \times_{\Horn{n}{j}(Y)} \Horn{n}{j}(X) \& \Horn{n}{j}(Z)\\
\Horn{n}{j}(X) \& \Horn{n}{j}(Y)\\
}; 
  \path[->,font=\scriptsize] 
   (m-1-1) edge node[auto] {$\pr_1$} (m-1-2)
   (m-1-1) edge node[auto,swap] {$\pr_2$} (m-2-1)
   (m-1-2) edge node[auto] {$\iota^*$} (m-2-2)
   (m-2-1) edge node[auto] {$p_{f\ast}$} (m-2-2)
   (m-2-1) edge node[auto,swap] {$\pr_3$} (m-3-1)
   (m-2-2) edge node[auto] {$g_{\ast}$} (m-3-2)
   (m-3-1) edge node[auto] {$f_{\ast}$} (m-3-2)
  ;

%begin pullback symbol%
  \begin{scope}[shift=($(m-1-1)!.4!(m-2-2)$)]
  \draw +(-0.25,0) -- +(0,0)  -- +(0,0.25);
  \end{scope}
  %end pullback symbol%

 %begin pullback symbol%
  \begin{scope}[shift=($(m-2-1)!.4!(m-3-2)$)]
  \draw +(-0.25,0) -- +(0,0)  -- +(0,0.25);
  \end{scope}
  %end pullback symbol%

\end{tikzpicture}
\end{equation}
Hence, the pasting law for pullbacks gives an isomorphism of sheaves
\begin{equation} \label{eq:diag1}
\begin{split}
\Hom(\iota,p_f) \cong \Horn{n}{j}(X) \times_{\Horn{n}{j}(Y)} Z_n.
\end{split}
\end{equation}
The above isomorphism gives another commuting diagram of pullback squares,
which via the universal property, contains the morphism \eqref{eq:axiom4.1}:
\begin{equation}\label{diag:pbdiag2}
\begin{tikzpicture}[descr/.style={fill=white,inner sep=2.5pt},baseline=(current  bounding  box.center)]
\matrix (m) [matrix of math nodes, row sep=1.5em,column sep=3em,
  ampersand replacement=\&]
  {  
Z_n \times_{Y_n} X_n \& X_n \& \\
 \Hom(\iota,p_f) \&  Y_n \times_{\Horn{n}{j}(Y)} \Horn{n}{j}(X) \& \Horn{n}{j}(X) \\
 Z_n \& Y_n \& \Horn{n}{j}(Y)\\
}; 
  \path[->,font=\scriptsize] 
   (m-1-1) edge node[auto] {$$} (m-1-2)
   (m-1-1) edge node[auto,swap] {$(\iota^\ast,p_{f \ast})$} (m-2-1)
   (m-1-2) edge node[auto] {$(\iota^*,f_*)$} (m-2-2)
   (m-2-1) edge node[auto] {$$} (m-2-2)
   (m-2-1) edge node[auto,swap] {$$} (m-3-1)
   (m-2-2) edge node[auto] {$$} (m-2-3)
   (m-2-2) edge node[auto] {$$} (m-3-2)
   (m-3-1) edge node[auto] {$g_{\ast}$} (m-3-2)
   (m-3-2) edge node[auto] {$\iota^{\ast}$} (m-3-3)
   (m-2-3) edge node[auto] {$f_\ast$} (m-3-3)
  ;
% %begin pullback symbol%
   \begin{scope}[shift=($(m-1-1)!.4!(m-2-2)$)]
   \draw +(-0.25,0) -- +(0,0)  -- +(0,0.25);
   \end{scope}
%   %end pullback symbol%

%  %begin pullback symbol%
   \begin{scope}[shift=($(m-2-1)!.4!(m-3-2)$)]
   \draw +(-0.25,0) -- +(0,0)  -- +(0,0.25);
   \end{scope}
%   %end pullback symbol%

%  %begin pullback symbol%
   \begin{scope}[shift=($(m-2-2)!.4!(m-3-3)$)]
   \draw +(-0.25,0) -- +(0,0)  -- +(0,0.25);
   \end{scope}
%   %end pullback symbol%

\end{tikzpicture}
\end{equation}
Note that since $f \maps X \to Y$ is a Kan fibration, the morphism
$X_n \xto{(\iota^\ast,f_\ast)}  Y_n \times_{\Horn{n}{j}(Y)} \Horn{n}{j}(X)$
is represented by a cover. Hence, to show that the morphism
\eqref{eq:axiom4.1} is represented by a cover, the above diagram
implies that it is sufficient to show that the sheaf
\begin{equation} \label{eq:axiom4.2}
\begin{split}
\Horn{n}{j}(X) \times_{\Horn{n}{j}(Y)} Z_n 
\end{split}
\end{equation} 
is representable for all $n \geq 1$ and $0 \leq j \leq n$.

First consider the $n=1$ case. Diagram \eqref{diag:pbdiag1} 
and the isomorphism \eqref{eq:diag1} imply
that we have the pullback diagram
\begin{equation} \label{diag:pbdiag3}
\begin{tikzpicture}[descr/.style={fill=white,inner sep=2.5pt},baseline=(current  bounding  box.center)]
\matrix (m) [matrix of math nodes, row sep=1em,column sep=1em,
  ampersand replacement=\&]
  {  
X_0 \times_{Y_0} Z_1  \& Z_1 \\
Z_0 \times_{Y_0} X_0 \& Z_0\\
}; 
  \path[->,font=\scriptsize] 
   (m-1-1) edge node[auto] {$$} (m-1-2)
   (m-1-1) edge node[auto,swap] {$$} (m-2-1)
   (m-1-2) edge node[auto] {$\iota^*=d_j$} (m-2-2)
   (m-2-1) edge node[auto] {$$} (m-2-2)
  ;

%begin pullback symbol%
  \begin{scope}[shift=($(m-1-1)!.4!(m-2-2)$)]
  \draw +(-0.25,0) -- +(0,0)  -- +(0,0.25);
  \end{scope}
  %end pullback symbol%

\end{tikzpicture}
\end{equation}
The sheaf $Z_0 \times_{Y_0} X_0$ is representable by hypothesis, and
since $Z$ is an $\infty$-groupoid, $Z_1 \xto{d_j} Z_0$ is a
cover. Therefore,  for $n=1$ and $j=0,1$, the sheaf \eqref{eq:axiom4.2} is representable,
$Z_1 \times_{Y_1}X_1$ is representable, and $p_f \maps Z \times_Y X
\to Z$ satisfies $\Kan(1,j)$.

Now suppose $p_f$ satisfies $\Kan(m,j)$ for all $1 \leq m < n$ and $0
\leq j \leq m$. This plus the fact that $\Horn{n}{j}$ is a collapsible
subset of $\Delta^n$, allows us to apply Lemma \ref{lem:representable}
to conclude that $\Hom(\iota_{n,j},p_f) \cong \Horn{n}{j}(X) \times_{\Horn{n}{j}(Y)} Z_n$. 
is representable. Hence, $p_f \maps Z \times_Y X \to Z$ satisfies
$\Kan(n,j)$, and so it is a Kan fibration.

Statement (3) immediately follows. Indeed, since $Z$ is an $\infty$-groupoid,
$Z \to \ast$ is a Kan fibration. Therefore, Prop.\ \ref{prop:axiom3} 
implies that $Z \times_Y X \to \ast$ is also a Kan fibration.
\end{proof}

\begin{remark} \label{rmk:axiom4} If $X$, $Y$, and $Z$ in the
  statement of Prop.\ \ref{prop:axiom4} are $k$-groupoids, $k <
  \infty$, then one can show that the pullback $Z \times_Y X$, if it
  exists, is also a $k$-groupoid.  As previously mentioned in Remark
  \ref{rmk:iCFO_struct1}, this fact helps show that the
  category $\mathsf{Gpd}_{k}[\Cat,\covers]$ forms an iCFO.  We just
  need to verify that the morphism
\begin{equation} \label{eq:rmk1}
Z_n \times_{Y_n} X_{n} \to \Horn{n}{j}(Z) \times_{\Horn{n}{j}(Y)} \Horn{n}{j}(X)
\end{equation}
is an isomorphism for $n >k$. Since $Z$ is a $k$-groupoid, the
pullback diagram \eqref{diag:pbdiag1} implies that 
\begin{equation}\label{eq:rmk2}
\Hom(\iota_{n,j},p_f) \xto{\pr_2} \Horn{n}{j}(Z) \times_{\Horn{n}{j}(Y)} \Horn{n}{j}(X)
\end{equation}
is an isomorphism. Since $f \maps X \to Y$ is a Kan fibration between
$k$-groupoids, it is not hard to show that the morphism
\[
X_n \xto{(\iota^\ast,f_\ast)} Y_n \times_{\Horn{n}{j}(Y)} \Horn{n}{j}(X)
\]
is an isomorphism for all $n >k$. Hence, the pullback diagram
\eqref{diag:pbdiag2} implies that 
\[
Z_n \times_{Y_n} X_{n}  \xto{(\iota^\ast,p_{f \ast})} \Hom(\iota_{n,j},p_f)
\]
is an isomorphism. Composing this with the isomorphism
\eqref{eq:rmk2}, we conclude that \eqref{eq:rmk1} is an isomorphism.

\end{remark}

Proposition \ref{prop:axiom4} also makes it easy to show that the pullbacks of
acyclic fibrations always exist in $\gpd{\Cat,\covers}$, and are
always acyclic fibrations. (Since acyclic fibrations turn out to be equivalent to
hypercovers, this result is equivalent to \cite[Lemma 2.8]{Zhu:2009a}.)

\begin{proposition}[Axiom 5] \label{prop:axiom5}
Let  $f \maps X \to Y$ be an acyclic fibration in $\gpd{\Cat, \covers}$
and $g \maps Z\to Y$ a morphism in $\gpd{\Cat, \covers}$.
Then the morphism $Z\times_{Y} X \xto{q_f} Z$ induced by pulling back $f$ along $g$ 
is an acyclic fibration.
\end{proposition}
\begin{proof}
Since $f \maps X \to Y$ is an acyclic fibration, Prop.\ \ref{prop:hypercover}
implies that $f$  is a hypercover. Then, by definition,
$f_0 \maps X_0 \to Y_0$ is a cover, and hence the pullback $Z_0
\times_{Y_0} X_0$ exists in $\Cat$. Proposition \ref{prop:axiom4}
therefore implies that the morphism $Z\times_{Y} X \xto{q_f} Z$ is a
Kan fibration in $\gpd{\Cat,\covers}$.

Let $\ppt$ be a point. Then $\ppt f \maps \ppt X \to \ppt Y$ is an acyclic
fibration of simplicial sets. By definition, points preserve finite
limits. Hence, $\ppt \bigl (Z\times_{Y} X \bigr) \xto{\ppt q_f}  \ppt Z$
is the pullback of $\ppt f$, and is therefore a weak equivalence of
simplicial sets. So we conclude $q_f$ is a stalkwise weak equivalence,
hence an acyclic fibration.
\end{proof}

Finally, we show that if $X \in \gpd{\Cat,\covers}$, then
the $\infty$-groupoid $X^{\Delta^1}$ constructed
in Prop.\ \ref{prop:pobj_rep} is a path object for $X$. Let 
\[
X^{\Delta^1} \xto{(d^{0 \ast},d^{1 \ast})} X \times X \cong X^{\Delta^0 \cup \Delta^0},
\]
and
\[
s^{0 \ast} \maps X \to X^{\Delta^1}
\]
denote the morphisms \eqref{eq:pobj_d0d1} and \eqref{eq:pobj_s0},
respectively, induced by the 
inclusions $\partial \Delta^1 \cong \Delta^0 \cup \Delta^0 \xto{d^0 \cup d^1}
\Delta^1$ and the constant map $s^0 \maps \Delta^1 \to \Delta^0$, respectively.

\begin{proposition}[Axiom 6] \label{prop:axiom6}
Let $X \in \gpd{\Cat,\covers}$. Then
\begin{equation} \label{eq:axiom6.1}
X \xto{s^{0 \ast}} X^{\Delta^1} \xto{(d^{0 \ast},d^{1
    \ast})} X \times X. 
\end{equation}
is a factorization of the diagonal map $X \times X$ into a
stalkwise weak equivalence $s_0^\ast \maps X \to X^{\Delta^1}$
followed by a Kan fibration $X^{\Delta^1} \xto{(d^{0 \ast},d^{1
    \ast})} X \times X$.
\end{proposition}

To prove the above, we'll need the following lemma, whose proof we
postpone to Appendix  \ref{sec:lem_proof}.

\begin{lemma} \label{lem:pobj_fib}
 The  inclusion of simplicial sets
\[
\Horn{n}{j}\times\Simp{1} \sqcup_{\Horn{n}{j} \times \partial \Simp{1}} 
  \Simp{n}\times{\partial \Simp{1}} \to \Simp{n} \times \Simp{1}
\]
is a collapsible extension.
\end{lemma}

\begin{proof}[Proof of Prop.\ \ref{prop:axiom5}]
First, observe that $d^{i \ast} \circ s^{0 \ast}=\id_X$. This follows
from fact that the assignment $K \mapsto (\yon X)^K$, where
$(\yon X)^K$ is the simplicial sheaf \eqref{eq:pobj_sheaf}, is a
contravariant functor from simplicial sets to simplicial sheaves.
Hence, \eqref{eq:axiom6.1} is a factorization of the diagonal map.

Next, since we have an isomorphism of sheaves $\yon(X^{\Delta^1}) \cong
(\yon X)^{\Delta^1}$, Prop.\ \ref{prop:stalkwise_pobject} implies that 
$s_0^\ast \maps X \to X^{\Delta^1}$ is a stalkwise weak equivalence.

Now we show $f:=(d^{0 \ast},d^{1 \ast}) \maps X^{\Delta^1} \to X
\times X$ satisfies the condition $\Kan(n,j)$ for all $n \geq 1$ and
$0 \leq j \leq n$, i.e. the morphism of sheaves
%\begin{equation} \label{eq:axiom5.1}
\[
X^{\Delta^1}_n \xto{ (\iota_{n,j}^\ast, f_\ast)} \Hom(\iota_{n,j},f),
\]
%\end{equation}
where $\Hom(\iota_{n,j},f):=
\Hom(\Lambda^n_j \xto{\iota_{n,j}} \Delta^n, X^{\Delta^1}  \xto{f} X
\times X)$ is represented by a cover. It follows from Prop.\ \ref{prop:matchobj}
that for any finitely generated simplicial sets $K$ and $L$,  we have an isomorphism
of sheaves
\[
\Hom(L,X^{K}) \cong \Hom(L \times K, X).
\] 
Therefore, we have the following
isomorphisms of sheaves
\[
\begin{split}
\Hom(\iota_{n,j},f) & \cong \Hom(\Horn{n}{j} \times \partial \Delta^1,
X)\times_{\Hom(\Horn{n}{j} \times \partial \Delta^1,X)}
\Hom(\Horn{n}{j} \times \Delta^1,X)\\
& \cong \Hom \bigl (\Horn{n}{j}\times\Simp{1} \sqcup_{\Horn{n}{j} \times \partial \Simp{1}} 
  \Simp{n}\times{\partial \Simp{1}}, X \bigr)
\end{split}
\]
Hence, showing that $f$ satisfies $\Kan(n,j)$ is equivalent to showing
that the morphism of sheaves
\begin{equation} \label{eq:axiom5.3}
\Hom(\Delta^n \times \Delta^1, X) \xto{ (\iota_{n,j}^\ast, f_\ast)} 
\Hom \bigl (\Horn{n}{j}\times\Simp{1} \sqcup_{\Horn{n}{j} \times \partial \Simp{1}} 
  \Simp{n}\times{\partial \Simp{1}}, X \bigr)
\end{equation}
is a cover. Lemma \ref{lem:pobj_fib} implies that the inclusion
$\Horn{n}{j}\times\Simp{1} \sqcup_{\Horn{n}{j} \times \partial \Simp{1}} 
  \Simp{n}\times{\partial \Simp{1}} \to \Simp{n} \times \Simp{1}$ is a
  collapsible extension. Hence, Lemma \ref{lem:covers_in_groupoid}
  implies that in order to show \eqref{eq:axiom5.3} is a cover, it
  suffices to show that 
\[
\Hom(\iota_{n,j},f) \cong \Hom \bigl (\Horn{n}{j}\times\Simp{1} \sqcup_{\Horn{n}{j} \times \partial \Simp{1}}   \Simp{n}\times{\partial \Simp{1}}, X \bigr)
\]
is representable for all $n$ and $j$.

Consider the $n=1$ case. Then we have the pullback square
\[
\begin{tikzpicture}[descr/.style={fill=white,inner sep=2.5pt},baseline=(current  bounding  box.center)]
\matrix (m) [matrix of math nodes, row sep=2em,column sep=3em,
  ampersand replacement=\&]
  {  
\Hom(\iota_{1,j},f) \& X_1 \times X_1 \\
X_1  \& X_0 \times X_0\\
}; 
  \path[->,font=\scriptsize] 
   (m-1-1) edge node[auto] {$$} (m-1-2)
   (m-1-1) edge node[auto,swap] {$$} (m-2-1)
   (m-1-2) edge node[auto] {$(d_j,d_j)$} (m-2-2)
   (m-2-1) edge node[auto] {$(d_0,d_1)$} (m-2-2)
  ;

%begin pullback symbol%
  \begin{scope}[shift=($(m-1-1)!.4!(m-2-2)$)]
  \draw +(-0.25,0) -- +(0,0)  -- +(0,0.25);
  \end{scope}
  %end pullback symbol%

\end{tikzpicture}
\]
Since $X \times X$ is an $\infty$-groupoid, the map $X_1 \times X_1
\xto{d_j,d_j} X_0 \times X_0$ is a cover. Hence, the pullback
$\Hom(\iota_{1,j},f)$ exists and so $f$ satisfies $\Kan(1,j)$.
Now if $n>1$ and $f$ satisfies $\Kan(m,j)$ for all $m <n$ and $1 \leq
j \leq m$, then Lemma \ref{lem:representable} implies that
$\Hom(\iota_{n,j},f)$ is representable.
\end{proof}

\subsubsection{Simplicial localization for $\gpd{\Cat, \covers}$} \label{sec:simp_loc_gpd}
We note that the path object $X^{\Delta^1}$ used in the proof of Thm.\ \ref{thm:icfo_struct}  
is functorial, in the sense of Sec.\ \ref{sec:simp_loc}.
This can be easily deduced from the fact that $X^{\Delta^1}$ represents the sheaf $(\yon X)^{\Delta{^1}}$
(Prop.\ \ref{prop:pobj_rep}). The iCFO structure on $\LnGpd{\infty}$ in particular is equipped with both functorial path objects, as well as functorial pullbacks of acyclic fibrations. Indeed, the pullbacks in this case are characterized in each simplicial dimension by the unique Banach manifold structure on the set-theoretic fiber product. (See, for example, Prop.\ 2.5 and Prop.\ 2.6 of \cite{Lang:95}.)   
Hence, for a small full subcategory of Lie $n$-groupoids closed under the iCFO structure, Thm.\ \ref{thm:simploc} would provide a convenient description of its simplicial localization. A potentially useful example of this sort, which will be studied in future work, is the category of $n$-groupoids internal to the category of separable Banach manifolds\footnote{We thank E.\ Getzler for this observation.}.

\subsection{Alternative characterization of weak
  equivalences}\label{sec:more-on-w-eq}
The incorporation of stalkwise weak equivalences into our iCFO structure for
$\gpd{\Cat,\covers}$ turns out to be quite convenient for some
applications. % For example, as we will see in Section
% \ref{sec:integration}, the integration of an $L_\infty$-quasi-isomorphism
% is a stalkwise weak equivalence of Lie $\infty$-groups.
However, in general, verifying directly that a morphism 
is a stalkwise weak equivalence could be cumbersome. 
Furthermore, we have the
aesthetically inelegant fact that the stalkwise weak equivalences are
the only piece of the iCFO structure on $\gpd{\Cat,\covers}$ which
requires us to leave the realm of simplicial objects in $\Cat$ for the
larger world of simplicial sheaves on $\Cat$.

Fortunately, as was mentioned in Sec.\ \ref{sec:iCFO}, the weak
equivalences in an iCFO are completely determined by the acyclic
fibrations.  This very useful fact is emphasized in the work of
Behrend and Getzler \cite{Behrend-Getzler:2015}
 on CFOs for higher geometric groupoids in descent categories.
What this implies in particular for the iCFO structure on $\gpd{\Cat,\covers}$,
is the following: If $f \maps X \to Y$ is a morphism in
$\gpd{\Cat,\covers}$, we consider the pullback diagram
\[
\begin{tikzpicture}[descr/.style={fill=white,inner sep=2.5pt},baseline=(current  bounding  box.center)]
\matrix (m) [matrix of math nodes, row sep=1em,column sep=3em,
  ampersand replacement=\&]
  {  
X \times_{Y} Y^{\Delta^1} \& Y^{\Delta^1} \\
X  \& Y  \\
}; 
  \path[->,font=\scriptsize] 
   (m-1-1) edge node[auto] {$\pr_2$} (m-1-2)
   (m-1-1) edge node[auto,swap] {$\pr_1$} (m-2-1)
   (m-1-2) edge node[auto] {$d_0$} (m-2-2)
   (m-2-1) edge node[auto] {$f$} (m-2-2)
  ;

%begin pullback symbol%
  \begin{scope}[shift=($(m-1-1)!.4!(m-2-2)$)]
  \draw +(-0.25,0) -- +(0,0)  -- +(0,0.25);
  \end{scope}
  %end pullback symbol%

\end{tikzpicture}
\]
Then it follows from Lemma \ref{lemma:fact} and Prop.\
\ref{prop:catfib_morita_eq} (and Thm.\ \ref{thm:icfo_struct}) that $f
\maps X \to Y$ is a stalkwise weak equivalence if and only if the composition
\[
p_f \maps X \times_{Y} Y^{\Delta^1} \xto{\pr_2} Y^{\Delta^1} \xto{d_1} Y
\]
is a hypercover.
Moreover, the path object construction can be avoided altogether, and
weak equivalences can be characterized directly in terms of $f$ and
covers between representable sheaves.  

To give just a simple example, denote by $\iota_{a} \maps \Delta^n \to \Delta^1 \times
\Delta^n$ for $a=0,1$ the inclusions $m \mapsto (a,m)$. Similarly, 
there are the inclusions $\partial \iota_a \maps 
\partial \Delta^n \to \Delta^1 \times \partial \Delta^n$. There is the pushout diagram
\[
\begin{tikzpicture}[descr/.style={fill=white,inner sep=2.5pt},baseline=(current  bounding  box.center)]
\matrix (m) [matrix of math nodes, row sep=2em,column sep=3em,
  ampersand replacement=\&]
  {  
\Horn{1}{1} \times \partial \Delta^n  \&  \Delta^1 \times \partial \Delta^n  \\
\Horn{1}{1} \times \Delta^n \&    \Delta^1 \times \partial \Delta^n
\cup \Horn{1}{1} \times \Delta^n \\
}; 
\path[->,font=\scriptsize] 
 (m-1-1) edge node[auto] {$$} (m-1-2)
 (m-1-1) edge node[auto,swap] {$$} (m-2-1)
 (m-1-2) edge node[auto] {$\jmath_1$} (m-2-2)
 (m-2-1) edge node[auto] {$\jmath_2$} (m-2-2)
;

%begin pushout symbol%
 \begin{scope}[shift=($(m-1-1)!.60!(m-2-2)$)]
 \draw +(.25,0) -- +(0,0)  -- +(0,.-.25);
 \end{scope}
 %end pushout symbol%
\end{tikzpicture}
\]
%The following is just a first step in the more detailed proof of
The following is parallel to the first step in the proof of 
Thm. 5.1 in  \cite{Behrend-Getzler:2015}:
\begin{proposition}\label{prop:w-eq}
A morphism $f \maps X \to Y$ in $\gpd{\Cat,\covers}$ is a stalkwise
weak equivalence if and only if, for all $n \geq 0$ the morphism in $\Cat$
\begin{equation} \label{eq:w-eq-first-step}
\Hom \bigl( \Delta^n \xto{i_1} \Delta^1 \times \Delta^n, X \xto{f} Y
\bigr) \to \Hom \bigl ( \partial \Delta^n  \xto{ \jmath_1
  \circ \partial \iota_1} \Delta^1 \times \partial \Delta^n
\cup \Horn{1}{1} \times \Delta^n, X \xto{f} Y \bigr)
\end{equation}
is a cover.
\end{proposition}
\begin{proof}
First, it follows from Lemma 2.4 in \cite{Zhu:2009a} that the sheaf
\[
\Hom(j,p_f):=\Hom \bigl (\partial \Delta^n \xto{j} \Delta^n, X \times_Y Y^{\Delta^1}
\xto{p_f} Y \bigr)
\]
is representable. From the discussion preceding the proposition, we
know $f \maps X \to Y$ is a weak equivalence if and only if for all $n
\geq 0$, the morphism in $\Cat$
\[
X_n \times_{Y_n} Y^{\Delta^1}_n \to \Hom(j,p_f)
\]
is a cover. Since $Y^{\Delta^1}_n \cong \Hom(\Delta^1 \times \Delta^n,
Y)$, there is the pullback square
\[
\begin{tikzpicture}[descr/.style={fill=white,inner sep=2.5pt},baseline=(current  bounding  box.center)]
\matrix (m) [matrix of math nodes, row sep=2em,column sep=3em,
  ampersand replacement=\&]
  {  
X_n \times_{Y_n} Y^{\Delta^1}_n   \&   \Hom(\Delta^1 \times \Delta^n, Y) \\
\Hom(\Delta^n, X) \& \Hom(\Delta^n, Y)\\
}; 
  \path[->,font=\scriptsize] 
   (m-1-1) edge node[auto] {$$} (m-1-2)
   (m-1-1) edge node[auto,swap] {$$} (m-2-1)
   (m-1-2) edge node[auto] {$\iota_{1}^\ast$} (m-2-2)
   (m-2-1) edge node[auto] {$f_\ast$} (m-2-2)
  ;

%begin pullback symbol%
  \begin{scope}[shift=($(m-1-1)!.4!(m-2-2)$)]
  \draw +(-0.25,0) -- +(0,0)  -- +(0,0.25);
  \end{scope}
  %end pullback symbol%

\end{tikzpicture}
\]
Hence,
\[
X_n \times_{Y_n} Y^{\Delta^1}_n \cong
\Hom \bigl( \Delta^n \xto{i_1} \Delta^1 \times \Delta^n, X \xto{f} Y
\bigr) 
\]

To complete the proof, we just need to show 
\[
\Hom(j,p_f) \cong \Hom \bigl ( \partial \Delta^n  \xto{ \jmath_1
  \circ \partial \iota_1} \Delta^1 \times \partial \Delta^n
\cup \Horn{1}{1} \times \Delta^n, X \xto{f} Y \bigr).
\]
This follows from pasting together the following pullback squares:
\[
\begin{tikzpicture}[descr/.style={fill=white,inner sep=2.5pt},baseline=(current  bounding  box.center)]
\matrix (m) [matrix of math nodes, row sep=1.5em,column sep=2em,
  ampersand replacement=\&]
  {  
\Hom(j,p_f) \& \Hom \bigl ( \Delta^1 \times \partial \Delta^n
\cup \Horn{1}{1} \times \Delta^n, Y \bigr) \&
\Hom(\Horn{1}{1} \times \Delta^n,Y) \\
\Hom(\partial \Delta^n, X \times_Y Y^{\Delta^1} ) \& \Hom(\Delta^1
\times \partial \Delta^n,Y) \& \Hom ( \Horn{1}{1} \times \partial
\Delta^n,Y) \\
\Hom(\partial \Delta^n, X) \& \Hom(\partial \Delta^n,Y)\\
}; 
  \path[->,font=\scriptsize] 
   (m-1-1) edge node[auto] {$$} (m-1-2)
   (m-1-2) edge node[auto] {$\jmath^{\ast}_2$} (m-1-3)
   (m-1-1) edge node[auto,swap] {$$} (m-2-1)
   (m-1-2) edge node[auto,swap] {$\jmath^{\ast}_1$} (m-2-2)
   (m-1-3) edge node[auto] {$(\id \times j)^\ast$} (m-2-3)
   (m-2-1) edge node[auto] {$\pr_{2 \ast}$} (m-2-2)
   (m-2-2) edge node[auto] {$$} (m-2-3)
   (m-2-1) edge node[auto,swap] {$\pr_{1 \ast}$} (m-3-1)
   (m-2-2) edge node[auto] {$\partial \iota^{\ast}_1$} (m-3-2)
   (m-3-1) edge node[auto] {$f_{\ast}$} (m-3-2)
  ;
% %begin pullback symbol%
   \begin{scope}[shift=($(m-1-1)!.4!(m-2-2)$)]
   \draw +(-0.25,0) -- +(0,0)  -- +(0,0.25);
   \end{scope}
%   %end pullback symbol%

%  %begin pullback symbol%
   \begin{scope}[shift=($(m-2-1)!.4!(m-3-2)$)]
   \draw +(-0.25,0) -- +(0,0)  -- +(0,0.25);
   \end{scope}
%   %end pullback symbol%

%  %begin pullback symbol%
   \begin{scope}[shift=($(m-1-2)!.5!(m-2-3)$)]
   \draw +(-0.25,0) -- +(0,0)  -- +(0,0.25);
   \end{scope}
%   %end pullback symbol%

\end{tikzpicture}
\]
\end{proof}

Theorem 5.1 in \cite{Behrend-Getzler:2015} further shows that if the category
of $n$-groupoids in $(\Cat,\covers)$ form a category of fibrant objects, then $f \maps X
\to Y$  is a weak equivalence if and only if the morphism
\begin{equation}\label{eq:w-eq}
\Hom(\Simp{n} \to \Simp{n+1}, X\to Y) \to \Hom(\partial \Simp{n}
\to  \Horn{n+1}{n+1}, X\to Y)
\end{equation}
is a cover for $n\ge 0$.
This turns out to be true in our iCFO case as well. 

\begin{cor}\label{cor:w-eq-combinatoric}
A morphism $f: X\to Y$ in $\LnGpd{\infty}$ is a stalkwise weak
equivalence if and only if the natural morphism \eqref{eq:w-eq}
is a cover for $n\ge 0$.
\end{cor}
\begin{proof}
By Proposition \ref{prop:w-eq}, we see that $f: X\to Y$ in $\LnGpd{\infty}$ is a stalkwise weak
equivalence if and only if the morphism
\eqref{eq:w-eq-first-step} is a cover for all $n\ge 0$. The 
morphisms \eqref{eq:w-eq-first-step} and \eqref{eq:w-eq} are 
exactly the morphisms (5.2) and (5.1), respectively, in \cite{Behrend-Getzler:2015}. 
The sources and targets for the morphisms in \cite{Behrend-Getzler:2015} are $n$-groupoids in
a descent category of spaces: a small category with finite limits, equipped with a subcategory of covers closed under pullback, which satisfy  a ``2 of 3'' property. (Axioms D1, D2, and D3, respectively in \cite{Behrend-Getzler:2015}.)
A category equipped with a locally stalkwise pretopology satisfies all of these axioms, except 
D1. Indeed, D2 is included in the definition of a pretopology, and D3 follows from 
Def.\ \ref{assump:cover}.
Even though D1 is not satisfied in this context, the proof of Thm. 5.1 in \cite{Behrend-Getzler:2015} 
still applies. A direct verification shows that all limits appearing in
the proof exist in $(\Cat,\covers)$. And clearly, the proof works for $n=\infty$. 
\end{proof}
\begin{remark}\label{rmk:BGThm}
We also mention that a characterization of weak equivalences similar to \eqref{eq:w-eq}
between Lie 2-groupoids can be deduced using properties of the join
construction of simplicial sets and the theory of Morita bibundles
developed in Li's Ph.D.\ thesis \cite{Li:2015}. This fact is generalized to
all Lie $n$-groupoids in \cite{blohmann-zhu}, which provides another interpretation of 
the combinatorial formula \eqref{eq:w-eq}.  
\end{remark}

\section{Lie $n$-algebras} \label{sec:LnA} 
In this section, we summarize the main results we will need from the companion paper \cite{Rogers:2018}  concerning the homotopy theory of finite type Lie $n$-algebras, and we refer the reader there for complete proofs and details. Throughout this section and the remainder of this paper, we adopt the notation and conventions from \cite[Sec.\ 2; Sec.\ 3]{Rogers:2018}.

\subsection{$L_\infty$-algebras and their morphisms}
We begin with a quick review of basic facts and definitions concerning $L_\infty$-algebras and their morphisms.  We follow the presentation in \cite[Sec.\ 3]{Rogers:2018} which is based on the
standard reference \cite{Lada-Markl:1995}.

Recall that an \textbf{$L_\infty$-algebra} $(L, \el)$
is a $\Z$-graded $\R$-vector space $L$ equipped with 
a collection $\el = \{\el_1, \el_2, \el_3, \ldots \}$
of graded skew-symmetric linear maps (or brackets)
%\begin{equation} \label{eq:brackets}
\[
\el_k \maps \Lambda^k L \to L, \quad 1 \leq k < \infty 
\]
%\end{equation}
with  $\deg{\el_k} = k-2$, satisfying an infinite sequence of Jacobi-like identities of the form:
\begin{align} \label{eq:Jacobi}
   \sum_{\substack{i+j = m+1, \\ \sigma \in \Sh(i,m-i)}}
  (-1)^{\sigma}\epsilon(\sigma)(-1)^{i(j-1)} l_{j}
   (l_{i}(x_{\sigma(1)}, \dots, x_{\sigma(i)}), x_{\sigma(i+1)},
   \ldots, x_{\sigma(m)})=0
\end{align}
for all $m \geq 1$. Above, the permutation $\sigma$ ranges over all $(i,m-i)$ unshuffles, 
and $\epsilon(\sigma)$ denotes the Koszul sign. In particular, Eq.\ \ref{eq:Jacobi} implies that $(L, \el_1)$ is a (homological) chain complex.

Equivalently, a $L_\infty$-structure on a graded vector space $L$ is a degree $-1$ codifferential $\delta$ on the reduced cocommutative coalgebra $\bar{S}(\bs L) = \bigoplus_{i \geq 1} S^{i}(\bs L)$. See \cite[Sec.\ 2.4; Sec.\ 3.1]{Rogers:2018} for further details. Here $\bs L$ denotes the suspension of the graded vector space $L$, i.e., $\bs L_{i}:=L[-1]_{i}=L_{i-1}$. % The correspondence between the ``structure maps'' of the codifferential $\delta$:
% \[
% \delta^1_m:= \pr_{\bs L} \circ \delta \vert_{\S^{m}(\bs L)} \maps \S^{m}(\bs L) \to \bs L
% \]
% and the brackets $\el_m$ is given by the formula
% %\begin{equation}\label{eq:struc_skew}
% \[
% \delta^1_{m} = (-1)^{\frac{m(m-1)}{2}} \bs \circ
% \el_{m} \circ {(\ds)}^{\tensor m}.
% \]
% %\end{equation}

A (weak) {\bf $L_\infty$-morphism} $f \maps (L,\el) \to (L',\el')$
is a collection $f=\{f_1,f_2,\ldots\}$
of graded skew-symmetric linear maps
$f_k \maps \Lambda^k L \to L' \quad 1 \leq k < \infty$
%\begin{equation} \label{eq:morphism1}
%\end{equation}
with $\ddeg{f_k} = k-1$, satisfying an infinite sequence of equations of the form:
\[
%\begin{equation}\label{eq:mor_formula}
\begin{split}
&\sum_{j+k=m+1} \sum_{\sigma} \pm f_j \bigl( l_k(x_\sigma(1), \ldots, 
x_{\sigma(k)}), x_{\sigma(k+1)},\ldots x_{\sigma(m)} \bigr)  
  \\
 & + \sum_{\substack{ 1 \leq t \leq m \\ i_{1} + \cdots i_{t} = m}}
\sum_{\tau} \pm l'_{t} \bigl( f_{i_1}(x_{\tau(1)},\ldots,x_{\tau(i_1)}),
f_{i_2}(x_{\tau(i_1 + 1 )},\ldots,x_{\tau(i_1 + i_2)}), \\
& \quad \dots, f_{i_t}(x_{\tau(i_1 + \cdots + i_{t-1} +1 )},\ldots,x_{\tau(m)})\bigr) =0.
\end{split}
%\end{equation}
\]
Above  $\sigma$ ranges over all $(k,m-k)$ unshuffles, and $\tau$ ranges through
certain $(i_1,\ldots,i_t)$ unshuffles. (See, for example, \cite[Def.\ 2.3]{Allocca:2014}.)
% A {\bf morphism} of $L_\infty$-algebras 
% \begin{equation} \label{eq:morphism0}
% f=\{f_{k\geq 1} \} \maps
% (L,l_k) \to
% (L',l'_k)
% \end{equation}
% is a collection of graded skew-symmetric $k$-linear maps of degree $k
% -1$
% \begin{equation} \label{eq:morphism1}
% f_k \maps \Lambda^k L \to L' \quad 1 \leq k < \infty
% \end{equation}
%  (See \cite[Def.\ 2.3]{Allocca:2014} for 
% further details and the precise signs.) 
More conveniently, a morphism between $L_\infty$-algebras $L$ and $L'$ is equivalently a degree 0 morphism of dg coalgebras $F \maps \bigl(\S(\bs L), \delta \bigr) \to \bigl(\S(\bs L'), \delta' \bigr)$. 
% The correspondence between the ``structure maps'' of the coalgebra morphism:
% \[
% F^1_m:= \pr_{\bs L'} \circ F \vert_{\S^{m}(\bs L)} \maps \S^{m}(\bs L) \to \bs L'
% \]
% and the skew-symmetric maps $f_m$ is given by the formula
% \[
%  F^{1}_{k} = (-1)^{\frac{k(k-1)}{2}} \bs \circ
%  f_{k} \circ {(\ds)}^{\tensor k}.
% \]
% The morphism $F$ is compatible with the codifferentials if and only if
% for all $m\geq 1$ and for all $x_1,\ldots,x_m \in L$:
% %\begin{equation}\label{eq:dgmap}
% \[
% \begin{split}
% \sum^m_{k=1} \delta'^1_k F^k_m(\bs x_1,\ldots, \bs x_m) =    
% \sum^m_{k=1} F^1_k\delta^k_m (\bs x_1,\ldots, \bs x_m),
% \end{split}
% \]
% %\end{equation} 
% where $F^k_{m} := \pr_{\S^{k}(\bs L')} \circ F \vert_{\S^{m}(\bs L)}$ and
% $\delta^k_{m} := \pr_{\S^{k}(\bs L)} \circ \delta \vert_{\S^{m}(\bs L)}$. We recall that the linear maps 
% $\delta^k_{m}$,  $\delta'^k_{m}$, and $F^{k}_{m}$ are uniquely determined by the structure maps 
% $\delta^1_{n}$,  $\delta'^1_{n}$, and $F^{1}_{n}$, respectively. 
See again \cite[Sec.\ 2.4; Sec.\ 3.1]{Rogers:2018} for further details. In particular, treating $L_\infty$ morphisms as dg coalgebra morphisms gives us a clear way to compose them
\cite[Eq.\ 2.6]{Rogers:2018}. It is typical to consider the category $\Linf$ of $L_\infty$-algebras and $L_\infty$-morphisms as a full subcategory of the category of dg cocommutative coalgebras.

\begin{remark} \label{rmk:h0}
\mbox{}
\begin{itemize}
\item[-] As in \cite{Rogers:2018}, we will write morphisms in $\Linf$ using a single lower-case letter, e.g. 
\[
f \maps (L,\el) \to (L',\el'),
\]
and the $k$-ary map in the collection $f$ will always be denoted by $f_k$. % The dg coalgebra morphism encoded by the collection $f$ will always be written using the corresponding upper-case letter, e.g.  
% $F \maps (\S(\bs L),\delta) \to (\S(\bs L'),\delta')$.

\item[-] Recall that if $f \maps (L,\el) \to (L',\el')$ is a $L_\infty$-morphism, % then it follows from the definitions that 
% the degree 0 map $f_1 \maps (L,\el_1) \to (L',\el_1')$ is a chain map.
% Furthermore, for all $x,y \in L$, we have
% \begin{equation} \label{eq:morphism2.5}
% f_1 \bigl ( \el_2(x,y) \bigr ) -  \el'_{2} \bigl( f_1(x),f_1(y) \bigr)
%  = \el'_1 f_2(x,y).
% \end{equation}
then 
\[
H(f_{1}) \maps \bigl( H_0(L),[\cdot,\cdot]) \to (H_0(L'),[\cdot,\cdot]' \bigr)
\]
 is a morphism of Lie algebras,  where the above Lie brackets are
induced by the bilinear brackets $\el_2$ and $\el'_2$, respectively. % on any $L_\infty$-algebra $(L,\el)$
% induces a Lie algebra structure  on $H_0(L)$,  it follows from 
% Eq.\ \ref{eq:morphism2.5} that 
 \end{itemize}
\end{remark}

Next we recall several important classes of $L_\infty$-morphisms:

\begin{definition} \label{def:quasi-iso}
Let $f \maps (L,\el) \to (L',\el')$ be a morphism of $L_\infty$-algebras. 
\begin{enumerate}
\item We say $f$ is a \textbf{$L_\infty$-isomorphism} iff the linear map $f_1 \maps L \to L'$ is an isomorphism of graded vector spaces.

\item We say $f$ is a \textbf{$L_\infty$-quasi-isomorphism} iff the chain map
$f_1$  is a quasi-isomorphism, i.e.\ the induced map on homology
$H(f_1) \maps H(L) \to H(L')$ is an isomorphism of graded vector spaces.

\item We say $f$ is a \textbf{strict $L_\infty$-morphism}  
iff $f_k =0$ for all $k \geq 2$. In this case we write $f=f_1 \maps (L,\el) \to (L',\el)$ and it follows that every $k$-ary bracket $\el_k$ is preserved by the chain map $f_1$:
\[
\el'_k \circ f_1^{\tensor k} = f_1 \circ \el_k \quad \text{for all $k \geq 1$}.
\]

\end{enumerate}
\end{definition}

\subsection{Finite type Lie $n$-algebras} \label{sec:lnaft}
Let $n \in \N \cup \{\infty\}$. We recall that a 
$L_\infty$-algebra $(L,\el_k)$ is a \textbf{Lie $n$-algebra}
iff the graded vector space $L$ is concentrated in the first $n-1$ non-negative degrees, i.e.\
$L= \bigoplus_{i \geq 0}^{n-1} L_i$. 
The standard reference for Lie $n$-algebras is \cite[Def.\ 4.3.2]{Baez-Crans:2004}.
See \cite[Sec.\ 3.2]{Rogers:2018} for a list of relevant examples.
For a fixed $n \in \N \cup \{\infty\}$, we denote by $\LnA{n}$ the full subcategory of $\Linf$ 
whose objects are Lie $n$-algebras.

If $(L,\el)$ is a Lie $n$-algebra and each $L_i$ is a finite-dimensional vector space, then we say $(L,\el)$ is a \textbf{finite type Lie $n$-algebra}, and  we denote by $\lnaft$ the
category whose objects are finite type Lie $n$-algebras, and whose morphisms are (weak) $L_\infty$-morphisms.

\subsubsection{The Chevalley--Eilenberg algebra} \label{sec:CE_alg}
Throughout the companion paper \cite{Rogers:2018}, the category $\lnaft$ is taken to be a full subcategory of the category of conilpotent dg cocommutative coalgebras. % This  
% perspective is convenient since the approach there involves adopting certain basic ideas from deformation theory.
For the purposes of the present paper, in order to make contact with Henriques' work \cite{Henriques:2008}, we now recall the dual picture.  Let $(L, \el) \in \lnaft$ be a finite type Lie $n$-algebra and $(\S(\bs L),\delta)$ the associated dg coalgebra. We adjoin the ground field to obtain the corresponding coaugmented counital dg cocommutative coalgebra $(S(\bs L),\delta)$, where $S(\bs L)=\R \oplus \S(\bs L)$ is the usual symmetric algebra and $\delta(1)=0$. (See \cite[Sec.\ 2.4]{Rogers:2018}.) 

We denote by 
\[
\CE(L):= \bigl(S(\bs L^\vee), \delta_{\CE} \bigr)
\]
the \textbf{Chevalley--Eilenberg algebra} of $(L,\el)$ i.e.\ the $\R$-linear dual of $(S(\bs L),\delta)$. It is naturally a unital commutative dg-algebra (cdga), whose differential $\delta_{\CE}:=\delta^{\vee}$ has degree $1$. Since $(L,\el)$ is finite type, $\CE(L)$ is semi-free. That is, its
underlying commutative graded algebra is freely generated by the non-negatively graded vector space $\bs L^\vee$, where $L^\vee_i:=\hom_{\R}(L_i,\R)$.

\subsubsection{$\lnaft$ as a category of fibrant objects} \label{sec:lnaft_cfo}
As previously mentioned in Sec.\ \ref{sec:iCFO} the only difference between our definition of an iCFO \eqref{def:catfibobj} and
Brown's original definition of a category of fibrant objects is axiom 4. Brown requires the strong axiom that the pullback of a fibration always exists.

We summarize the  main result of \cite{Rogers:2018} in the following theorem:

\begin{theorem}[Thm.\ 5.1 \cite{Rogers:2018}] \label{thm:lnaft_cfo}
Let $n \in \N \cup \{\infty\}$.  The category $\lnaft$ of finite type Lie $n$-algebras 
and weak $L_\infty$-morphisms has the structure of a category of fibrant objects, in which a morphism
$f \maps (L,\el) \to (L',\el)$ is:
\begin{itemize}
\item a weak equivalence iff it is a $L_\infty$-quasi-isomorphism (Def.\ \ref{def:quasi-iso}),

\item a fibration iff the chain map $f_1 \maps (L,\el_1) \to (L',\el'_1)$ is surjective in all \underline{positive} degrees,

\item an acyclic fibration iff $f$ is a $L_\infty$-quasi-isomorphism and
the chain map $f_1 \maps (L,\el_1) \to (L',\el'_1)$ is surjective in all degrees.
\end{itemize}
\end{theorem}

% The above characterization  acyclic fibrations is given merely for convenience.
% It follows immediately from the fact that for any Lie $n$-algebra $(L,\el)$, we have $H_{0}(L)=L_0/ \im \el_1$.
We note that an explicit construction for path objects in $\lnaft$ is provided in \cite[Sec 3.3]{Rogers:2018}.

\begin{remark} \label{rmk:lna_fibs}
Fibrations between Lie $n$-algebras, as defined in Thm.\ \ref{thm:lnaft_cfo}, coincide with fibrations in the projective model structure on non-negatively graded chain complexes. In contrast, the chain map $f_1$ associated to a fibration between unbounded $\Z$-graded $L_\infty$-algebras is required to be surjective in all degrees. (See for example \cite{Ezra-infty}, \cite{Rogers:2017}, and \cite{Vallette:2014}.)

In \cite{Severa:2007}, \v{S}evera constructed a functor which provides a differentiation procedure $\mathrm{Diff} \maps \LnG{n}^{\ft} \to \lnaft$ from finite-dimensional Lie $n$-groups to Lie $n$-algebras. In certain cases, it is clear how this functor interacts with the iCFO structure (Thm.\ \ref{thm:icfo_struct}) on the category  $\LnGpd{n}$,  
and this provides us with evidence that the notion of fibration given in Thm.\ \ref{thm:lnaft_cfo}
is the ``correct'' one for our applications. For example, suppose $\g_\bullet$ is a simplicial Lie algebra, and let $G_\bullet$  
denote the (level-wise) 1-connected simplicial Lie group integrating $\g_\bullet$.
Then its classifying space $\overline{W}G_\bullet$ is a Lie $\infty$-group.
Jur\v{c}o showed in \cite{Jurco:2012} that $\mathrm{Diff}(\bar{W}G_\bullet)$ is isomorphic to $N\g_\bullet$, the dg Lie algebra obtained from $\g_\bullet$ via Quillen's normalized chains functor \cite{Quillen:1969}.  Now suppose $f \maps \overline{W}G_\bullet \to \overline{W}G'_\bullet$ is a Kan fibration between classifying spaces of 1-connected simplicial Lie groups in $\LnG{\infty}$. Then, by using Jur\v{c}o's result, it is not difficult to see that $\mathrm{Diff}(f) \maps N\g_\bullet \to N\g'_\bullet$  is a dgla morphism that is surjective in all positive degrees, but not surjective, in general, in degree 0. Therefore, if we hope to prove that $\mathrm{Diff}$ preserves fibrations, then 
the definition of fibration in the CFO structure on $\lnaft$ must be the one given in Thm.\ \ref{thm:lnaft_cfo}.
A more detailed analysis of \v{S}evera's functor is the subject of our future work.
\end{remark}

\subsection{Fibrations of Lie $n$-algebras}

The following lemma from \cite{Rogers:2018} is based on a result of Vallette \cite{Vallette:2014} concerning the factorization of epimorphisms between $\Z$-graded homotopy algebras.
The lemma allows us to dramatically simplify many constructions involving fibrations between Lie $n$-algebras. Roughly, it says that every fibration in $\lnaft$ is a strict fibration up to isomorphism. 

\begin{lemma}[Lemma 3.11 \cite{Rogers:2018}] \label{lem:strict_fib}
Let $f \maps (L,\el) \to (L',\el')$ be a fibration between Lie $n$-algebras. 
Then there exists a Lie $n$-algebra $(L,\tilde{\el})$ and an 
isomorphism $\phi \maps (L,\tilde{\el}) \xto{\cong} (L,\el)$ such that
\[
f\phi \maps (L,\tilde{\el}) \to (L',\el) 
\]
is a strict fibration with  $f\phi=(f\phi)_1 = f_1$.
\end{lemma}

It turns out that not every fibration in $\lnaft$ integrates to a fibration in $\LnG{\infty}$ (see
Remark \ref{rmk:fib_not_int}). However, we will show in Thm.\ \ref{thm:int_split} 
that there is a distinguished class of fibrations in $\lnaft$ that do. We call these ``quasi-split fibrations''.
\begin{definition}\label{def:split_fib}
A fibration of Lie $n$-algebras $f \maps (L,\el) \to (L',\el')$ is a \textbf{quasi-split fibration}
iff 
\begin{enumerate}

\item the induced map in homology $H(f_1) \maps H(L) \to H(L')$ is surjective in all degrees and,

\item $H_0(L) \cong \ker H_0(f_1) \oplus H_0(L')$ in the category of Lie algebras.
% in degree zero, the map $H(f_1) \maps (H_0(L), [\cdot,\cdot])  \to (H_0(L'), [\cdot,\cdot]^\prime)$
% is a split epimorphism of Lie algebras.
\end{enumerate}
\end{definition}

% As previously mentioned, the homology $H(L)$ of a Lie $n$-algebra is naturally a graded module over the Lie algebra $H_0(L)$. 

\begin{remark}
Note that every acyclic fibration in $\lnaft$ is a quasi-split fibration. More generally, $f \maps (L,\el) \to (L',\el')$ is a quasi-split fibration if $\ker H(f_1)$ is central and $H(f_1) \maps H(L) \to H(L')$ is a split epimorphism in the category of $H_0(L)$-modules.
\end{remark}

Besides acyclic fibrations, there are other examples of quasi-split fibrations which naturally arise in interesting applications. In particular, the string Lie 2-algebra, whose integration was the original motivation for \cite{Henriques:2008}, is  a special case of the following construction.  

\begin{example}[Central $n$-extensions] \label{ex:string}
Let $(\g,[\cdot,\cdot])$ be a Lie algebra and $c \maps \Lambda^{n+1}
\g \to \R$ a degree $n+1$ cocycle in the Chevalley-Eilenberg complex
associated to $\g$. From this data we obtain a Lie $n$-algebra $\widehat{\g}_{c}$
whose underlying vector space is concentrated in degrees $0$ and $n-1$:
\[
\widehat{\g}_{c}= \g \oplus \R[1-n]
\]
and whose only non-trivial brackets are
\[
\begin{split}
\el_{2}(x_{1},x_{2}) &= [x_{1},x_{2}],  \quad \text{if $x_{1},x_{2} \in \g$}\\
\el_{n+1}(x_{1},\hdots,x_{n+1}) &= c(x_{1},\hdots,x_{n+1}),  \quad \text{if $x_{1},\hdots,x_{n+1} \in \g$}\\
\end{split}
\]
A straightforward verification shows that the linear projection $\pi \maps \widehat{\g}_{c} \to \g$
extends to a strict quasi-split fibration sequence
$\R[1-n] \emb \widehat{\g}_{c} \xto{\pi} \g$ in $\lnaft$.
\end{example}

We conclude this section with a corollary concerning the ``strictification'' of quasi-split fibrations. The proof follows directly from Lemma \ref{lem:strict_fib} and Def.\ \ref{def:split_fib}.

\begin{cor} \label{cor:strict_fib}
Let $f \maps (L,\el) \to (L',\el')$ be a quasi-split fibration between Lie $n$-algebras. 
Then there exists an 
isomorphism $\phi \maps (L,\tilde{\el}) \xto{\cong} (L,\el)$ in $\lnaft$ such that
$f\phi \maps (L,\tilde{\el}) \to (L',\el)$ is a strict quasi-split fibration with  $f\phi=(f\phi)_1 = f_1$.
\end{cor}

\subsection{Postnikov tower for Lie $n$-algebras} \label{sec:postnikov}
Next, we recall the following technical results from \cite[Sec.\ 7]{Rogers:2018} concerning Postnikov towers.
These will play a crucial role in our analysis of the integration functor in Sec.\ \ref{sec:int_exact}.

Let $(L,\el)$ be a Lie $n$-algebra. Following \cite[Def.\ 5.6]{Henriques:2008},
we consider two different truncations
of the underlying chain complex $(L,d=\el_1)$. 
For any $m \geq 0$, denote by $\tau_{\leq m}L$ and $\tau_{< m}L$ the following $(m+1)$-term complexes:
\[
%\begin{equation} \label{eq:trunc_defs}
\begin{split}
(\tau_{\leq m} L)_i=
\begin{cases}
L_i & \text{if $i < m$,}\\
\coker(d_{m+1}) & \text{if $i=m$,}\\
0 & \text{if $i >m$,}
\end{cases}
\qquad 
(\tau_{< m} L)_i=
\begin{cases}
L_i & \text{if $i < m$,}\\
\im (d_{m}) & \text{if $i=m$,}\\
0 & \text{if $i >m$.}
\end{cases}
\end{split}
%\end{equation}
\]
In degree $m$, the differentials for $\tau_{\leq m}L$ and $\tau_{<  m}L$ are $d_{m} \maps L_m/\im(d_{m+1}) \to L_{m-1}$, and the inclusion $ \im(d_m) \emb L_{m-1}$, respectively. The homology complexes of
$\tau_{\leq m}L$ and $\tau_{<  m}L$ are
\[
H_{i}(\tau_{\leq m}L) = 
\begin{cases}
H_i(L) & \text{if $i \leq m$,}\\
0 & \text{if $i>m$,}
\end{cases}
\qquad
H_{i}(\tau_{< m}L) = 
\begin{cases}
H_i(L) & \text{if $i <  m$,}\\
0 & \text{if $i \geq m$.}
\end{cases}
\]
We have the following obvious surjective chain maps
\begin{equation} \label{eq:projs}
\begin{split}
p_{\leq m} \maps L \to \tau_{\leq m}L \qquad p_{< m} \maps L \to \tau_{< m}L
\end{split}
\end{equation}
where in degree $m$, the map $p_{\leq m}$ is the surjection $L_m \to \coker(d_{m+1})$, and
$p_{< m}$ is the differential $d_{m} \maps L_m \to \im(d_{m})$. There are also the similarly defined
surjective chain maps
\begin{equation} \label{eq:projs1}
\begin{split}
q_{\leq m} \maps \tau_{\leq m}L \to \tau_{< m}L, \quad  q_{< m+1} \maps \tau_{< m +1}L \xto{\sim} \tau_{\leq m}L. 
\end{split}
\end{equation}
The map $q_{\leq m}$ in degree $m$ is the differential $d_m \maps \coker{d_{m+1}} \to \im d_m$, and the identity in all other degrees. The map $q_{<m+1}$ is the projection $L_m \to \coker d_{m+1}$ in degree $m$, the identity in all degrees $<m$, and the zero map in degree $m+1$.  
We note that $q_{<m+1}$  is a quasi-isomorphism of complexes.

\begin{proposition}[Prop.\ 7.2 \cite{Rogers:2018}] \label{prop:lna_tower}
Let $(L,\el)$ be a Lie $n$-algebra. 
\begin{enumerate}

\item The Lie $n$-algebra structure on $(L,\el)$ induces Lie $(m+1)$-structures on the complexes
$\tau_{\leq m} L$ and $\tau_{<m} L$ whose brackets are given by
\[
\tau_{\leq m} \el_{k}(\bar{x}_1,\ldots,\bar{x}_k):= p_{\leq m}\el_k(x_1,\ldots,x_k), \quad
\tau_{< m} \el_{k}(\bar{y}_1,\ldots,\bar{y}_k):= p_{<m}\el_k(y_1,\ldots,y_k),
\] 
where $\bar{x}_i= p_{\leq m}(x_i)$ and $\bar{y}_i= p_{< m}(y_i)$.

\item The assignments $(L,\el) \mapsto (\tau_{\leq m} L,\tau_{\leq m}\el)$ and
$(L,\el) \mapsto (\tau_{< m} L,\tau_{< m}\el)$ are functorial.

\item An $L_\infty$-morphism $f \maps (L,\el) \to (L',\el')$ induces a morphism of towers of Lie $n$-algebras
\begin{equation} \label{diag:tower}
\begin{tikzpicture}[descr/.style={fill=white,inner sep=2.5pt},baseline=(current  bounding  box.center)]
\matrix (m) [matrix of math nodes, row sep=2em,column sep=1.4em,
  ampersand replacement=\&]
  {  
\cdots \tau_{\leq m-1} L \& \tau_{< m-1} L \& \tau_{\leq m-2}L \& ~ \cdots ~ \& 
\tau_{\leq 1 } L \& \tau_{< 1} L \& \tau_{\leq 0} L\\
\cdots \tau_{\leq m-1} L' \& \tau_{< m-1} L' \& \tau_{\leq m-2}L' \& ~ \cdots ~ \& 
\tau_{\leq 1 } L' \& \tau_{< 1} L' \& \tau_{\leq 0} L'\\
};
\path[->,font=\scriptsize] 
(m-1-1) edge node[auto] {$q_{\leq m-1}$} (m-1-2)
(m-1-2) edge node[auto] {$q_{<m-1}$} (m-1-3)
(m-1-3) edge node[auto] {$q_{ \leq m-2}$} (m-1-4)
(m-1-4) edge node[auto] {$$} (m-1-5)
(m-1-5) edge node[auto] {$q_{\leq 1}$} (m-1-6)
(m-1-6) edge node[auto] {$q_{< 1}$} (m-1-7)
(m-2-1) edge node[auto,swap] {$q'_{\leq m-1}$} (m-2-2)
(m-2-2) edge node[auto,swap] {$q'_{<m-1}$} (m-2-3)
(m-2-3) edge node[auto,swap] {$q'_{ \leq m-2}$} (m-2-4)
(m-2-4) edge node[auto] {$$} (m-2-5)
(m-2-5) edge node[auto,swap] {$q'_{\leq 1}$} (m-2-6)
(m-2-6) edge node[auto,swap] {$q'_{< 1}$} (m-2-7)

(m-1-1) edge node[auto,swap] {$$} (m-2-1)
(m-1-2) edge node[auto,swap] {$\tau_{< m-1}f$} (m-2-2)
(m-1-3) edge node[auto,swap] {$\tau_{\leq m-2}f$} (m-2-3)
(m-1-5) edge node[auto,swap] {$\tau_{\leq 1}f$} (m-2-5)
(m-1-6) edge node[auto,swap] {$\tau_{< 1}f$} (m-2-6)
(m-1-7) edge node[auto,swap] {$\tau_{\leq 0}f$} (m-2-7)
;
\end{tikzpicture}
\end{equation}
in which the horizontal arrows are the strict $L_\infty$-morphisms induced by the surjective chain maps \eqref{eq:projs1}.

\end{enumerate}
\end{proposition}

\begin{remark} \label{rmk:tau0}
The vertical maps $\tau_{< m}f$ and $\tau_{\leq m} f$ in \eqref{diag:tower} induced by the morphism $f \maps (L,\el) \to (L',\el')$ are defined via the projections given in Eq.\ \ref{eq:projs}. For example, 
\[
\tau_{\leq m} f_k(\bar{x}_1,\ldots,\bar{x}_k):= p'_{\leq m}f_k(x_1,\ldots,x_k).
\]
%See the proof of Prop.\ 7.2 in \cite{Rogers:2018} for more details.

We also note that the Lie $n$-algebra $\tau_{\leq 0} L$ is just the Lie algebra $H_{0}(L)$ concentrated in degree zero. Given a morphism of Lie $n$-algebras 
$f \maps (L, \el) \to (L',\el')$, the induced morphism $\tau_{\leq 0} f \maps
H_{0}(L) \to H_{0}(L')$ of Lie algebras is the morphism $H_0(f_1)$ from Remark \ref{rmk:h0}.
\end{remark}

\subsection{Quasi-split fibrations and decomposition of  towers} 
The commutative diagram \eqref{diag:tower} has a convenient decomposition in the case when $f \maps (L,\el) \to (L',\el')$ is a quasi-split fibration \eqref{def:split_fib}. In these next two results recalled from Sec.\ 7.2 of \cite{Rogers:2018}, we only consider strict quasi-split fibrations. The general case involving arbitrary quasi-split fibrations  follows from applying Cor.\ \ref{cor:strict_fib}.

Let $\ker q_{< m+1}$ denote the kernel of the map 
$q_{<m+1} \maps (\tlt{m+1} L, \tlt{m+1} \el) \to (\tleq{m} L, \tleq{m} \el)$ defined in \eqref{eq:projs1}  
%strict acyclic fibration 
Then the chain complex $\ker q_{< m+1}$ is an abelian Lie $n$-algebra
concentrated in degrees $m$ and $m+1$ with
\begin{equation} \label{eq:ker-qm+1}
(\ker q_{< m+1})_m=\im d_{m+1}, \quad  (\ker q_{< m+1})_{m+1} = \im d_{m+1}[-1] 
\end{equation}
The induced differential $\el^{\ker}=\el_1$ on $\ker q_{< m+1}$ is simply the desuspension isomorphism.

\begin{proposition}[Prop.\ 7.5 \cite{Rogers:2018}]\label{prop:tower_decomp1}
Let $f=f_1 \maps (L, \el) \to (L',\el')$ be a strict quasi-split fibration.
Then there exists morphisms in $\lnaft$
\[
r \maps (\tau_{< m +1}L, \tlt{m+1}\el) \to (\ker q_{<m+1},\el^{\ker}), \quad 
r' \maps (\tau_{< m +1}L', \tlt{m+1}\el') \to (\ker q'_{<m+1},\el^{\prime \ker})
\] 
inducing $L_\infty$-isomorphisms
\[
\begin{split}
\bigl(q_{< m+1},r \bigr) &\maps \bigl (\tau_{< m +1}L, \tlt{m+1}\el \bigr) \xto{\cong} \bigl( \tau_{\leq m} L \oplus \ker q_{<m+1}, \tleq{m}\el \oplus \el^{\ker} \bigr)    \\
\bigl(q'_{< m+1},r' \bigr)&\maps \bigl( \tau_{< m +1}L', \tlt{m+1}\el' \bigr) \xto{\cong} \bigl( \tau_{\leq m} L' \oplus \ker q'_{<m+1}, \tleq{m}\el' \oplus \el^{\prime \ker} \bigr)
\end{split}
\]
such that the following diagram commutes in $\LnA{n}$:
\begin{equation} \label{diag:tower_decomp1.1}
\begin{tikzpicture}[descr/.style={fill=white,inner sep=2.5pt},baseline=(current  bounding  box.center)]
\matrix (m) [matrix of math nodes, row sep=2em,column sep=5em,
  ampersand replacement=\&]
  {  
\tau_{< m +1 }L \& \tau_{\leq m }L \oplus \ker q_{<m+1}\\
\tau_{< m +1 }L' \& \tau_{\leq m }L' \oplus \ker q'_{<m+1}\\
}; 
\path[->,font=\scriptsize] 
(m-1-1) edge node[auto]{$\bigl(q_{< m+1},r \bigr)$} node[auto,below] {$\cong$} (m-1-2)
(m-1-1) edge node[auto,swap] {$\tau_{< m +1 } f$} (m-2-1)
(m-1-2) edge node[auto] {$\tau_{\leq m} f \oplus \tau_{< m +1 }f \vert_{\ker}$} (m-2-2)
(m-2-1) edge node[auto]{$\bigl(q'_{< m+1},r' \bigr)$} node[below] {$\cong$} (m-2-2)
;
\end{tikzpicture}
\end{equation}
\end{proposition}

Now let $m\geq 1$. % We focus on those commuting squares in \eqref{diag:tower} whose top edges consist of the strict quasi-split fibrations  $q_{\leq m} \maps (\tleq{m} L, \tleq{m} \el) \to (\tlt{m} L, \tlt{m} \el)$ defined \eqref{eq:projs1}. 
The chain map $q_{\leq m}$ defined in \eqref{eq:projs1} induces a short exact sequence of chain complexes
\begin{equation} \label{eq:ses1}
H_m \xto{i} \tleq{m}L \xto{q_{\leq m}}  \tlt{m}L 
\end{equation}
where $H_m$ is the homology group $H_m(L)$ concentrated in degree $m$ with trivial differential.
The second decomposition result that we will need is:

\begin{proposition}[Prop.\ 7.6 \cite{Rogers:2018}] \label{prop:tower_decomp2}
Let $m \geq 1$ and let $f=f_1 \maps (L, \el) \to (L',\el')$ be a strict quasi-split fibration. % such that
% induced map in homology
% \[
% H(f_1) \maps H_{m} \to H'_m
% \] 
% is surjective in degree $m$.
Then there exists $L_\infty$-structures $\hat{\el}$ and $\hat{\el}'$ on the graded vector spaces $\tlt{m} L \oplus H_m$ and $\tlt{m} L' \oplus H'_m$, respectively,  and strict $L_\infty$-isomorphisms
\[
\begin{split}
\hat{q} \maps (\tleq{m} L, \tleq{m}\el) &\xto{\cong} \bigl( \tlt{m}L \oplus H_m, \hat{\el} ~ \bigr)\\
\hat{q}' \maps (\tleq{m} L', \tleq{m}\el') & \xto{\cong} \bigl( \tlt{m}L' \oplus H'_m, \hat{\el'} ~ \bigr)
\end{split}
\]
such that the following diagram of $L_\infty$-morphisms commutes
\begin{equation}
\label{diag:tower_decomp2.1}
\begin{tikzpicture}[descr/.style={fill=white,inner sep=2.5pt},baseline=(current  bounding  box.center)]
\matrix (m) [matrix of math nodes, row sep=2em,column sep=2em,
  ampersand replacement=\&]
  {  
( \tau_{\leq m }L, \tleq{m} \el) \& (\tau_{\lt m }L \oplus H_m, \hat{\el})\\
(\tau_{\leq m }L', \tleq{m} \el') \& (\tau_{\lt m }L' \oplus H_m', \hat{\el}^{\prime})\\
}; 
\path[->,font=\scriptsize] 
(m-1-1) edge node[auto]{$\hat{q}$} node[auto,below] {$\cong$} (m-1-2)
(m-1-1) edge node[auto,swap] {$\tleq{m} f$} (m-2-1)
(m-1-2) edge node[auto] {$\tlt{m} f \oplus H(f)$} (m-2-2)
(m-2-1) edge node[auto]{$\hat{q}'$} node[below] {$\cong$} (m-2-2)
;
\end{tikzpicture}
% \begin{tikzpicture}[descr/.style={fill=white,inner sep=2.5pt},baseline=(current  bounding  box.center)]
% \matrix (m) [matrix of math nodes, row sep=2em,column sep=5em,
%   ampersand replacement=\&]
%   {  
% \tau_{\leq m }L \& \tau_{\lt m }L \oplus H_m\\
% \tau_{\leq m }L' \& \tau_{\lt m }L' \oplus H_m'\\
% }; 
% \path[->,font=\scriptsize] 
% (m-1-1) edge node[auto]{$\hat{q}$} node[auto,below] {$\cong$} (m-1-2)
% (m-1-1) edge node[auto,swap] {$\tleq{m} f$} (m-2-1)
% (m-1-2) edge node[auto] {$\tlt{m} f \oplus H(f)$} (m-2-2)
% (m-2-1) edge node[auto]{$\hat{q}'$} node[below] {$\cong$} (m-2-2)
% ;
% \end{tikzpicture}
\end{equation}
\end{proposition}

\begin{remark} \label{rmk:tower_decomp1}
In order to prove Thm.\ \ref{thm:int_split} in the next section, we'll need to recall from the proof of Prop.\ 7.6 in \cite{Rogers:2018} 
some details concerning the isomorphisms $\hat{q}$ and $\hat{q}'$. In degree $m$, the short exact sequence \eqref{eq:ses1} and the analogous sequence for $q'_{\leq m}$ gives the following commutative diagram of vector spaces.
\[
\begin{tikzpicture}[descr/.style={fill=white,inner sep=2.5pt},baseline=(current  bounding  box.center)]
\matrix (m) [matrix of math nodes, row sep=2em,column sep=2em,
  ampersand replacement=\&]
  {  
H_m \& \coker d_{m+1} \& \im d_m[-1] \\
H'_m \& \coker d'_{m+1} \& \im d'_m[-1] \\
}; 
\path[->,font=\scriptsize] 
 (m-1-1) edge node[auto] {$i$} (m-1-2)
 (m-2-1) edge node[auto] {$i'$} (m-2-2)
;
\path[->,font=\scriptsize] 
 (m-1-2) edge node[auto] {$d_m$} (m-1-3)
 (m-2-2) edge node[auto] {$d'_m$} (m-2-3)
 ;
\path[->,font=\scriptsize] 
 (m-1-1) edge node[auto,swap] {$H(f_1)$} (m-2-1)
 (m-1-2) edge node[auto] {$\tleq{m}f_1$} (m-2-2)
 (m-1-3) edge node[auto] {$\tlt{m}f_1$} (m-2-3)
 ;
\end{tikzpicture}
\]
Since $f$ is a strict quasi-split fibration, there exists 
sections $s \maps \im d_m[-1] \to \coker d_{m+1}$, and $s' \maps \im d'_{m}[-1] \to \coker d'_{m+1}$, of $d_m$ and $d'_m$, respectively, such that $\tleq{m }f_1 \circ s =s' \circ \tlt{m}f_1$.
Let $t \maps \tau_{< m}L \to \tau_{\leq m}L$ and $r \maps \tleq{m} L \to H_m$ be the linear maps
\[
t(x):=
\begin{cases}
s(x), & \text{if $\deg{x} =m $} \\
x, & \text{if $\deg{x} < m$,} \\
\end{cases} \qquad \hat{r}:= \id -tq_{\leq m}
\] 
respectively. Then the strict $L_\infty$-morphism $\hat{q} \maps \tleq{m}L \to \tlt{m}L \oplus H_m$ 
is defined to be
%\begin{equation} \label{eq:qhat}
\[
\hat{q}(z):=\bigl (q_{\leq m}(z), \hat{r}(z) \bigr)
\]
%\end{equation}
The map $\hat{q}'$ is defined in the analogous way, using the 
section $s'$ instead of $s$.
\end{remark}

\begin{remark} \label{rmk:tower_decomp2}
As shown in  Sec.\ 7.2.1 of \cite{Rogers:2018}, the structure maps
$\hat{\el}_k \maps \Lambda^k (\tlt{m}L \oplus H_m) \to \tlt{m}L \oplus H_m$
for the $L_\infty$-structure on $\tlt{m}L \oplus H_m$ are given by the formula:
% \[
% \hat{\el}_k := \hat{q} \circ \tleq{m}\el_k \circ (\hat{q}^{-1})^{\tensor k}.
% \]
% It then follows from the definition of $\hat{q}$ given above in Remark \ref{rmk:tower_decomp1} that,
\begin{multline} \label{eq:lhat}
\hat{\el}_k\bigl( (x_1,y_1), (x_2,y_2), \ldots, (x_k,y_k) \bigr) =  \\ 
 \Bigl( \tlt{m} \el_k \bigl (x_1,x_2,\ldots,x_k), ~ \hat{r}\circ \tleq{m} \el_k \bigl (t x_1 + y_1, t x_2 +y_2,\ldots, t x_k+ y_k) \Bigr),
\end{multline}    
for all $x_1,\ldots,x_k \in \tlt{m}L$ and $y_1,\ldots,y_k \in H_m$.
%for all $x_1,\ldots,x_k \in \tlt{m}L$ and $y_1,\ldots,y_k \in H_m$. 
The above formula implies 
that there is only one non-trivial structure map 
$\hat{\el}_k$ that involves non-zero inputs from $H_m$, since 
$H_m$ is concentrated in top degree $m$. Namely:
\[
\hat{\el}_2\bigl( (x,0), (0,y) \bigr) = \bigl(0, \el_2(x,y) \bigr),
\]
where $x \in L_0=\tlt{m}L_0$ is an element of degree 0 and $y \in H_m$. Moreover, since $\el_2$ satisfies the Leibniz rule with respect to the differential $\el_1=d$, we obtain an action of the Lie algebra $H_0(L)=L_0/\im d_1$ on $H_m$:
\begin{equation} \label{eq:lie_action1}
\begin{split}
H_0(L) \times H_m \to H_m \\
(x + \im d_1, y) \mapsto \el_2(x,y)
\end{split}
\end{equation}
This observation will play a key role when we integrate quasi-split fibrations in Sec.\ \ref{sec:int_exact}.
\end{remark}

\subsection{Maurer--Cartan elements} \label{sec:MC}
In Sec.\ \ref{sec:int_exact}, we express Henriques' integration functor in terms of Maurer-Cartan sets by applying some results  from Sec.\ 6 of \cite{Rogers:2018}.
Throughout the present section, $S$ denotes a submanifold (possibly with corners) of $\R^N$,
%\ccomment{Reworded the definition of $\Om(S)$}
Fix an integer $r \geq 1$. Following \cite[Sec.\ 5.1]{Henriques:2008}, 
we denote by 
\[
\bigl(\Omega(S), d_{\dR} \bigr)
\] 
the differential graded Banach algebra
of ``$r$-times continuously differentiable forms''. By definition, a $k$-form $\alpha$ on $S$ is an element of $\Omega(S)$ if and only if both $\alpha$ and the $(k+1)$-form $d_{\dR}\alpha$ are $r$-times continuously differentiable.
% % \] 
% % ,  $\Omega(S)$ %%a clearer way?
% %\begin{equation} \label{eq:deRham}
% %\[
%  %\Omega^n(S) := \{ C^r $n$ \text{form $\alpha$ on $S$ such that
%    %$d\alpha$ is a $C^r$ $n+1$ form} \}  %\bigl(\Omega(S), d_{\dR} \bigr)
% %\]
% %\end{equation}
% the set of $r$-times continuously differentiable (i.e. $C^r$) forms on $S$ such that
%    $d_{\dR}\alpha$ is again $C^r$. It
% turns out that $\bigl(\Omega(S), d_{\dR} \bigr)$ is a differential graded Banach algebra.

Note that we treat $(\Omega(S), d_{\dR})$ as a \underline{cochain} complex with $d_{\dR}=\Omega(S)^{\ast} \to \Omega(S)^{\ast +1 }$ as usual.
Let $(L,\el_{k}) \in \lnaft$ be a finite type Lie $n$-algebra, and denote by 
\[
(L \tensor \OS, \el^{\Om}) 
\]
the $\Z$--graded $L_\infty$-algebra whose underlying \underline{chain} complex is $(L \tensor \OS , \el^{\Om}_1)$ where 
\[
\begin{split}
(L \tensor \OS)_m &:= \bigoplus_{i-j=m} L_{i} \tensor \OS^{j}\\
\el^{\Om}_1 &:= \el_1 \tensor \id_{\OS} + \id_{L} \tensor d_{\dR}
\end{split}
\]
and whose higher brackets are defined as:
%\begin{equation} \label{eq:elB}
\[
\el^{\Om}_k \bigl(x_1 \tensor \omega_1, \ldots, x_k \tensor \omega_k \bigr):= (-1)^{\varepsilon} \el_k(x_1,\ldots,x_k) \tensor \omega_1\omega_2 \cdots \omega_k,
\]
with 
\[
\varepsilon :=  \sum_{1 \leq i < j \leq k} \deg{\omega_i}\deg{x_j}.
\]
%\end{equation} 
% If $\delta$ denotes the codifferential on $\S(\bs L)$ encoding the Lie $n$-algebra structure, then
% we let $\delta^{\Om}$ denote the induced codifferential for the $L_\infty$-structure on $L \tensor \OS$.

It is easy to see that for every $(L,\el) \in \lnaft$, the $L_\infty$-algebra $(L \tensor \OS, \el^{\Om})$ is ``tame'', in the sense of \cite[Def.\ 6.1]{Rogers:2018}. This means that the \textbf{curvature} of an element $a \in (L \tensor \OS)_{-1}$:
\begin{equation} \label{eq:curv}
\begin{split}
% \curv^{\Om}(a) & := \bs^{-1} \sum_{k \geq 1} \frac{1}{k!} (\delta^\Om)^1_k \bigl(\bs a, \bs a, \ldots, \bs a)\\ 
\curv^{\Om}(a)  = \ell^\Om_1(a) + \sum_{k \geq 2} \sgn{k} \frac{1}{k!} \ell^\Om_k(a,a,\ldots,a) \in (L \tensor \OS)_{-2}
\end{split}
\end{equation}
is well-defined, i.e., the above summation is finite. 
We recall that elements of the set
\[
\MC \bigl(L \tensor \OS \bigr):= \bigl \{ a \in (L \tensor \OS)_{-1} ~\vert ~ \curv^\Om(a)=0 \bigr \}.
\]
are called the \textbf{Maurer--Cartan elements} of the $L_\infty$-algebra $(L \tensor \OS, \el^{\Om})$.

Next, if $f \maps (L, \el) \to (L',\el')$ is a morphism of Lie $n$-algebras, then it is easy to verify that the maps $f^{\Om}_k\maps \Lambda^{k} \bigl(L \tensor \OS \bigr) \to L' \tensor \OS$ defined as
%\begin{equation} \label{eq:fB}
\[
f^{\Om}_k\bigl(x_1 \tensor \omega_1, \ldots, x_k \tensor \omega_k \bigr):= (-1)^{\varepsilon} f_k(x_1,\ldots,x_k) \tensor \omega_1\omega_2 \cdots \omega_k,
\]
%\end{equation} 
assemble together to give a $L_\infty$-morphism $f^\Om \maps (L\tensor \OS, \el^\Om) \to (L' \tensor \OS, \el^{ \prime \Om})$.
Furthermore, for any morphism $f \maps (L, \el) \to (L',\el')$ in $\lnaft$,  the $L_\infty$-morphism $f^\Om \maps (L\tensor \OS, \el^\Om) \to (L' \tensor \OS, \el^{ \prime \Om})$ is tame 
in the sense of \cite[Def.\ 6.1]{Rogers:2018}.
This implies that the morphism 
$f^\Om$ induces a well defined function 
\[
f^\Om_\ast \maps (L \tensor \OS)_{-1} \to (L'\tensor \OS)_{-1} 
\]
where
\begin{equation} \label{eq:Fstar}
\begin{split}
%f^\Om_\ast(a) &:= \bs^{-1} \sum_{k \geq 1} \frac{1}{k!} (F^\Om)^1_k( \bs a, \bs a, \ldots, \bs a) \\
f^\Om_\ast(a)  = f^\Om_1(a) + \sum_{k \geq 2} \sgn{k} \frac{1}{k!} f^\Om_k(a,a,\ldots,a).
\end{split}
\end{equation}
for all $a \in \bigl(L \tensor \OS \bigr)_{-1}$. Note that if $f=f_1 \maps (L, \el) \to (L',\el')$ is a strict morphism, then Eq.\ \ref{eq:Fstar} implies that
\begin{equation} \label{eq:strict_Fstar}
f^\Om_\ast = f \tensor \id_{\OS} 
\end{equation}
\begin{proposition}[] \label{prop:MC}
\mbox{}
\begin{enumerate}

\item  If  $f \maps (L, \el) \to (L',\el')$ is a morphism in $\lnaft$, then  $f_\ast \maps (L \tensor \OS)_{-1} \to (L'\tensor \OS)_{-1}$  
restricts to a well-defined function $$f_\ast \maps \MC(L \tensor \OS) \to \MC(L' \tensor \OS)$$ between the corresponding Maurer-Cartan sets. % Moreover, the assignment
% \[
% f \maps (L, \el) \to (L', \el') \quad \longmapsto \quad f_\ast \maps \MC(L) \to \MC(L').
% \]
% is functorial.

\item The assignments $(L\tensor \OS, \el^\Om) \mapsto \MC \bigl( (L\tensor \OS \bigr)$, and 
$f^\Om \mapsto f^{\Om}_\ast$ define a functor
%\begin{equation} \label{eq:MC-cdga}
\[
\MC(-\tensor \OS) \maps \lnaft \to \Set,
\]
%\end{equation}
natural in $S \subseteq \R^{N}$.

\item Let $(L,\el) \in \lnaft$ and let $\CE(L)$ denote the 
Chevalley--Eilenberg algebra of $(L,\el)$ defined in Sec.\ \ref{sec:CE_alg}.
Then the isomorphism of vector spaces
\[
\begin{split}
\vphi \maps L \tensor \OS  \xto{\cong} \hom_{\R}(\bs L^\vee, \OS) \\
\vphi(x \tensor \omega)(f):=f(\bs x) \omega
\end{split}
\] 
extends to a bijection of sets
\[
\MC(L \tensor \OS) \xto{\cong} \hom_{\cdga}(\CE(L),\OS)
\]
natural in $(L, \el) \in \lnaft$ and in $S  \subseteq \R^{N}$.
\end{enumerate}
\end{proposition}
\begin{proof}
Statements (1) and (2) follow from Prop.\ 6.3 and Lemma 6.4 in \cite{Rogers:2018}.

For (3), as mentioned in Sec.\ \ref{sec:CE_alg}, $\CE(L)$ is freely generated as a graded commutative algebra by the vector space $\bs L^{\vee}$. Hence, $\vphi$ induces an isomorphism 
$(L \tensor \OS)_{-1} \cong \hom_{\mathsf{cga}}(\S(\bs L^\vee),\OS)$. A direct calculation shows that 
$\curv^\Om(x \tensor \omega)=0$ if and only if $\vphi(x\tensor \omega) \circ \delta_{\CE} = d_{\dR} \circ \vphi(x\tensor \omega)$.
\end{proof}

\section{Integration of Lie $n$-algebras} \label{sec:int_exact}

\subsection{From Lie $n$-algebras to Lie $\infty$-groups}
A Lie $\infty$-group is an $\infty$-group object (Def.\ \ref{def:ngpd}) 
in $(\Mfd,\covers_{\subm})$. We denote by $\LnG{\infty} \sse \LnGpd{\infty}$ the full subcategory of Lie $\infty$-groups.
Let us recall Henriques' construction of Lie $\infty$-groups from Lie
$n$-algebras of finite type. Recall that we denote by $\Omega(\Delta^n)$ the differential graded Banach algebra of ``$r$-times continuously differentiable forms'' (as defined in Sec.\ \ref{sec:MC})
on the geometric $n$-simplex. 

\begin{pdef}[Def.\ 5.2, Thm.\ 5.10 \cite{Henriques:2008}] \label{pdef:int}
Let $L \in \LnA{n}^{\ft}$ be a finite type Lie $n$-algebra. 
The assignment
\[
L \mapsto \left(\int L \right)_m:= \MC\bigl ( L \tensor \Omega(\Delta^n) \bigr)
\cong \hom_{\mathrm{cdga}} \bigl( \CE(L),\Omega(\Delta^m) \bigr)
\]
induces a functor
\begin{equation} \label{eq:int}
\sint \maps \LnA{n}^{\ft} \to \LnG{\infty} 
\end{equation}
from the category of finite type Lie $n$-algebras to the category of
Lie $\infty$-groups.
\end{pdef}

\begin{remark} \label{rmk:smoothness}
In \cite{Severa-Siran}, \v{S}evera and \v{S}ira\v{n} prove that $\sint L$ is a simplicial Banach manifold using an approach that differs from the one taken by Henriques' in his proof of 
\cite[Thm.\ 5.10]{Henriques:2008}. If $(L,\el) \in \lnaft$, then since $L$ is finite type, the Banach algebra structure on $\Omega(\Delta^m)$ naturally makes $L \tensor \Omega(\Delta^m)$ into a graded Banach space. Then Prop.\ 4.3 of \cite{Severa-Siran} implies that
$\MC\bigl ( L \tensor \Omega(\Delta^m) \bigr) \sse (L \tensor \Omega(\Delta^m)_{-1}$ is a Banach submanifold in the sense of \cite[Ch.\ I,3]{Lang:95}. From Eq.\ \ref{eq:curv},
we see that the curvature $\curv^{\Om} \maps (L \tensor \Omega(\Delta^m))_{-1} \to (L \tensor \Omega(\Delta^m))_{-2}$ is a polynomial function and hence smooth. In other words, 
%\begin{equation} \label{eq:MC_equalizer}
\[
\begin{tikzpicture}[descr/.style={fill=white,inner sep=2.5pt},baseline=(current  bounding  box.center)]
\matrix (m) [matrix of math nodes, row sep=2em,column sep=2em,
  ampersand replacement=\&]
  {  
\MC({L}\tensor \Omega(\Delta^m)) \&  ({L}\tensor \Omega(\Delta^m))_{-1} \& ({L}\tensor \Omega(\Delta^m))_{-2}\\
}; 
%Tikz Equalizer stuff: ($(m-1-1.east)+(x,y)$)
   \path[->,font=\scriptsize] 
($(m-1-2.east)+(0,0.25)$) edge node[above] {${\curv^\Om}$} ($(m-1-3.west)+(0,0.25)$) 
($(m-1-2.east)+(0,-0.25)$) edge node[below] {$0$} ($(m-1-3.west)+(0,-0.25)$) 
;   
   \path[right hook->,font=\scriptsize] 
   (m-1-1) edge node[auto] {$$} (m-1-2);
\end{tikzpicture}
%\end{equation}
\]
is an equalizer diagram in the category $\Mfd$ of Banach manifolds.  
\end{remark}

Before we proceed further, let us recall a very useful example of a Lie $\infty$-group stemming from the integration of a Lie $n$-algebra.
\begin{example} \label{ex:flatconn}
Following \cite[Example 5.5]{Henriques:2008}, let $(L,\el) \in \lnaft$ and
recall from Remark \ref{rmk:tau0} that $\tau_{\leq 0} L=H_0(L)$ is a Lie algebra.
Consider the Lie $\infty$-group $\int \tau_{\leq 0} L$. For each $m \geq 0$, we have a natural identification
%there is a natural 1-1 correspondence, as in Lemma \ref{lem:MC_CE}:
\[
\Bigl(\int \tau_{\leq 0}L  \Bigr )_m= \MC \bigl( \tau_{\leq 0}L \tensor \Omega^{\ast}(\Delta^m) \bigr)
 \cong \Bigl( \tau_{\leq 0}L  \tensor \Omega^{1}(\Delta^m) \Bigr)^\flat
% \left(\int \g \right)_m= \hom_{\mathrm{cdga}} \bigl( C^\ast(\g),\Omega^{\ast}(\Delta^m) \bigr)
% \cong \Bigl( \Omega^{1}(\Delta^m)\tensor \g \Bigr)^\flat
\]
between Maurer-Cartan elements
and flat connections on the trivial bundle $G \times \Delta^m \to \Delta^m$, where $G$ is the 1-connected Lie group integrating $\tau_{\leq 0}L$.  
There is also an identification
\begin{equation} \label{eq:flatconn}
\begin{split}
T \maps \Map(\Delta^m,G)/G &\xto{\cong} \Bigl( \tau_{\leq 0}L  \tensor \Omega^{1} \bigl(\Delta^m \bigr) \Bigr)^\flat \\
\Delta^m \xto{\Gamma} G \quad &\mapsto \quad  T(\Gamma):=\Gamma^{-1}d\Gamma
\end{split}
\end{equation}
between the set  $\Map(\Delta^m,G)/G$ of $G$-valued $C^{r+1}$ maps, modulo constants, and the set of flat connections. This identification is natural in $(L,\el)$. See Lemma \ref{lem:nat_id}.
\end{example}

\begin{remark} \label{rmk:Lie_n-groups}
Note that the integration functor \eqref{eq:int} a priori
assigns to a Lie $n$-algebra not a Lie $n$-group but a Lie
$\infty$-group. To resolve this, Henriques introduced a truncation functor 
\cite[Def.\ 3.5]{Henriques:2008} from Lie $\infty$-groups to simplical sheaves over $\Mfd$.
Furthermore, Henriques observed that, in general, there are obstructions to 
representing the output of the truncation functor as an actual Lie $n$-group. Further discussion of the truncation functor within the context of the homotopy theory of Lie $n$-groups will appear in future work. 
\end{remark}

We conclude this section with a warm-up result concerning the preservation of pullback squares by the 
integration functor \eqref{eq:int}. We address this in further detail in Thm.\ \ref{thm:int_exact}.

\begin{proposition} \label{prop:MC-pullback}
Let $f \maps (L,\el) \to (L'',\el'')$ be a fibration and $g \maps (L',\el') \to (L'',\el'')$ be a morphism in $\lnaft$, and let $(\ti{L},\ti{\el})$ be the pullback of the diagram
$(L',\el') \xto{g} (L'',\el'') \xleftarrow{f} (L'',\el'')$.
If the pullback of
\[
\MC(L' \tensor \Omega^{\ast}(\Delta^m)) \xto{g^{\Om}_{\ast}} \MC(L'' \tensor \Omega^{\ast}(\Delta^m))
\xleftarrow{f^{\Om}_{\ast}} \MC(L \tensor \Omega^{\ast}(\Delta^m))
\]
exists in the category $\Mfd$, then the induced commutative diagram
\[
\begin{tikzpicture}[descr/.style={fill=white,inner sep=2.5pt},baseline=(current  bounding  box.center)]
\matrix (m) [matrix of math nodes, row sep=2em,column sep=3em,
  ampersand replacement=\&]
  {  
\MC(\ti{L} \tensor \Omega^{\ast}(\Delta^m)) \& \MC({L} \tensor \Omega^{\ast}(\Delta^m)) \\
\MC(L' \tensor \Omega^{\ast}(\Delta^m)) \& \MC(L'' \tensor \Omega^{\ast}(\Delta^m))\\
}; 
  \path[->,font=\scriptsize] 
   (m-1-1) edge node[auto] {$$} (m-1-2)
   (m-1-1) edge node[sloped,below] {$$} (m-2-1)
   (m-1-2) edge node[auto] {$f^\Om_\ast$}  (m-2-2)
   (m-2-1) edge node[auto] {$g^\Om_\ast$} (m-2-2)
   % (m-2-1) edge node[auto,swap] {$\ppr' \vert_P$} (m-3-1)
   % (m-2-2) edge node[auto] {$F$} (m-3-2)
   % (m-3-1) edge node[auto] {$G$} (m-3-2)
  ;

% %begin pullback symbol%
  % \begin{scope}[shift=($(m-1-1)!.4!(m-2-2)$)]
  % \draw +(-0.25,0) -- +(0,0)  -- +(0,0.25);
  % \end{scope}
  %end pullback symbol%

\end{tikzpicture}
\]
is a pullback diagram in $\Mfd$.
\end{proposition} 
\begin{proof}
Remark \ref{rmk:smoothness} implies that the Maurer-Cartan set $\MC({L} \tensor \Omega^{\ast}(\Delta^m))$ is an equalizer in $\Mfd$ for any $(L,\el) \in \lnaft$. Therefore 
Assumption 6.6 in Sec.\ 6.2 of \cite{Rogers:2018} is satisfied, and
the proposition then follows from \cite[Cor.\ 6.7]{Rogers:2018}.
\end{proof}

We will use the following corollary in our proof of Thm.\ \ref{thm:int_split} in the next section.

\begin{cor} \label{cor:MC-pullback}
The integration functor $\sint \maps \LnA{n}^{\ft} \to \LnG{\infty}$ preserves products.
\end{cor}
\begin{proof}
The category $\LnG{\infty}$ has finite products and every trivial morphism $(L,\el) \to 0$ 
in $\lnaft$ is a fibration.
% Prop.\ \ref{prop:MC-pullback} immediately implies the desired result.
\end{proof}

\subsection{Integrating fibrations} \label{sec:int_fib}
We begin by analyzing the naturality of a construction used by Henriques 
in his proof of the following result:
\begin{proposition}[Thm.\ 5.10 \cite{Henriques:2008}] \label{prop:int_tower_andre}
Let $(L,\el) \in \lnaft$. The integrations of the $L_\infty$-morphisms $q_{\leq m}$ and $q_{\lt m+1}$ defined in \eqref{eq:projs1}:  
\[
\smallint q_{\leq m} \maps \sint \tleq{m}L \to \sint \tlt{m} L,
\quad \smallint q_{\lt m+1} \maps \sint \tlt{m+1}L \to \sint \tleq{m} L
\]
are fibrations between Lie $\infty$-groups.
\end{proposition} 

In what follows, the submanifold
\[
S \subseteq \R^{k+1}
\]
denotes either the geometric simplex $\Delta^k$ or a geometric horn $\Lambda^{k}_{j} \subseteq \Delta^k$.

Let us fix a Lie $n$-algebra $(L,\el) \in \lnaft$ and an integer $m\geq 1$. 
Our goal is to express elements of $\MC \bigl(\tleq{m}L \tensor \Omega(S) \bigr)$
as pairs consisting of an element of $\MC \bigl(\tlt{m}L \tensor \Omega(S) \bigr)$ and
a $H_m(L)$-valued differential form. Via the strict $L_\infty$-isomorphism
\[
\hat{q} \maps (\tleq{m} L, \tleq{m}\el) \xto{\cong} \bigl( \tlt{m}L \oplus H_m, \hat{\el} ~ \bigr)
\]
given in Prop.\ \ref{prop:tower_decomp2}, we have the identification
\[
\begin{split}
\MC \bigl(\tleq{m}L \tensor \Omega(S) \bigr) 
&\cong \MC \bigl( (\tlt{m}L \oplus H_m)\tensor \Omega(S) \bigr) \\
& \quad \subseteq \bigl(\tlt{m} L \tensor \Omega(S) \bigr)_{-1}  \oplus ~ H_m \tensor \Omega^{m+1}(S),
\end{split}
\]
where $H_m=H_m(L)$.
Let $\wht{\curv}^\Omega$, $\curv^{\Omega}_{\leq m}$ and $\curv^{\Omega}_{\lt m}$
denote the curvature functions \eqref{eq:curv} for the $L_\infty$-algebras
\[
\Bigl(\bigl(\tlt{m}L \oplus H_m \bigr) \tensor \Omega(S), \hat{\el}^{\sOm} \Bigr), \quad
\bigl(\tleq{m}L \tensor \Omega(S), \tleq{m}\el^{\sOm} \bigr), \quad
\bigl(\tlt{m}L \tensor \Omega(S), \tlt{m} \el^{\sOm} \bigr),
\]
respectively. 

Since the $L_\infty$-isomorphism $\hat{q}^{\sOm}=\hat{q}\tensor \id_{\OS}$ is strict, 
we can easily write $\wht{\curv}^\Omega$ in terms of the other two curvature functions above.
Indeed, consider a degree $-1$ element $(\sigma, \nu) \in 
\bigl(\tlt{m} L \tensor \Omega(S) \bigr)_{-1}  \oplus \: H_m \tensor \Omega^{m+1}(S) $.
We first express its curvature as the sum of degree $-2$ elements: 
${\wht{\curv}}^{\Omega}(\sigma,\nu) = \sum^m_{i=1} \wht{\curv}^{\Omega}(\sigma,\nu)_i$
where
\[
\wht{\curv}^{\Omega}(\sigma,\nu)_ i \in  \tlt{m} L_i \tensor \Omega^{i+2}(S)
\quad \text{for $i < m$}
\]
and
\[
\wht{\curv}^{\Omega}(\sigma,\nu)_m \in  \tlt{m} L_m \tensor \Omega^{m+2}(S)
~ \oplus ~ H_m \tensor \Omega^{m+2}(S).
\]
We then use formula \eqref{eq:lhat} for the $L_\infty$-structure on $\bigl(\tlt{m}L \oplus H_m \bigr)$ to obtain the equality
\[
\wht{\curv}^{\Omega}(\sigma,\nu)_i = 
\begin{cases}
\bigl( \curv^{\Omega}_{\lt m}(\sigma)_i, 0 \bigr) & \text{if $i < m$} \\
\bigl( \curv^{\Omega}_{\lt m}(\sigma)_m,  d_{\dR}\nu - \el^\Om_2(\sigma_0, \nu) - \kappa(\sigma) \bigr),
&  \text{if $i =m$.}
\end{cases}
\]
Above $\sigma_0 \in L_0 \tensor \Omega^{1}(S)$ is the component of $\sigma$ in bidegree $(0,-1)$ and $\kappa(\sigma)$ denotes the $H_m$-valued $(m+2)$-form
\begin{equation} \label{eq:tower_MCelts0}
\kappa(\sigma):=  -(\pr_{H_m} \tensor \id_{\OS}) \circ \: \hat{q}^{\sOm} \circ \curv^{\Omega}_{\leq m} \circ (\hat{q}^{\sOm})^{-1}(\sigma,0)  \in H_m \tensor \Omega^{m+2}(S),
\end{equation}
where $\pr_{H_m} \maps \tlt{m}L \oplus H_m \to H_m$ is the linear projection.

Hence,  we have the following characterization of the Maurer-Cartan elements:
\begin{multline*} %\label{eq:tower_MCelts1}
\MC \bigl(\tleq{m}L \tensor \Omega(S) \bigr) = \\
\Bigl\{ (\sigma, \nu) \in \MC \bigl(\tlt{m}L \tensor \Omega(S) \bigr) \oplus \:
H_m \tensor \Omega^{m+1}(S) ~ \vert ~ d_{\dR}\nu - \el^{\sOm}_2(\sigma_0, \nu) = \kappa(\sigma)
\Bigr \}.
\end{multline*}

Now suppose $\sigma$ is a Maurer-Cartan element  of $\tlt{m}L \tensor \Omega(S)$
and $\nu \in H_m \tensor \Omega^{m+1}(S)$ is a $H_m$-valued $(m+1)$-form. 
As observed by Henriques in his proof of \cite[Thm.\ 5.10 ]{Henriques:2008}, 
the condition 
\begin{equation} \label{eq:tower_MCelts2}
d_{\dR}\nu - \el^{\sOm}_2(\sigma_0, \nu) = \kappa(\sigma),
\end{equation}
can be rewritten by exploiting the action \eqref{eq:lie_action1} of the Lie algebra $H_0(L)$ on $H_m \tensor \Omega^{\ast}(S)$. We do this in the following way: Let $\alpha \in H_0(L) \tensor 
\Omega^{1}(S)$ denote the class represented by $\sigma_0$. Since $\curv^{\Omega}_{\lt m}(\sigma)=0$, we have $d_{\dR} \alpha - \frac{1}{2}[\alpha,\alpha]=0$, where $[ \cdot,\cdot]$ denotes the Lie bracket on $H_0(L)\tensor \OS$ induced by $\el^{\sOm}_2$. Hence, $\alpha$ is a flat $H_0(L)$-valued connection on
$S$. Furthermore, since $\el_2$ satisfies the Leibniz rule, one observes that $(\sigma,\nu)$ satisfies  Eq.\ \ref{eq:tower_MCelts2} if and only if 
\begin{equation} \label{eq:tower_MCelts3}
d_{\dR}\nu - [\alpha, \nu] = \kappa(\sigma).
\end{equation}
Let $G$ be the 1-connected Lie group integrating $H_0(L)$. Then $H_m \tensor \Omega^{\ast}(S)$ integrates to a $G$-module in the standard way. Via the identification \eqref{eq:flatconn} introduced in Example \ref{ex:flatconn}, there exists a unique function $\Gamma \maps S \to G$
such that 
\begin{equation}\label{eq:tower_MCelts4}
-\alpha = -\Gamma^{-1}d\Gamma, \quad  \Gamma(v_0)=e,
\end{equation}
where $v_0$ is an arbitrary fixed vertex of $S$. Then, as shown in \cite[Eqs.\ 30--32]{Henriques:2008}, $(\sigma, \nu)$ satisfy Eq.\ \ref{eq:tower_MCelts3} if and only if
\begin{equation} \label{eq:tower_MCelts5}
d_{\dR} (\Gamma \cdot \nu) = \Gamma \cdot \kappa(\sigma).
\end{equation}
This characterization of Maurer-Cartan elements using the $H_0(L)$ action is natural in the following sense: 
\begin{proposition} \label{prop:tower_MCelts}
Let $S=\Lambda^k_j$ or $\Delta^k$ and $m \geq 1$. Suppose $f=f_1 \maps (L, \el) \to (L',\el')$ is a strict quasi-split fibration in $\lnaft$. Let 
\begin{equation}
\label{diag:tower_MCelts}
\begin{tikzpicture}[descr/.style={fill=white,inner sep=2.5pt},baseline=(current  bounding  box.center)]
\matrix (m) [matrix of math nodes, row sep=2em,column sep=5em,
  ampersand replacement=\&]
  {  
\MC \Bigl ( \tau_{\leq m }L \tensor \Omega(S) \Bigr) \& \MC \Bigl ( (\tau_{\lt m }L \oplus H_m) \tensor \Omega(S) \Bigr)\\
\MC \Bigl (\tau_{\leq m }L' \tensor \Omega(S) \Bigr) \& \MC \Bigl ( (\tau_{\lt m }L' \oplus H_m')\tensor \Omega(S)  \Bigr)\\
}; 
\path[->,font=\scriptsize] 
(m-1-1) edge node[auto]{$\hat{q}^\Omega_\ast$} node[auto,below] {$\cong$} (m-1-2)
(m-1-1) edge node[auto,swap] {$\tleq{m} f^\Omega_*$} (m-2-1)
(m-1-2) edge node[auto] {$\tlt{m} f^\Omega_\ast \oplus H(f)\tensor \id_\Om$} (m-2-2)
(m-2-1) edge node[auto]{$\hat{q}^{\prime \Omega }_\ast$} node[below] {$\cong$} (m-2-2)
;
\end{tikzpicture}
\end{equation}
be the commutative diagram of smooth manifolds induced by diagram
\eqref{diag:tower_decomp2.1} in Prop.\ \ref{prop:tower_decomp2}.
If   $ a \in \MC \bigl(\tleq{m}L \tensor \Omega(S) \bigr)$ is a Maurer-Cartan element and we define
\[
(\sigma,\nu):=\hat{q}^{\sOm}_\ast(a), \quad (\sigma',\nu'):= \hat{q}^{\prime \sOm}_\ast \circ \tleq{m}f^\Omega_\ast(a),
\]
then we have the following equalities of $H'_m$-valued differential forms:
\[
\begin{split}
(H(f) \tensor \id_{\OS}) \bigl( \Gamma \cdot \nu \bigr) &= \Gamma' \cdot \nu'\\ 
(H(f) \tensor \id_{\OS}) \bigl( \Gamma \cdot \kappa(\sigma) \bigr) &= \Gamma' \cdot \kappa'(\sigma'),
\end{split}
\] 
where  $\Gamma' \maps S \to G'$ is the unique integration defined in Eq.\ \ref{eq:tower_MCelts4}
of the flat $H_0(L')$-valued connection induced by $\sigma'_0 \in L'_0 \tensor \Omega^{1}(S)$, and $\kappa'(\sigma')$ is the differential form defined via Eq.\ \ref{eq:tower_MCelts0}.

\end{proposition}

\begin{proof}
Since $f=f_1$ is a strict $L_\infty$-morphism, 
it follows from the commutativity of diagram
\eqref{diag:tower_MCelts} and the formula for $\tlt{m}f^\Omega_\ast$ (Eq.\ \ref{eq:strict_Fstar}) that: 
\[
(\sigma',\nu') = \Bigl( (\tlt{m}f \tensor \id_{\OS})(\sigma), (H(f) \tensor \id_{\OS}) (\nu) \Bigr).
\]
In particular, $\sigma'_0 = (f\tensor \id_{\OS})(\sigma_0)$.
If $\Gamma \maps S \to G$ is the unique integration \eqref{eq:tower_MCelts4}
of the flat $H_0(L)$-valued connection induced by $\sigma_0$, then, by uniqueness, $\Gamma' = \Phi \circ \Gamma$, where $\Phi \maps G \to G'$ is the Lie group homomorphism integrating $H(f) \maps H_0(L) \to H_0(L')$. The linear map $H(f) \tensor \id_{\OS} \maps H_m \tensor \Omega(S) \to H'_m \tensor \Omega(S)$ intertwines the Lie algebra actions, since $f$ is a strict $L_\infty$-morphism. Hence, it also intertwines the Lie group actions:
\[
(H(f) \tensor \id_{\OS}) \bigl( \Gamma \cdot \nu \bigr)
%= (H(f_1) \tensor \id) \bigl( \Gamma \cdot \nu \bigr)
=(\Phi \circ \Gamma) \cdot \nu' = \Gamma' \cdot \nu'.
\]
Similarly, we have
\begin{equation} \label{eq:tower_MCelts6}
(H(f) \tensor \id_{\OS}) \bigl( \Gamma \cdot \kappa(\sigma) \bigr) = 
%(H(f) \tensor \id_{\OS}) \bigl( \Gamma \cdot \kappa(\sigma) \bigr) =
\Gamma' \cdot 
(H(f) \tensor \id_{\OS}) \bigl(\kappa(\sigma) \bigr).
\end{equation}
It follows from the definition of $\kappa(\sigma)$ in Eq.\ \ref{eq:tower_MCelts0} that:
\[
\bigl (H(f) \tensor \id_{\OS} \bigr) \bigl(\kappa(\sigma) \bigr)  =
\bigl ((H(f)\pr_{H_m}\hat{q} \bigr) \tensor \id_{\OS}  \circ \curv^{\Omega}_{\leq m} \circ (\hat{q}^{-1}\tensor \id_{\OS})(\sigma,0).  
\]
The commutative diagram \eqref{diag:tower_decomp2.1} implies that 
$H(f) \circ \pr_{H_m}\hat{q}= \pr_{H'_m}\hat{q}^{\prime} \circ \tleq{m}f$. Since $\tleq{m}f$ is a strict $L_\infty$-morphism, we have 
\[
(\tleq{m}f \tensor \id_{\OS})\circ \curv^\Omega_{\leq m} = 
\curv^{\prime \sOm}_{\leq m} \circ (\tleq{m}f \tensor \id_{\OS})
\]
where $\curv^{\prime \sOm}_{\leq m}$ is the curvature function for the $L_\infty$-algebra $\tleq{m}L' \tensor \Omega(S)$. Hence, we obtain the following equalities:
\[
\begin{split}
(H(f) \tensor \id_{\OS}) \bigl(\kappa(\sigma) \bigr)  & = 
(\pr_{H'_m}\hat{q}^{\prime}  \tensor \id_{\OS}) \curv^{\prime \Omega}_{\leq m} \circ (\tleq{m}f \tensor \id_{\OS})
(\hat{q}^{-1}\tensor \id_{\OS})(\sigma,0)  \\
&= (\pr_{H'_m}\hat{q}^{\prime}  \tensor \id_{\OS}) \curv^{\prime \Omega}_{\leq m} 
\circ (\hat{q}^{\prime -1}\tensor \id_{\OS})\circ \\
& \qquad \bigl( (\tlt{m}f \oplus H(f)) \tensor \id_{\OS} \bigr) (\sigma,0)  \\
&= (\pr_{H'_m}\hat{q}^{\prime}  \tensor \id_{\OS}) \curv^{\prime \Omega}_{\leq m} 
\circ (\hat{q}^{\prime -1}\tensor \id_{\OS})(\sigma',0)  \\
&=\kappa'(\sigma').
\end{split}
\]
By combining this last equality with Eq.\ \ref{eq:tower_MCelts6}, we conclude that
\[
(H(f) \tensor \id_{\OS}) \bigl( \Gamma \cdot \kappa(\sigma) \bigr) = \Gamma' \cdot \kappa'(\sigma').
\]
\end{proof}

%\ccomment{changes of $\oplus$ to $\times$ in statement below}
\begin{cor} \label{cor:tower_MCelts}
Let $m \geq 1$ and suppose $f=f_1 \maps (L, \el) \to (L',\el')$ is a strict quasi-split fibration as in Prop.\ \ref{prop:tower_MCelts}. Let
\[
\begin{split}
X(S)&:= \Bigl\{ (\sigma, \rho) \in \MC \bigl(\tlt{m}L \tensor \Omega(S) \bigr) \times 
\bigl(H_m \tensor \Omega^{m+1}(S) \bigr) ~ \vert ~ d_{\dR}\rho =\gamma(\sigma)
\Bigr \} \\
X'(S)&:= \Bigl\{ (\sigma', \rho') \in \MC \bigl(\tlt{m}L' \tensor \Omega(S) \bigr) \times 
\bigl(H'_m \tensor \Omega^{m+1}(S) \bigr) ~ \vert ~ d_{\dR}\rho' =\gamma'(\sigma')
\Bigr \},
\end{split}
\]
where $\gamma(\sigma)$ and $\gamma'(\sigma')$ denote the $(m+2)$-forms $\Gamma \cdot \kappa(\sigma)$ and $\Gamma' \cdot \kappa'(\sigma')$, respectively, as in Prop.\ \ref{prop:tower_MCelts}.
Then the following diagram of smooth manifolds commutes:
\begin{equation}
\label{diag:tower_MCelts2}
\begin{tikzpicture}[descr/.style={fill=white,inner sep=2.5pt},baseline=(current  bounding  box.center)]
\matrix (m) [matrix of math nodes, row sep=2em,column sep=8em,
  ampersand replacement=\&]
  {  
\MC \Bigl ( (\tau_{\lt m }L \oplus H_m) \tensor \Omega(S) \Bigr) \& X(S) \\
\MC \Bigl ( (\tau_{\lt m }L' \oplus H_m')\tensor \Omega(S)  \Bigr) \& X'(S)\\
}; 
\path[->,font=\scriptsize] 
(m-1-1) edge node[auto]{$(\id_{\tlt{m}L \tensor \Om(S)}, \Gamma \cdot -)$} node[auto,below] {$\cong$} (m-1-2)
(m-1-1) edge node[auto,swap] {$\tlt{m} f^\Omega_\ast \oplus (H(f)\tensor \id_\Om)$} (m-2-1)
(m-1-2) edge node[auto] {$\tlt{m}f \tensor \id_\Om \, \times \, H(f)\tensor \id_\Om$} (m-2-2)
(m-2-1) edge node[auto]{$(\id_{\tlt{m}L' \tensor \Om(S)}, \Gamma' \cdot -)$} node[below] {$\cong$} (m-2-2)
;
\end{tikzpicture}
\end{equation}
\end{cor}

\begin{proof}
Given a finite-dimensional $G$-module $H$ and a smooth map $g \maps S \to G$, the linear map
\[
H \tensor \Omega(S) \xto{g\cdot -} H \tensor \Omega(S), \quad
\nu \mapsto g \cdot \nu
\]
is a smooth isomorphism between Banach spaces.
The rest of the corollary then follows from Prop.\ \ref{prop:tower_MCelts} and Eq.\ \ref{eq:tower_MCelts5}.
\end{proof}

We now prove the main result of this section.
\begin{theorem} \label{thm:int_split}
Let $f \maps (L,\el) \to (L',\el')$ be a quasi-split fibration in $\lnaft$.
Then the simplicial map
\[
\smallint f  \maps \sint L \to \sint L'
\]
is a fibration between Lie $\infty$-groups.
\end{theorem}

\begin{proof}
Corollary \ref{cor:strict_fib} implies that we can factor any quasi-split fibration into an $L_\infty$-isomorphism, followed by a strict quasi-split fibration. Therefore, in order to prove the theorem, we only need to consider the case when $f$ is strict.

Let $f=f_1 \maps (L,\el) \to (L',\el')$ be a strict quasi-split fibration. Let $k >0$ and $0 \leq j \leq k$. It is straightforward to show that $\sint f$ satisfies the Kan condition $\Kan(k,j)$ if and only if, for $m \geq k+1$, the integration of the truncated $L_\infty$-morphism
\[
\sint \tleq{m} f \maps \sint (\tleq{m}L, \tleq{m}\el) \to  \sint(\tleq{m}L', \tleq{m}\el'), 
\]
satisfies $\Kan(k,j)$. Hence, it suffices to show that every vertical map appearing in the morphism of Postnikov towers \eqref{diag:tower} integrates to a fibration between Lie $\infty$-groups.

We proceed by induction on the truncation level $m$. For the base case, we need to verify that
$\sint \tleq{0}L \to \sint \tleq{0}L'$ is a fibration. By definition of the truncation functor $\tleq{0}$, this is equivalent to verifying that Lie algebra morphism $H(f) \maps H_0(L) \to H_0(L')$ integrates to a fibration between Lie $\infty$-groups of the type discussed in Example \ref{eq:flatconn}. % \ccomment{changed next sentence to emphasize new def of quasi-split}
The hypothesis on $H_0(f)$ implies that  $H_0(L)$ is the trivial extension of $H_0(L')$ by $\ker H_0(f)$. Therefore, we have the following commutative diagram of Lie algebras
\[
\begin{tikzpicture}[descr/.style={fill=white,inner sep=2.5pt},baseline=(current  bounding  box.center)]
\matrix (m) [matrix of math nodes, row sep=1em,column sep=1em,
  ampersand replacement=\&]
  {  
H_0(L)  \& \& \ker f \oplus H_0(L')   \\
 \&  H_0(L') \&   \\
};
\path[->,font=\scriptsize] 
(m-1-1) edge node[auto] {$\cong$} (m-1-3)
(m-1-1) edge node[auto,swap] {$H(f)$} (m-2-2)
(m-1-3) edge node[auto] {$\pr_{H_0(L')}$} (m-2-2)
;
\end{tikzpicture}
\]
Corollary \ref{cor:MC-pullback} implies that the integration functor preserves products.
Since $\ker f \oplus H_0(L')$ is a product of Lie algebras, we conclude that    
$\sint \pr_{H_0(L')} = \pr_{\sint H_0(L')}$ is a fibration, and hence, $\sint H(f)$ is a fibration as well. 
This completes the base case.

The induction step involves two cases.\\ 

\underline{Case 1:} First, let $m \geq 0$ and 
consider the following commutative diagram in $\lnaft$:
\[
\begin{tikzpicture}[descr/.style={fill=white,inner sep=2.5pt},baseline=(current  bounding  box.center)]
\matrix (m) [matrix of math nodes, row sep=2em,column sep=2em,
  ampersand replacement=\&]
  {  
\tau_{< m +1 }L \& \tau_{\leq m }L\\
\tau_{< m +1 }L' \& \tau_{\leq m }L'\\
}; 
\path[->,font=\scriptsize] 
(m-1-1) edge node[auto] {$q_{< m+1}$} (m-1-2)
(m-1-1) edge node[auto,swap] {$\tlt{m+1}f$} (m-2-1)
(m-1-2) edge node[auto] {$\tleq{m}f$} (m-2-2)
(m-2-1) edge node[auto] {$q'_{< m+1}$} (m-2-2)
;
\end{tikzpicture}
\]
Suppose that the morphism $\sint \tleq{m}f$ is a fibration in $\LnG{\infty}$. We will show that $\sint \tlt{m+1}f$ is also a fibration. Since $f$ is a strict quasi-split fibration, Prop.\ \ref{prop:tower_decomp1} implies that the above diagram in $\lnaft$ decomposes into  
\begin{equation} \label{diag:int_split1}
\begin{tikzpicture}[descr/.style={fill=white,inner sep=2.5pt},baseline=(current  bounding  box.center)]
\matrix (m) [matrix of math nodes, row sep=2em,column sep=5em,
  ampersand replacement=\&]
  {  
\tau_{< m +1 }L \& \tau_{\leq m }L \oplus \ker q_{<m+1}\\
\tau_{< m +1 }L' \& \tau_{\leq m }L' \oplus \ker q'_{<m+1}\\
}; 
\path[->,font=\scriptsize] 
(m-1-1) edge node[auto]{$\bigl(q_{< m+1},r \bigr)$} node[auto,below] {$\cong$} (m-1-2)
(m-1-1) edge node[auto,swap] {$\tau_{< m +1 } f$} (m-2-1)
(m-1-2) edge node[auto] {$\tau_{\leq m} f \oplus \tau_{< m +1 }f \vert_{\ker}$} (m-2-2)
(m-2-1) edge node[auto]{$\bigl(q'_{< m+1},r' \bigr)$} node[below] {$\cong$} (m-2-2)
;
\end{tikzpicture}
\end{equation}
Recall from \eqref{eq:ker-qm+1} that 
$\ker q_{<m+1} = \im d_{m+1}[-1] \oplus \im d_{m+1}$ and $\ker q'_{<m+1}=\im d'_{m+1}[-1] \oplus \im d'_{m+1}$ are abelian Lie $n$-algebras with isomorphisms as differentials.
The $L_\infty$-morphism $f$ is surjective in all degrees, and hence $f \vert_{\im d_{m+1}}$ is surjective as well.
We obtain the following isomorphisms directly from the definition of the integration functor:
\[
\sint \ker q_{<m+1} \cong \im d_{m+1} \tensor \Omega^{m+1}(\Delta^\bullet), \quad \sint \ker q'_{<m+1} \cong \im d'_{m+1} \tensor \Omega^{m+1}(\Delta^\bullet).
\]
Hence, $\sint \tau_{< m +1 }f \vert_{\ker} = f \vert_{\im d_{m+1}} \tensor \id_{\Omega}$ is a surjective linear map between simplicial Banach spaces. Therefore from Lemma 5.9 in \cite{Henriques:2008} we deduce that
$\sint \tau_{< m +1 }f \vert_{\ker}$ is a fibration between Lie $\infty$-groups. This observation, when combined with the induction hypothesis and Cor.\ \ref{cor:MC-pullback}, implies that 
\[
\sint \tau_{\leq m} f \times \sint \tau_{< m +1 }f \vert_{\ker} = \sint (\tau_{\leq m} f \oplus \tau_{< m +1 }f \vert_{\ker})
\]   
is also a fibration. It then follows from the commutativity of diagram \eqref{diag:int_split1} that 
$\sint \tlt{m+1}f$ is also a fibration.\\

\underline{Case 2:} 
Now let $m \geq 1$ and consider the commutative diagram
\begin{equation} \label{diag:int_split2}
\begin{tikzpicture}[descr/.style={fill=white,inner sep=2.5pt},baseline=(current  bounding  box.center)]
\matrix (m) [matrix of math nodes, row sep=2em,column sep=2em,
  ampersand replacement=\&]
  {  
\tau_{\leq m }L \& \tau_{\lt m }L\\
\tau_{\leq m }L' \& \tau_{\lt m }L'\\
}; 
\path[->,font=\scriptsize] 
(m-1-1) edge node[auto] {$q_{\leq m}$} (m-1-2)
(m-1-1) edge node[auto,swap] {$\tleq{m}f$} (m-2-1)
(m-1-2) edge node[auto] {$\tlt{m}f$} (m-2-2)
(m-2-1) edge node[auto] {$q'_{\leq m}$} (m-2-2)
;
\end{tikzpicture}
\end{equation}
in $\lnaft$. Suppose that the morphism $\sint \tlt{m}f$ is a fibration in $\LnG{\infty}$. We will show that $\sint \tleq{m}f$ is also a fibration. Proposition \ref{prop:tower_decomp2} 
gives us the following diagram of strict $L_\infty$-morphisms
\begin{equation} \label{diag:int_split2a}
\begin{tikzpicture}[descr/.style={fill=white,inner sep=2.5pt},baseline=(current  bounding  box.center)]
\matrix (m) [matrix of math nodes, row sep=2em,column sep=2em,
  ampersand replacement=\&]
  {  
( \tau_{\leq m }L, \tleq{m} \el) \& (\tau_{\lt m }L \oplus H_m, \hat{\el})\\
(\tau_{\leq m }L', \tleq{m} \el') \& (\tau_{\lt m }L' \oplus H_m', \hat{\el}^{\prime})\\
}; 
\path[->,font=\scriptsize] 
(m-1-1) edge node[auto]{$\hat{q}$} node[auto,below] {$\cong$} (m-1-2)
(m-1-1) edge node[auto,swap] {$\tleq{m} f$} (m-2-1)
(m-1-2) edge node[auto] {$\tlt{m} f \oplus H(f)$} (m-2-2)
(m-2-1) edge node[auto]{$\hat{q}'$} node[below] {$\cong$} (m-2-2)
;
\end{tikzpicture}
\end{equation}

% Recall from Remark \ref{rmk:tower_decomp2} that the $L_\infty$-structure on the direct sum $\tau_{\lt m }L \oplus H_m$ of graded vector spaces is not the product structure. However, 
It follows from the definitions of $\hat{q}$ and $\hat{\el}$ provided in Remarks \ref{rmk:tower_decomp1} and \ref{rmk:tower_decomp2}, respectively, that the linear projections
\begin{equation} \label{eq:int_split1}
\begin{split}
\pr_{\tlt{m}L} &= q_{\leq m} \circ \hat{q}^{-1} \maps (\tau_{\lt m }L \oplus H_m, \hat{\el})  \to 
(\tau_{\lt m }L,\tau_{\lt m } \el) \\
\pr_{\tlt{m}L'} &= q'_{\leq m} \circ \hat{q}^{\prime -1} \maps 
(\tau_{\lt m }L' \oplus H'_m, \hat{\el'})  \to (\tau_{\lt m }L',\tau_{\lt m } \el')
\end{split}
\end{equation}
are strict $L_\infty$-morphisms.

We will now show that $\sint \tlt{m} f \oplus H(f)$ is a fibration by 
applying results based on the discussion following Prop.\ \ref{prop:int_tower_andre} above.
Let $k >0$ and $0 \leq j \leq k$. For the sake of brevity, let
\begin{align*} 
X&:= \sint \tau_{\lt m }L \oplus H_m,  &X'&:= \sint \tau_{\lt m }L'
                                            \oplus H'_m \\ 
Y&:= \sint \tau_{\lt m }L,   &Y'&:= \sint \tau_{\lt m }L'.
\end{align*}
Via Cor. \ref{cor:tower_MCelts}, we tacitly make the following identifications:
\begin{equation*}
\begin{split}
X_k &= \Bigl\{ (\sigma, \rho) \in Y_k \times 
\bigl(H_m \tensor \Omega^{m+1}(\Delta^k) \bigr) ~ \vert ~ d_{\dR}\rho =\gamma(\sigma)
\Bigr \} \\
X(\Lambda^k_j)&= \Bigl\{ (\eta, \mu) \in Y(\Lambda^k_j)  \times \bigl(H_m \tensor \Omega^{m+1}(\Lambda^k_j) \bigr) ~ \vert ~ d_{\dR}\mu =\gamma(\eta)
\Bigr \},
\end{split}
\end{equation*}
and the analogous identifications for $X'_k$ and $X'(\Lambda^k_j)$. 
Furthermore, in what follows, we adopt the notation used in the commutative diagram \eqref{diag:tower_MCelts2} in Cor. \ref{cor:tower_MCelts}. For example,  we identify $\tlt{m} f^{\Om}_\ast \oplus H(f)\tensor \id_\Omega$
with the map
\begin{equation} \label{eq:map_rewrite}
\tlt{m}f \tensor \id_\Om \, \times \, H(f)\tensor \id_\Om \maps X_k \to X'_k.
\end{equation}
Let $\jmath \maps  \Horn{k}{j} \hookrightarrow \Delta^k$ denote the inclusion. 
We wish to show that the map
\[
\bigl(\tlt{m}f \tensor \id_\Om \, \times \, H(f)\tensor \id_\Om  ~ ,~ \jmath^\ast \bigr) \maps
X_k \to X'_k \times_{X'(\Horn{k}{j})} X(\Horn{k}{j})
\]
is a cover, i.e.\ a surjective submersion. Note that the pullback on the right hand side above is indeed a manifold since the projection $X'_k \to X'(\Horn{k}{j})$ is a cover. By hypothesis, $\sint \tlt{m}f$ is a fibration hence the usual map
\[
Y_k \to Y'_k \times_{Y'(\Horn{k}{j})} Y(\Horn{k}{j})
\]  
is a cover. The integrations of the  $L_\infty$-morphisms $\pr_{\tlt{m}L}$ and $\pr_{\tlt{m}L'}$ from Eq.\ \ref{eq:int_split1} induce a map 
$X'_k \times_{X'(\Horn{k}{j})} X(\Horn{k}{j}) \to Y'_k \times_{Y'(\Horn{k}{j})} Y(\Horn{k}{j})$.
Let $Z_1$ denote the following pullback
\begin{equation} \label{diag:int_split2b}
\begin{tikzpicture}[descr/.style={fill=white,inner sep=2.5pt},baseline=(current  bounding  box.center)]
\matrix (m) [matrix of math nodes, row sep=2em,column sep=2em,
  ampersand replacement=\&]
  {  
Z_1  \& Y_k \\
X'_k \times_{X'(\Horn{k}{j})} X(\Horn{k}{j}) \& Y'_k \times_{Y'(\Horn{k}{j})} Y(\Horn{k}{j}) \\
}; 
  \path[->,font=\scriptsize] 
   (m-1-1) edge node[auto] {$$} (m-1-2)
   (m-1-1) edge node[auto,swap] {$\pi^{Z_1}$} (m-2-1)
   (m-1-2) edge node[auto] {$$} (m-2-2)
   (m-2-1) edge node[auto] {$$} (m-2-2)
  ;

%begin pullback symbol%
  \begin{scope}[shift=($(m-1-1)!.25!(m-2-2)$)]
  \draw +(-0.25,0) -- +(0,0)  -- +(0,0.25);
  \end{scope}
  %end pullback symbol%
\end{tikzpicture}
\end{equation}
Since the vertical map on the right hand side above is a cover, so is
$\pi^{Z_1}$. Hence, $Z_1$ is a manifold. 
% Applying \eqref{eq:xk-xkj}
% and \eqref{eq:xx'} \eqref{eq:yy'},
%we notice that an element of $Z_1$
Note that an element of $Z_1$
\[
\bigl( (\sigma', \rho'), (\eta, \mu), \theta \bigr) \in Z_1
\]
consists of: 
\begin{itemize}
\item a pair $(\sigma', \rho') \in \MC \bigl(\tlt{m}L' \tensor \Omega(\Delta^k) \bigr) \times 
\bigl(H'_m \tensor \Omega^{m+1}(\Delta^k) \bigr)$, 
\item a pair $(\eta, \mu) \in \MC \bigl(\tlt{m}L \tensor \Omega(\Horn{k}{j}) \bigr) \times 
\bigl(H_m \tensor \Omega^{m+1}(\Horn{k}{j}) \bigr)$, and 
\item a Maurer-Cartan element $\theta \in \MC \bigl(\tlt{m}L \tensor \Omega(\Delta^k) \bigr)$ 
\end{itemize}
such that the following equations hold:
% \begin{equation} \label{eq:int_split2}
% \begin{split}
% d_{\dR}\mu  = \gamma(\eta), &\quad  d_{\dR}\rho' = \gamma'(\sigma')\\
% \jmath^\ast \sigma' = (\tlt{m} f \tensor \id)(\eta), & \quad \jmath^\ast \rho' = (H(f) \tensor \id)(\mu)\\ 
% (\tlt{m} f \tensor \id)(\theta) = \sigma', & \quad \jmath^\ast \theta = \eta.
% \end{split}
% \end{equation}
\begin{align} \label{eq:int_split2a}
&d_{\dR}\mu  = \gamma(\eta) &  &d_{\dR}\rho' = \gamma'(\sigma')\\ \label{eq:int_split2b}
&\jmath^\ast \sigma' = (\tlt{m} f \tensor \id_\Omega)(\eta) &  &\jmath^\ast \rho' = (H(f) \tensor \id_\Omega)(\mu)\\ \label{eq:int_split2c}
&(\tlt{m} f \tensor \id_\Omega)(\theta) = \sigma' &  &\jmath^\ast \theta = \eta.
\end{align}

Next, we recall that Prop.\ \ref{prop:int_tower_andre} implies that the integration of the $L_\infty$-morphism $q_{\leq m}$ in diagram \eqref{diag:int_split2} is a fibration. It then follows from Eq.\ \ref{eq:int_split1} that 
\[
\sint \pr_{\tlt{m} L} \maps X \to Y
\] 
is also a fibration, and hence the induced map
\[
X_k \to Y_k \times_{Y(\Horn{k}{j})} X(\Horn{k}{j})
\] 
is a cover. Equation \ref{eq:int_split2c} implies that the projection $Z_1 \to X(\Horn{k}{j})$ and the vertical
map $Z_1  \to Y_k$ in \eqref{diag:int_split2b} 
%composition of the projections $Z_1 \to X_k \xto{\pr_Y} Y_k$ 
induce a map $Z_1 \to Y_k \times_{Y(\Horn{k}{j})} X(\Horn{k}{j})$.
Let $Z_2$ denote the following pullback
\begin{equation} \label{diag:int_split2c}
\begin{tikzpicture}[descr/.style={fill=white,inner sep=2.5pt},baseline=(current  bounding  box.center)]
\matrix (m) [matrix of math nodes, row sep=2em,column sep=2em,
  ampersand replacement=\&]
  {  
Z_2  \& X_k \\
Z_1 \& Y_k \times_{Y(\Horn{k}{j})} X(\Horn{k}{j}). \\
}; 
  \path[->,font=\scriptsize] 
   (m-1-1) edge node[auto] {$$} (m-1-2)
   (m-1-1) edge node[auto,swap] {$\pi^{Z_2}$} (m-2-1)
   (m-1-2) edge node[auto] {$$} (m-2-2)
   (m-2-1) edge node[auto] {$$} (m-2-2)
  ;

%begin pullback symbol%
  \begin{scope}[shift=($(m-1-1)!.25!(m-2-2)$)]
  \draw +(-0.25,0) -- +(0,0)  -- +(0,0.25);
  \end{scope}
  %end pullback symbol%
\end{tikzpicture}
\end{equation} 
As in the previous case, it follows that $\pi^{Z_2}$ is a cover. 
% Applying \eqref{eq:xk-xkj}
% and \eqref{eq:xx'} \eqref{eq:yy'},
An element of $Z_2$
\[
\bigl( (\sigma', \rho'), (\eta, \mu), (\theta, \theta_H) \bigr) \in Z_2
\]
consists of an element $\bigl( (\sigma', \rho'), (\eta, \mu), \theta \bigr) \in Z_1$
and  a $(m+1)$-form $\theta_H \in H_m \tensor \Omega^{m+1}(\Delta^k)$ such that
the following equations hold:
\begin{equation} \label{eq:int_split3}
d_{\dR} \theta_H = \gamma(\theta), \quad \jmath^\ast \theta_H = \mu.
\end{equation} 
In general, we might not have $(H(f) \tensor \id_\Om)(\theta_H) = \rho'$. 
However, it follows from the definition of the map \eqref{eq:map_rewrite} 
that
% the commutativity of the diagram in Cor.\ \ref{cor:tower_MCelts} implies that
\[
d_{\dR} (H(f) \tensor \id_\Om)(\theta_H) = \gamma' \bigl( (\tlt{m} f \tensor \id_\Om) \theta \bigr).
\]
Combining this with Eqs.\ \ref{eq:int_split2a}--\ref{eq:int_split2c} and Eq.\ \ref{eq:int_split3}, we deduce that
\[
\begin{split}
% d_{\dR} (\tlt{m} f \tensor \id_\Om)(\theta_H) &= \gamma' \bigl( (\tlt{m} f \tensor \id_\Om) \theta \bigr)\\
d_{\dR} (H(f) \tensor \id_\Om)(\theta_H) %&= \gamma' \bigl( (\tlt{m} f \tensor \id_\Om) \theta \bigr)\\
= \gamma'(\sigma')=d_{\dR}\rho'.
\end{split}
\]
Hence, we have a smooth map
\[
\phi \maps Z_2 \to H'_m \tensor \Omega^{m+1}_{\cl}(\Delta^k), 
\quad \phi\bigl( (\sigma', \rho'), (\eta, \mu), (\theta, \theta_H) \bigr) := \rho' -
(H(f) \tensor \id_\Omega)(\theta_H),
\]
where $\Omega^{m+1}_{\cl}(\Delta^k)$ denotes the Banach space of closed $(m+1)$-forms on $\Delta^k$. 

By hypothesis, $H(f) \maps H_m \to H'_m$ is surjective. Since $H_m$ and $H'_m$ are finite dimensional,
it follows from Lemma 5.9 in \cite{Henriques:2008} that 
\[
H(f) \tensor \id_\Om \maps H_m \tensor \Omega_{\cl}(\Delta^\bullet) \to  H'_m \tensor \Omega_{\cl}(\Delta^\bullet)
\]
is a fibration between simplicial Banach spaces. Hence, the obvious map
\begin{equation} \label{eq:int_split3.1}
H_m \tensor \Omega_{\cl}(\Delta^k) \to H'_m \tensor \Omega_{\cl}(\Delta^k) \times_{H'_m \tensor \Omega_{\cl}(\Horn{k}{j}) }
H_m \tensor \Omega_{\cl}(\Horn{k}{j})
\end{equation}
is a cover. Furthermore, since $\jmath^\ast (H(f) \tensor \id_\Om)(\theta_H) - \jmath^\ast \rho' =0$, 
the map $\phi \maps Z_2 \to H'_m \tensor \Omega^{m+1}_{\cl}(\Delta^k)$ together with the  constant
function $Z_2 \xto{0} H'_m \tensor \Omega^{m+1}_{\cl}(\Horn{k}{j})$ induce a map from $Z_2$ into the pullback appearing in \eqref{eq:int_split3.1}.

Finally, let $Z_3$ denote the following pullback
\begin{equation} \label{diag:int_split3}
\begin{tikzpicture}[descr/.style={fill=white,inner sep=2.5pt},baseline=(current  bounding  box.center)]
\matrix (m) [matrix of math nodes, row sep=2em,column sep=2em,
  ampersand replacement=\&]
  {  
Z_3  \& H_m \tensor \Omega_{\cl}(\Delta^k)  \\
Z_2 \& H'_m \tensor \Omega_{\cl}(\Delta^k) \times_{H'_m \tensor \Omega_{\cl}(\Horn{k}{j}) }
H_m \tensor \Omega_{\cl}(\Horn{k}{j}) \\
}; 
  \path[->,font=\scriptsize] 
   (m-1-1) edge node[auto] {$$} (m-1-2)
   (m-1-1) edge node[auto,swap] {$\pi^{Z_3}$} (m-2-1)
   (m-1-2) edge node[auto] {$$} (m-2-2)
   (m-2-1) edge node[auto] {$$} (m-2-2)
  ;

%begin pullback symbol%
  \begin{scope}[shift=($(m-1-1)!.25!(m-2-2)$)]
  \draw +(-0.25,0) -- +(0,0)  -- +(0,0.25);
  \end{scope}
  %end pullback symbol%
\end{tikzpicture}
\end{equation} 
It follows that $Z_3$ is a manifold and that $\pi^{Z_3}$ is a cover. An element of $Z_3$
\[
\bigl( (\sigma', \rho'), (\eta, \mu), (\theta, \theta_H), \rho \bigr) \in Z_3
\]
consists of an element $\bigl( (\sigma', \rho'), (\eta, \mu), (\theta,\theta_H) \bigr) \in Z_2$
and a closed $(m+1)$-form $\rho \in H_m \tensor \Omega_{\cl}^{m+1}(\Delta^k)$ such that
%\begin{equation} \label{eq:int_split4}
\[
\bigl(H(f) \tensor \id_\Om \bigr)(\rho) =  \rho' - (H(f) \tensor \id_\Om)(\theta_H), \quad \jmath^\ast \rho = 0.
\]
%\end{equation}
Thus we have a smooth map $\chi \maps Z_3 \to X_k$:
\[
\chi\bigl( (\sigma', \rho'), (\eta, \mu), (\theta, \theta_H), \rho \bigr):= 
(\theta, \theta_H + \rho)
\]
that fits into the following commutative diagram:
\[
\begin{tikzpicture}[descr/.style={fill=white,inner sep=2.5pt},baseline=(current  bounding  box.center)]
\matrix (m) [matrix of math nodes, row sep=2em,column sep=2em,
  ampersand replacement=\&]
  {  
Z_3 \& Z_2 \& Z_1 \& X'_k \times_{X'(\Horn{k}{j})} X(\Horn{k}{j}) \\
X_k \&     \&     \& \\
}; 
  \path[->,font=\scriptsize] 
   (m-1-1) edge node[auto] {$\pi^{Z_{3}}$} (m-1-2)
   (m-1-2) edge node[auto] {$\pi^{Z_2}$} (m-1-3)
   (m-1-3) edge node[auto] {$\pi^{Z_1}$} (m-1-4)
   (m-1-1) edge node[auto,swap] {$\chi$} (m-2-1)
   (m-2-1) edge node[auto,swap] {$\bigl( \tlt{m}f \tensor \id_\Om \, \times \, H(f)\tensor \id_\Om ~ ,~ \jmath^\ast \bigr)$} (m-1-4)
  ;
\end{tikzpicture}
\]

Note that $\chi$ is surjective. Indeed, if $(\theta, \theta_H) \in X_k$ then set:
\[
\begin{split}
(\sigma',\rho') &= \bigl( \tlt{m}f \tensor \id_\Om \, \times \, H(f)\tensor \id_\Om \bigr)(\theta, \theta_H),\\
(\eta, \mu) &= \jmath^\ast(\theta, \theta_H),\\
\rho &=0,
\end{split}
\]
in order to obtain a pre-image of
$(\theta, \theta_H)$. Furthermore, the composition 
\[
\bigl( \tlt{m}f \tensor \id_\Om \, \times \, H(f)\tensor \id_\Om ~ ,~ \jmath^\ast \bigr) \circ \chi
\]
is a surjective submersion, by construction. It then follows from
Lemma \ref{lem:surj_sub_2of3} that 
$\bigl( \tlt{m}f \tensor \id_\Om \, \times \, H(f)\tensor \id_\Om ~ ,~ \jmath^\ast \bigr)$
is a surjective submersion. This concludes the proof of the theorem.
\end{proof}

%\subsection{Remarks on integrating fibrations} 
% Even though the iCFO structure on $\lnaft$ originates from the projective model structure on $\Ch_{\geq 0}(\R)$, it seems unreasonable to expect that the integration functor would preserve all fibrations, even in the abelian case. If $(L,\el_1)$ is a chain complex, then $\int L$ is obvious much ``larger'' than the simplicial vector space $\mathrm{DK}(L)$ produced by the Dold-Kan construction.      
% Getzler \cite{Ezra-infty} showed, however, that $\mathrm{DK}(L)$ embeds into $\int L$, and that this embedding is a homotopy equivalence. 

\begin{remark} \label{rmk:fib_not_int}
Not every fibration in $\lnaft$ integrates to a fibration of Lie $\infty$-groups.
If $f \maps \h \to \g$ is a Lie algebra morphism, then $f$ is a fibration between Lie 1-algebras, since it is trivially surjective in all positive degrees. 
Since $(\int \h)_{0} = (\int \g)_{0} = \pt$, the integrated map
$\sint f \maps \int \h \to \int \g$ satisfies the Kan condition $\Kan(1,1)$ if and only if 
$(\sint f)_1 \maps (\int \h)_{1} \to (\int \g)_{1}$ is a surjective submersion.
Let $H$ and $G$ be the 1-connected Lie groups integrating $\h$ and $\g$, respectively.
From Example \ref{ex:flatconn}, it follows that $(\sint f)_1$ is a surjective submersion 
if and only if the induced map  $f^I \maps P_eH \to P_eG$ between based path spaces is a surjective submersion. So if, for example, $f \maps \h \to \g$ is the inclusion of a proper Lie subalgebra $\h$ into $\g$, then $\sint f$ is not a fibration.

On the other hand, we presently do not have an example which demonstrates that the identification $H_0(L) \cong \ker H(f) \oplus H_0(L')$ is necessary in order for $f$ to be an integrable fibration. Although we do use this property of quasi-split fibrations 
when we establish the base case in the proof of Thm.\ \ref{thm:int_split}, it might be possible that the integration functor is exact with respect to a larger class of surjective fibrations. 
\end{remark}

\subsection{Integrating weak equivalences}\label{sec:int_weq}
Next, we show that weak equivalences in $\lnaft$, i.e. $L_\infty$-quasi-isomorphisms, between finite type $L_\infty$-algebras integrate to weak equivalences in $\LnG{\infty}$.  
Our proof is based on the following result of Henriques, which involves
the simplicial homotopy groups $\pi^{\spl}_{\ast}(X)$ of a $k$-group $X$ defined in Def.\ \ref{def:andre_pi_n}.

\begin{proposition}[Thm.\ 6.4 \cite{Henriques:2008}]\label{prop:les}
Let $(L,\el) \in \lnaft$, and let $G$ be the 1-connected Lie group integrating the Lie
  algebra $H_0(L)$. Then $\pi^{\spl}_{1}
  \bigl( \sint L \bigr) \cong G$, and there is a long exact
  sequence of (sheaves of) groups
\begin{equation} \label{eq:les}
\begin{split}
\cdots \to \pi^{\spl}_{n+1} \bigl( \sint L \bigr) \to \pi_{n+1}(G,e) \to H_{n-1}(L)
\to \pi^{\spl}_{n} \bigl( \sint L \bigr) \to \pi_{n}(G,e) \to \cdots\\
\cdots \to  \pi^{\spl}_{3} \bigl( \sint L \bigr) \to \pi_{3}(G,e) \to
H_{1}(L) \to \pi^{\spl}_{2} \bigl( \sint L \bigr) \to \pi_{2}(G,e).
\end{split}
\end{equation}
\end{proposition}
The long exact sequence \eqref{eq:les} is functorial with respect to morphisms of Lie $n$-algebras. Although this is not shown in \cite{Henriques:2008}, it 
does follow from the arguments made there, as we now explain.

The sequence \eqref{eq:les} can be constructed by applying the integration functor of Prop/Def.\ \ref{pdef:int} to the Postnikov tower \eqref{diag:tower} of the Lie $n$-algebra $L$.
This gives a tower of Kan fibrations between  Lie $\infty$-groups \cite[Thm.\ 5.10]{Henriques:2008}.
As discussed in the remarks preceding Cor.\ 6.5 in \cite{Henriques:2008},
the spectral sequence associated to this tower is
\begin{equation} \label{eq:specseq}
E^{1}_{m,k}=\pi^{\spl}_m\Bigl(\int H_{k-1}(L)[k-1] \Bigr) \Rightarrow \pi^{\spl}_{m+k} \Bigl( \int L \Bigr)
\end{equation}
where $H_{k-1}(L)[k-1]$ is the Lie $n$-algebra with only $H_{k-1}(L)$ in degree $k-1$. 
Prop.\ \ref{prop:lna_tower} implies that the 
spectral sequence is functorial. The spectral sequence is also sparse. As a result, it reduces to a long exact sequence, which becomes \eqref{eq:les} after identifications are
made between certain simplicial homotopy groups, and the Lie group $G$ and its homotopy groups.
Therefore, to verify the functoriality of \eqref{eq:les}, it remains to check the naturality of these identifications.

% Following \cite[Example 5.5]{Henriques:2008}, we consider the Lie $\infty$-group $\int \g$, where $\g$ is the Lie algebra $\tau_{\leq 0} L=H_0(L)$. For each $m \geq 0$, we have a natural identification
% %there is a natural 1-1 correspondence, as in Lemma \ref{lem:MC_CE}:
% \[
% \left(\int \g \right)_m= \MC \bigl( \g \tensor \Omega^{\ast}(\Delta^m) \bigr)
%  \cong \Bigl( \g \tensor \Omega^{1}(\Delta^m) \Bigr)^\flat
% % \left(\int \g \right)_m= \hom_{\mathrm{cdga}} \bigl( C^\ast(\g),\Omega^{\ast}(\Delta^m) \bigr)
% % \cong \Bigl( \Omega^{1}(\Delta^m)\tensor \g \Bigr)^\flat
% \]
% between Maurer-Cartan elements
% and flat connections on the trivial bundle $G \times \Delta^n \to \Delta^n$. 
% There is also an identification
% \begin{equation} \label{eq:flatconn}
% \begin{split}
% \rho^n_\g \maps \Map(\Delta^n,G)/G &\xto{\cong} \Bigl( \Omega^{1}(\g \tensor \Delta^n) \Bigr)^\flat \\
% f \maps \Delta^n \to G \quad &\mapsto \quad \alpha_f = f^{-1}df
% \end{split}
% \end{equation}
% between the set  $\Map(\Delta^n,G)/G$ of $G$-valued $C^{r+1}$ maps, modulo constants, and the set of flat connections.

\begin{lemma} \label{lem:nat_id}
Let $f \maps (L,\el) \to (L',\el')$ be a morphism in $\lnaft$.  Let $\Phi \maps G \to G'$ be the unique homomorphism of 1-connected Lie groups integrating the Lie algebra morphism 
$\tau_{\leq 0}(f)= \tau_{\leq 0}L  \to \tau_{\leq 0}L'$. Then the following diagram commutes
\[
\begin{tikzpicture}[descr/.style={fill=white,inner sep=2.5pt},baseline=(current  bounding  box.center)]
\matrix (m) [matrix of math nodes, row sep=2em,column sep=4.5em,
  ampersand replacement=\&]
  {  
\Map(\Delta^n,G)/G \& \Map(\Delta^n,G')/G'\\
 \Bigl(\tau_{\leq 0}L \tensor \Omega^{1}(\Delta^n) \Bigr)^\flat \&
\Bigl(\tau_{\leq 0}L' \tensor \Omega^{1}(\Delta^n) \Bigr)^\flat\\
};
\path[->,font=\scriptsize] 
(m-1-1) edge node[auto] {$\Gamma \mapsto \Phi \circ \Gamma$} (m-1-2)
(m-1-1) edge node[auto,swap] {$T$} (m-2-1)
(m-2-1) edge node[auto] {$\tau_{\leq 0}(f) \tensor \id_\Om$} (m-2-2)
(m-1-2) edge node[auto] {$T'$} (m-2-2);
\end{tikzpicture}
\]
where $T$ and $T'$ are the bijections defined in Example \ref{eq:flatconn}.
\end{lemma}

\begin{proof}
The commutativity of the diagram is verified by using the following elementary facts:
(1) $\Phi$, being a homomorphism, intertwines the left multiplication on $G$ with that of $G'$, and (2) the differential of $\Phi$ at the identity is $\tau_{\leq 0}(f)$.
\end{proof}

\begin{remark} \label{rmk:nat_id}
We note that Lemma \ref{lem:nat_id} also implies that the identifications made in Example 6.2 of 
\cite{Henriques:2008}:
\[
 \pi^{\spl}_{1} \bigl (\sint H_0(L)  \bigr) \cong G, \quad  \pi^{\spl}_{k \geq 2} \bigl (\sint H_0(L) \bigr)
\cong \pi_{k \geq 2}(G,e).
\]
are also natural in $(L,\el)$.
\end{remark}
Now that the functoriality of the long exact sequence \eqref{eq:les}
has been clarified, we can prove the following:
\begin{theorem} \label{thm:int_weq} 
If $f \maps (L,\el) \to (L',\el')$
  is a weak equivalence in $\lnaft$, then the morphism
\begin{equation} \label{eq:int_weq0}
\sint f \maps \sint L \to \sint L' 
\end{equation}
is a weak equivalence of Lie $\infty$-groups.
\end{theorem}

\begin{proof}
It is sufficient to show that $\sint f$ induces an isomorphism of simplicial homotopy groups, i.e.\
the induced morphisms of group sheaves
\[
\psi_m:=\pi^{\spl}_m\bigl ( \sint f \bigr ) \maps \pi^{\spl}_m\bigl( \sint L \bigr) \to
\pi^{\spl}_m\bigl( \sint L' \bigr) \quad m > 0
\]
are isomorphisms. Indeed, if that is the case, then Prop.\ \ref{thm:homotopy_grps} 
implies that $\sint f$ is a stalkwise weak equivalence.
% For $i \leq n$, the $i$th simplicial homotopy groups of $\int_{\leq n} L$ and
% $\int_{\leq n} L'$ are equal to those of $\int L$ and $\int L'$, respectively. 
% (For $i > n$, the groups are trivial.) 
% So let $\psi \maps X \to X'$ denote the pre-truncated morphism of Lie $\infty$-groups
% \[
% %\begin{equation} 
% %\label{eq:int_weq1}
% \int \phi \maps \int L \to \int L'. 
% \]
% %\end{equation}

Let $G$ and
$G'$ denote the simply connected Lie groups integrating the Lie
algebras $H_0(L)$ and $H_0(L')$, respectively. From Remark
\ref{rmk:h0}, it follows that
\[
H_0(f_1) \maps H_0(L) \xto{\cong} H_0(L')
\]
is an isomorphism of Lie algebras. Let $\Phi \maps G \xto{\cong} G'$
denote the corresponding isomorphism of Lie groups induced by $H_0(f_1)$.
It follows from Lemma \ref{lem:nat_id} and Remark \ref{rmk:nat_id} that we have a commuting diagram
\begin{equation} \label{eq:int_weq0.2}
\begin{tikzpicture}[descr/.style={fill=white,inner sep=2.5pt},baseline=(current  bounding  box.center)]
\matrix (m) [matrix of math nodes, row sep=2em,column sep=3em,
  ampersand replacement=\&]
  {  
\pi^{\spl}_1(\sint L)  \& \pi^{\spl}_1(\sint L') \\
G \& G' \\
}; 
\path[->,font=\scriptsize] 
 (m-1-1) edge node[auto] {$\psi_1$} (m-1-2)
 (m-1-1) edge node[below,sloped] {$\cong$} (m-2-1)
 (m-1-2) edge node[above,sloped] {$\cong$} (m-2-2)
 (m-2-1) edge node[auto] {$\Phi$} node[below]{$\cong$} (m-2-2)
%  (m-1-2) edge [bend left=35] node[auto] {$x$} (m-3-3)
%  (m-2-1) edge [bend right=35] node[auto,swap] {$x_0$} (m-3-3)
;
\end{tikzpicture}
\end{equation}
Hence, $\psi_m$ is an isomorphism for $m=1$.

The functoriality of the  long exact sequence \eqref{eq:les} gives the commutative diagram (with exact rows)
\begin{equation} \label{eq:int_weq0.5}
\begin{tikzpicture}[descr/.style={fill=white,inner sep=2.5pt},baseline=(current  bounding  box.center)]
\matrix (m) [matrix of math nodes, row sep=2em,column sep=2em,
  ampersand replacement=\&]
  {  
\pi_{3}(G) \& H_{1}(L) \& \pi^{\spl}_{2} (\sint L) \& \pi_{2}(G) \\
\pi_{3}(G') \& H_{1}(L') \& \pi^{\spl}_{2} (\sint L') \& \pi_{2}(G') \\
}; 
\path[->,font=\scriptsize] 
 (m-1-1) edge node[auto] {$$} (m-1-2)
 (m-1-2) edge node[auto] {$$} (m-1-3)
 (m-1-3) edge node[auto] {$$} (m-1-4)
 (m-2-1) edge node[auto] {$$} (m-2-2)
 (m-2-2) edge node[auto] {$$} (m-2-3)
 (m-2-3) edge node[auto] {$$} (m-2-4)

 (m-1-1) edge node[auto] {$\pi_3 (\Phi)$} node[below,sloped]{$\cong$} (m-2-1)
 (m-1-2) edge node[auto] {$H_1(f_1)$} node[below,sloped]{$\cong$} (m-2-2)
 (m-1-3) edge node[auto] {$\psi_2$} (m-2-3)
 (m-1-4) edge node[auto] {$\pi_2(\Phi)$} node[below,sloped]{$\cong$} (m-2-4)
 % (m-1-1) edge node[auto,swap] {$\cong$} (m-2-1)
 % (m-1-2) edge node[auto] {$\cong$} (m-2-2)
 % (m-2-1) edge node[auto] {$\Phi$} (m-2-2)
%  (m-1-2) edge [bend left=35] node[auto] {$x$} (m-3-3)
%  (m-2-1) edge [bend right=35] node[auto,swap] {$x_0$} (m-3-3)
;
\end{tikzpicture}
\end{equation}
Since $\pi_2(G)=\pi_2(G')=0$, a simple diagram chase shows that
$\psi_2 \maps \pi^{\spl}_2(\sint L) \to \pi^{\spl}_{2}(\sint L')$ is an isomorphism.
For $m >2$, we can apply the 5-lemma to 
\begin{equation} \label{eq:int_weq2}
\begin{tikzpicture}[descr/.style={fill=white,inner sep=2.5pt},baseline=(current  bounding  box.center)]
\matrix (m) [matrix of math nodes, row sep=2em,column sep=2em,
  ampersand replacement=\&]
  {  
\cdots \& \pi^{\spl}_{m+1} (\sint L) \& \pi_{m+1}(G) \& H_{m-1}(L)
\& \pi^{\spl}_{m} (\sint L) \& \pi_{m}(G) \& \cdots\\
\cdots \& \pi^{\spl}_{m+1} (\sint L') \& \pi_{m+1}(G') \& H_{m-1}(L')
\& \pi^{\spl}_{m} (\sint L') \& \pi_{m}(G') \& \cdots\\
}; 
\path[->,font=\scriptsize] 
 (m-1-1) edge node[auto] {$$} (m-1-2)
 (m-1-2) edge node[auto] {$$} (m-1-3)
 (m-1-3) edge node[auto] {$$} (m-1-4)
 (m-1-4) edge node[auto] {$$} (m-1-5)
 (m-1-5) edge node[auto] {$$} (m-1-6)
 (m-1-6) edge node[auto] {$$} (m-1-7)
 (m-2-1) edge node[auto] {$$} (m-2-2)
 (m-2-2) edge node[auto] {$$} (m-2-3)
 (m-2-3) edge node[auto] {$$} (m-2-4)
 (m-2-4) edge node[auto] {$$} (m-2-5)
 (m-2-5) edge node[auto] {$$} (m-2-6)
 (m-2-6) edge node[auto] {$$} (m-2-7)

 (m-1-2) edge node[auto] {$\psi_{m+1}$} (m-2-2)
 (m-1-3) edge node[auto] {$\pi_{m+1}(\Phi)$} node[below,sloped]{$\cong$} (m-2-3)
 (m-1-4) edge node[auto] {$H_{m-1}(f_1)$} node[below,sloped]{$\cong$} (m-2-4)
 (m-1-5) edge node[auto] {$\psi_m$} (m-2-5)
 (m-1-6) edge node[auto] {$\pi_m(\Phi)$} node[below,sloped]{$\cong$} (m-2-6)
 % (m-1-1) edge node[auto,swap] {$\cong$} (m-2-1)
 % (m-1-2) edge node[auto] {$\cong$} (m-2-2)
 % (m-2-1) edge node[auto] {$\Phi$} (m-2-2)
%  (m-1-2) edge [bend left=35] node[auto] {$x$} (m-3-3)
%  (m-2-1) edge [bend right=35] node[auto,swap] {$x_0$} (m-3-3)
;
\end{tikzpicture}
\end{equation}
and deduce that $\psi_m$ is an isomorphism. This completes the proof.
\end{proof}

\subsection{The exactness of integration} \label{sec:int_acyc}
Let $F \maps \C \to \C'$ be a functor between iCFOs. We recall from 
Def.\ \ref{def:distingfib} and Def.\ \ref{def:exactfunctor} the notion of $F$ being exact with respect to a class of distinguished fibrations $S$ in $\C$. 
By definition, the class $S$ is required to contain all acyclic fibrations and all morphisms into the terminal object. The functor $F$ is required to preserve the terminal object and acyclic fibrations, and for all $f \in S$, $F(f)$ must be a fibration in $\C'$. Finally $F$ is required to preserve pullbacks of fibrations in $S$ along arbitrary morphisms in $\C$.

Note that it follows immediately from Def.\ \ref{def:split_fib} that the class of quasi-split fibrations in $\lnaft$ is distinguished.  The second main result of this paper is the following:
\pagebreak
\begin{theorem} \label{thm:int_exact}
\mbox{}
\begin{enumerate}
\item The integration functor $\sint \maps \LnA{n}^{\ft} \to \LnG{\infty}$
preserves finite products, weak equivalences, acyclic fibrations, and
sends quasi-split fibrations to Kan fibrations.

\item Any pullback square in $\lnaft$ of the form
\begin{equation} \label{diag:int_exact}
\begin{tikzpicture}[descr/.style={fill=white,inner sep=2.5pt},baseline=(current  bounding  box.center)]
\matrix (m) [matrix of math nodes, row sep=2em,column sep=3em,
  ampersand replacement=\&]
  {  
(\ti{L},\ti{\el}) \& (L,\el) \\
(L',\el') \& (L'',\el'') \\
}
; 
  \path[->,font=\scriptsize] 
   (m-1-1) edge node[auto] {$$} (m-1-2)
   (m-1-1) edge node[auto,swap] {$$} (m-2-1)
   (m-1-2) edge node[auto] {$f$} (m-2-2)
   (m-2-1) edge node[auto] {$$} (m-2-2)
  ;

%begin pullback symbol%
  \begin{scope}[shift=($(m-1-1)!.4!(m-2-2)$)]
  \draw +(-0.25,0) -- +(0,0)  -- +(0,0.25);
  \end{scope}
  %end pullback symbol%
\end{tikzpicture}
\end{equation}
in which $f$ is a quasi-split fibration, is mapped by the integration functor  to a pullback square in $\LnG{\infty}$.
\end{enumerate}
In particular, the functor $\int \maps \LnA{n}^{\ft} \to \LnG{\infty}$ is exact with respect to the class of quasi-split fibrations.
\end{theorem}

\begin{proof}
It follows from Cor.\ \ref{cor:MC-pullback} and Thm.\ \ref{thm:int_split} that the integration functor preserves finite products and sends quasi-split fibrations to Kan fibrations. Theorem \ref{thm:int_weq} implies that integration preserves weak equivalences. Therefore, since every acyclic fibration is a quasi-split fibration, it follows that integration preserves acyclic fibrations.

Now suppose we have a pullback diagram of finite type Lie $n$-algebras, as in \eqref{diag:int_exact},
in which $f \maps (L,\el) \to (L',\el')$ is a quasi-split fibration. Hence, $\sint f$ is a Kan fibration. Since Lie $\infty$-groups are reduced Lie $\infty$-groupoids, the hypotheses of Prop.\ \ref{prop:axiom4} are satisfied and therefore the pullback of the diagram
\[
\sint L' \xto{\sint g} \sint L'' \xleftarrow{\sint f} \sint L
\] 
exists in $\LnG{\infty}$. It then follows from Prop.\ \ref{prop:MC-pullback} that 
\[
\begin{tikzpicture}[descr/.style={fill=white,inner sep=2.5pt},baseline=(current  bounding  box.center)]
\matrix (m) [matrix of math nodes, row sep=2em,column sep=2em,
  ampersand replacement=\&]
  {  
\sint \ti{L} \& \sint L \\
\sint L' \& \sint L'' \\
}
; 
  \path[->,font=\scriptsize] 
   (m-1-1) edge node[auto] {$$} (m-1-2)
   (m-1-1) edge node[auto,swap] {$$} (m-2-1)
   (m-1-2) edge node[auto] {$\sint f$} (m-2-2)
   (m-2-1) edge node[auto] {$$} (m-2-2)
  ;

 %begin pullback symbol%
   \begin{scope}[shift=($(m-1-1)!.4!(m-2-2)$)]
   \draw +(-0.25,0) -- +(0,0)  -- +(0,0.25);
   \end{scope}
   %end pullback symbol%
\end{tikzpicture}
\]
is a pullback diagram of Lie $\infty$-groups. 

\end{proof}

\appendix
\section{Proof of Lemma \ref{lem:pobj_fib}} \label{sec:lem_proof}
We need to prove the following: For all $n\geq1$ and $0 \leq j \leq n$,
the natural inclusion
\begin{equation} \label{eq:path-horn-filling}
\Horn{n}{j}\times\Simp{1} \sqcup_{\Horn{n}{j} \times \partial \Simp{1}} 
  \Simp{n}\times \partial \Simp{1} \hookrightarrow \Simp{n} \times \Simp{1}
\end{equation} 
is a collapsible extension.

Let us first establish some notation. We fix $n$.
Via the usual triangulation, we write 
$\Delta^n \times \Delta^1$ 
as the union $\bigcup_{0 \leq  l \leq n} x_l$
of  $n+1$ $(n+1)$-simplices where:
\[
%\begin{equation}\label{eq:triang}
 x_l:= \Delta^{n+1}\bigl \{ (0,0), (1,0), \dots, (l,0), (l,1), (l+1,1), \dots,
 (n,1) \bigr \}.
%\end{equation}
\]
Also, for $0 \leq l \leq n$ and $0 \leq j \leq n$ denote by
$y^{j}_{l}$ the following face of the simplex $x_l$:
\begin{equation} \label{eq:yjl_def}
y^{j}_l := \begin{cases}
d_{j+1} x_l, & \text{if $0 \leq l \leq j$}\\
d_j x_l , & \text{if $j<l$}.
\end{cases}
\end{equation}

And for $l=-1,\ldots,n$, and $0 \leq j \leq n$, denote by $T_{j,l}
\subseteq \Delta^n \times \Delta^1$ the following simplicial subsets:
If $l=-1$, then
\[
T_{j,-1} := \Horn{n}{j}\times\Simp{1} \cup \Simp{n}\times \partial \Simp{1}
\]
and for $l \geq 0$
\[
T_{j,l} := T_{j,l-1} \cup x_{l}.
\]
Note that $T_{j,-1}$ is the left side of the inclusion
\eqref{eq:path-horn-filling}, while $T_{j,n}=\Delta^n \times \Delta^1$.
Below, we will prove the lemma by showing the inclusions $T_{j,l-1}
\sse T_{j,l}$ are collapsible extensions.
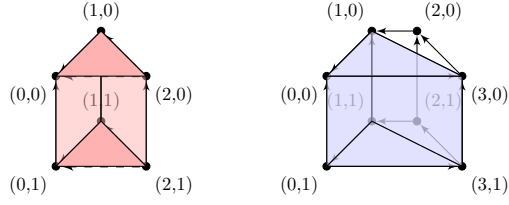
\begin{figure}[h]

\centering

\begin{tikzpicture}
[>=latex',mydot/.style={draw,circle,inner
    sep=1pt},every label/.style={scale=0.7},scale=0.6]

 \foreach \i in {0, 1}{
  \node[mydot, fill=black, label=240:\( (0 {,} \i ) \)]                at (-1, 0-2*\i)    (p\i0) {};
  \node[mydot,fill=black,label=90:$ (1 {,} \i )$]      at (0,1-2*\i) (p\i1) {};
  \node[mydot,fill=black,label=300:$(2 {,} \i )$]     at (1,0-2*\i)    (p\i2) {};
}
\begin{scope}[<-]
\foreach \i in {0, 1}{
    \draw (p\i0) --node[left,scale=0.6]{}  (p\i1);
\draw  [>=latex', thick, dashed] (p\i0)
   to node[above,scale=0.6]{} (p\i2);
    \draw (p\i1) --node[right,scale=0.6]{} (p\i2);
}
\foreach \j in {0, 1, 2} {
\draw (p0\j)  --node[left,scale=0.6]{}  (p1\j);
}
\end{scope}
% \node at (0,-3.25){$T_{1,-1}, \quad (n=2)$ };
\tikzstyle{hornfill} = [fill=red!30,fill opacity=0.5]
\tikzstyle{tbfill} = [fill=red!30,fill opacity=1]
\filldraw[hornfill](-1,0)--(0,1)--(0,-1)
                        --(-1,-2)--cycle;
\filldraw[hornfill](1,0)--(0,1)--(0,-1)
                        --(1,-2)--cycle;
\filldraw[tbfill](-1, 0)--(0, 1)--(1, 0)--cycle;
\filldraw[tbfill](-1, -2)--(0, -1)--(1, -2)--cycle;

 \foreach \i in {0, 1}{
  \node[mydot, fill=black, label=240:\( (0 {,} \i  ) \)]                at (-1+6, 0-2*\i)    (p\i0) {};
  \node[mydot,fill=black,label=100:$(1 {,} \i )$]      at (0+6,1-2*\i)
  (p\i1) {};
\node[mydot,fill=black,label=80:$(2 {,} \i )$]      at (1+6,1-2*\i) (p\i2) {};
  \node[mydot,fill=black,label=300:$(3 {,} \i )$]     at (2+6,0-2*\i)    (p\i3) {};
}
\begin{scope}[<-]
\foreach \i in {0, 1}{
    \draw (p\i0) --node[left,scale=0.6]{}  (p\i1);
    \draw (p\i1) --node[left,scale=0.6]{} (p\i2);
\draw (p\i2) --node[left,scale=0.6]{} (p\i3);
\draw (p\i3) --node[left,scale=0.6]{} (p\i0);
}
\foreach \j in {0, 1, 2, 3} {
\draw (p0\j)  --node[left,scale=0.6]{}  (p1\j);
}
\end{scope}
% \node at (6.5,-3.25){$d_2\Delta^3 \times \Delta^1 \sse \Delta^3 \times \Delta^1$,};
\tikzstyle{cfill}=[fill=blue!15, fill opacity=0.5]
\filldraw[cfill](-1+6,0)--(6,1)--(6,-1)
                        --(-1+6,-2)--cycle;
\filldraw[cfill] (6,1)--(6,-1)
                        --(2+6,-2)--(8, 0)--cycle;
\filldraw[cfill] (5,0)--(8, 0)--(2+6,-2)--(5, -2)--cycle;
\filldraw[cfill] (5, 0)--(6, 1)--(8, 0)--cycle;
\filldraw[cfill](5, -2)--(6, -1)--(8, -2)--cycle;

\end{tikzpicture}

\caption{For $n=2$,  the left-hand side depicts the geometric realization of the
  simplicial subset $T_{1,-1} \sse \Delta^2 \times \Delta^1$. For
  $n=3$, the right-hand side depicts the realization of the ``face'' $d_2 \Delta^3 \times \Delta^1$
which is missing from $T_{2,-1}$.}

\end{figure}

We will also need the following simple facts:
\begin{claim} \label{claim1}
Let $0 \leq j \leq n$ and $0 \leq l \leq n$. Every face of the
$(n+1)$-simplex $x_l$, with the possible exception of $d_{l+1}x_{l}$ and
$y^{j}_{l}$,  is contained in the simplicial subset $T_{j,l-1}$.
\end{claim}

\begin{proof}
We consider the faces $d_k x_l$. There are three cases. 
First, if $k=l$, then it is easy to verify that:
\[
%\begin{equation} \label{eq:claim1.1}
\begin{split}
d_l x_l & \subseteq \Delta^n \times \partial \Delta^1 \sse T_{j,l-1}, \quad \text{if
  $l=0$ or $l=n$}\\
d_l x_l &=d_l x_{l-1} \in T_{j,l-1} \quad \text{if $0 < l < n$.}
\end{split}
\]
%\end{equation}

Now, if $k < l$, then $d_k x_l \sse d_k \Delta^n \times \Delta^1$.
Hence, if $k \neq j$, then we have $d_k x_l \sse \Horn{n}{j} \times \Delta^1 \sse T_{j,l-1}.$
Otherwise,  we have $d_k x_l = d_j x_l = y^{j}_{l}$.

Finally, if $k\geq l+2$, then $d_k x_{l} \sse d_{k-1} \Delta^n \times
\Delta^1$. So if $k \neq j+1$, then
$d_k x_l \sse \Horn{n}{j} \times \Delta^1 \sse T_{j,l-1},$
Otherwise, we have $d_k x_l = d_{j+1}x_l = y^{j}_l$.
\end{proof}

\begin{claim}\label{claim2}
Let $\Sigma_j:=d_j\Delta^n$ for $0 \leq j \leq n$. Then for $0 \leq i
\leq n-1$, we have the inclusion of simplicial sets
\[
d_i \Sigma_j \times \Delta^1 \sse T_{j,-1}.
\]
\end{claim}
\begin{proof}
This follows directly from the simplicial
identities for face maps.
\end{proof}

\begin{claim} \label{claim3}
Let $0 \leq j \leq n$, $0 \leq l \leq n$, and let $z^{j}_{l}$ denote
the $(n-1)$-simplex
%\begin{equation}\label{eq:zjl_def}
\[
z^{j}_{l}:= \begin{cases}
d_{l+1} y^{j}_l, & \text{if $l \leq j$}\\
d_l y^{j}_l , & \text{if $j<l$},
\end{cases}
\]
%\end{equation}
where $y^j_l$ is the $n$-simplex \eqref{eq:yjl_def}. Every face of
$y^j_l$, with the possible exception of $z^{j}_{l}$, is contained in
the simplicial subset $T_{j,l-1}$. 
\end{claim}

\begin{proof}
There are a few cases to consider.\\

\underline{Case $k<l$:}  If $0 \leq j \leq l$, then we have 
$d_ky^{j}_l = d_k d_{j+1} x_l = d_j d_k x_{l}$. Claim \ref{claim1}
implies that $d_k x_l \sse T_{j,l-1}$, so therefore $d_ky^{j}_l \sse
T_{j,l-1}$.  If $j < l$, then $d_k y^{j}_{l} = d_k d_j x_l$. Hence, we
have $d_k y^{j}_l \sse d_k \Sigma_j \times \Delta^1$.
So Claim \ref{claim2} implies that $d_ky^{j}_l \sse T_{j,l-1}$.\\

\underline{Case $k=l$:}  If $j < l$, then $d_k
y^{j}_{l}=z^j_{l}$. If $l \leq j$, then we have $d_k y^{j}_{l} = d_l d_{j+1} x_l = d_j d_l x_l.$
Claim \ref{claim1} implies that $d_l x_l \sse T_{j,l-1}$, so therefore
$d_k y^{j}_{l}  \sse T_{j,l-1}$.\\

\underline{Case $k=l+1$:} If $ l \leq j$, then we have $d_k
y^{j}_{l} =z^{j}_l$. If $j < l$, then we have
$d_ky^{j}_{l} = d_{l+1} d_{j} x_l = d_j d_{l+2} x_l.$
Claim \ref{claim1} implies that $d_{l+2} x_l \sse T_{j,l-1}$. Hence, 
$d_ky^{j}_{l} \sse T_{j,l-1}$.\\

\underline{Case $k \geq l+2$:} If $l \leq j$, then either $d_k
y^{j}_l = d_j d_{k} x_l$ or $d_k y^{j}_{l}= d_{j+1} d_{k+1} x_l$,
depending on whether $k < j+1$ or $k \geq j+1$. In both cases, Claim
\ref{claim1} implies that $d_k y^{j}_l \sse T_{j,l-1}$. Finally, if $j
< l$, then $d_k y^{j}_l = d_k d_j x_l$, which implies that
$d_ky^{j}_l \sse d_k \Sigma_j \times \Delta^1.$
Hence, it follows from Claim \ref{claim2} that $d_ky^{j}_l \sse T_{j,l-1}$.
\end{proof}

We now arrive at:
\begin{proof}[Proof of Lemma \ref{lem:pobj_fib}]
Let $0 \leq j \leq n$ and $0 \leq l \leq n$. We will show that the
inclusion $T_{j,l-1} \sse T_{j,l}$
is a collapsible extension. First, we observe that
the boundary of $x_l$ is not contained in $T_{j,l-1}$:
either the face $y^{j}_{l}$ or the face $d_{l+1}x_{l}$ is missing.
If $y^{j}_{l}$ is  contained in
$T_{j,l-1}$, then Claim \ref{claim1} implies that there exists a map
\begin{equation} \label{eq:phi_push}
\psi \maps \Horn{n+1}{l+1} \to T_{j,l-1}
\end{equation}
which sends the generators of $\Horn{n+1}{l+1}$ to all the faces of $x_l$
except $d_{l+1}x_{l}$. Then the pushout of $\Horn{n+1}{l+1}
\hookrightarrow \Delta^n$ along $\psi$ is $T_{j,l}$.

On the other hand, if $y^{j}_{l}$ is  not contained in $T_{j,l-1}$
then let
\[
k= 
\begin{cases}
l+1 & \text{if $l \leq j$,} \\ 
l & \text{if $j < l$.} \\ 
\end{cases}
\]
We observe that the boundary of $y^{j}_{l}$ is not contained in $T_{j,l-1}$.
Claim \ref{claim3} implies that there exists a map 
$\phi \maps \Horn{n}{k} \to T_{j,l-1}$
which sends generators of the horn to all faces of $y^{j}_l$ except $z^{j}_l$.
The pushout of $\Horn{n}{k} \hookrightarrow \Delta^n$ along $\phi$
gives a simplicial subset $S_{j,l}:=T_{j,l-1} \cup y^{j}_{l}$. If
$x_l$ is not contained in $S_{j,l}$, then we compose $\psi$ \eqref{eq:phi_push} with the
inclusion $T_{j,l-1} \sse S_{j,l}$. The pushout of 
$\Horn{n+1}{l+1} \hookrightarrow \Delta^n$ along this composition is $T_{j,l}$. 

\end{proof}

\section{Sheaves on large categories} \label{sec:universe} 
As mentioned in Section \ref{subsec:set_theory}, there can be
set-theoretic technicalities when working with sheaves over large
categories.  For example, in this paper, we take colimits of representable
sheaves (e.g., the simplicial homotopy groups in
Def.\ \ref{def:andre_pi_n}) and this implicitly requires a
sheafification functor. Unfortunately, the usual plus-construction 
for producing a sheaf from a presheaf is not  well-defined 
for large categories, since it a priori requires taking 
colimits over proper classes.
All of this can be avoided by using Grothendieck universes, and in
particular the Universe Axiom, which allows us take colimits
in an ambient larger universe in which our classes are sets. However,
we would like our formalism to not depend on this ambient larger
universe in any way. One could have, for example, 
the colimit of a diagram of representables be a sheaf that takes values
in sets which properly reside in the ambient larger universe.

In this appendix, we verify that colimits of representable sheaves are independent 
of choice of the ambient universe, for those  pretopologies on large
categories which admit a ``small refinement'' (Def.\ \ref{def:small}).
We conclude by showing that our main example of interest: the surjective submersion
pretopology $\covers_{\subm}$ on the category $\Mfd$ of Banach
manifolds is a pretopology which admits a small refinement. 

We claim no particular originality for these results: they just provide
us with an elementary way  to comfortably ignore size issues.
Our main reference throughout for the set theory involved is
Sec.\ 1.1 of \cite{Borceux:1994}, as well as 
the preprint \cite{Low:2013}.

\subsection{Grothendieck universes}

\begin{definition} \label{def:universe}
A \textbf{universe} is a set $\cU$ that satisfies the following
axioms:
\begin{enumerate}

\item If $x \in y$ and $y \in \cU$, then $x \in \cU$.

\item If $x$ and $y$ are elements of $\cU$,  then $\{x,y\} \in \cU$.

\item If $x \in \cU$ then the power set $\pset{x}$ is an element of $\cU$.

\item If $I \in \cU$ and $\{x_\alpha\}_{\alpha \in I}$ is a family of
  elements of $\cU$, then the union $\bigcup_{\alpha \in I} x_\alpha$
  is an element of $\cU$.

\item The set of all finite von Neumann ordinals is an element of $\cU$.

\end{enumerate}

\end{definition}

We adopt the following universe axiom:

\begin{assumption}[Universe Axiom] \label{ass:universe}
For each set $x$, there exists a universe $\cU$ with $x \in \cU$.
\end{assumption}

Any universe $\cU$ is a model of ZFC, so all usual
set-theoretic constructions apply.  A \textbf{$\cU$-set} is a member
of $\cU$, and a \textbf{$\cU$-class} is a subset of $\cU$. A
\textbf{proper $\cU$-class} is a $\cU$-class which is not a $\cU$-set.

\begin{convention}
We \textbf{fix a universe $\cU$} in which we consider as the ``usual universe'' in
which we do our mathematics. We denote by $\Set$ the category of $\cU$-sets.
\end{convention}

\subsection{Sheaves on locally $\cU$-small categories}
A category $\Cat$ is \textbf{$\cU$-small} iff $\Ob(\Cat)$ and
$\Mor(\Cat)$ are $\cU$-sets.  We say $\Cat$ is \textbf{locally $\cU$-small} iff
$\hom_\Cat(x,y)$ is a $\cU$-set for all $x,y \in \Ob(\Cat)$.

\subsubsection{Universe extension}
Presheaves are only well-defined over
categories that are $\cU$-small. However, by the universe axiom, there
exists a universe $\tilde{\cU}$ such that
\[
\cU \in \tilde{\cU}. 
\]
It follows from Def.\ \ref{def:universe} that
$\pset{\cU}$ is also a $\tilde{\cU}$-set, and hence any $\cU$-class is as
well. We denote by $\widetilde{\Set}$, the category of $\tilde{\cU}$-sets.
Therefore, if $\Cat$ is a locally $\cU$-small category, then it is
$\tilde{\cU}$-small category. So, for a well-behaved theory of presheaves on $\Cat$, we
consider the category $\PSh(\Cat)$ of
 $\widetilde{\Set}$-valued functors
\[
F \maps \Cat^{\op} \to \widetilde{\Set}.
\]
The entire theory of presheaves and sheaves can be applied to
those on a locally $\cU$-small category
without worry of set-theoretic complications, provided we work in $\tilde{\cU}$.
Let $(\Cat,\covers)$ be a locally $\cU$-small category equipped with a pretopology (Def.\
\ref{def:pretopology}).  Let $\covers(-) \maps \Ob(\Cat) \to \pset{\Mor(\Cat)}$
denote the function which assigns to an object $X$ the $\tilde{\cU}$-set of
covers $\covers(X)$ of $X$. (Note that $\covers(X)$ is a priori not a $\cU$-set.)

Let $\Sh(\Cat)$ denote the category of sheaves
on $\Cat$. We have the adjunction
 \begin{equation*} %\label{eq:adjunct}
 \ell \maps \PSh(\Cat) \leftrightarrows  \Sh(\Cat) \maps i
 \end{equation*}
 where $i$ is the inclusion, and $\ell$ is the sheafification
 functor. As usual, the functor $\ell$ preserves finite limits, and
 the composite $\ell \circ i$ is naturally isomorphic to the identity
 functor.

\subsubsection{Sheafification}
Let us quickly describe the sheafification functor $\ell$ for
pretopologies in the sense of Def.\ \ref{def:pretopology}.
Let $\alpha \maps U \to X$ and $\beta
\maps V \to X$ be a covers of an object $X \in \Cat$. We say $\alpha$
\textbf{refines} $\beta$, and write $\alpha \prec \beta$ 
if there exists a morphism $f \maps U \to V$
such that $\beta \circ f =\alpha$. The axioms of a pretopology imply
that any two covers of an object $X$ have a common refinement, hence
the set $\covers(X)$ is equipped with a directed preorder.

Let $F$ be a presheaf and  $\alpha \maps U \to X$  a cover. A \textbf{matching family} for
$\alpha$ is an element $x \in F(U)$ such that the following diagram
commutes:
\[
\begin{tikzpicture}[descr/.style={fill=white,inner sep=2.5pt},baseline=(current  bounding  box.center)]
\matrix (m) [matrix of math nodes, row sep=2em,column sep=2em,
  ampersand replacement=\&]
  {  
U \times_\alpha U  \&  U  \& \\
U \& X  \& \\ [-1 cm]
\& \&  F  \\
}; 
\path[->,font=\scriptsize] 
 (m-1-1) edge node[auto] {$$} (m-1-2)
 (m-1-1) edge node[auto,swap] {$$} (m-2-1)
 (m-1-2) edge node[auto,swap] {$\alpha$} (m-2-2)
 (m-2-1) edge node[auto] {$\alpha$} (m-2-2)
 (m-1-2) edge [bend left=35] node[auto] {$x$} (m-3-3)
 (m-2-1) edge [bend right=35] node[auto,swap] {$x$} (m-3-3)
;

%begin pushout symbol%
% \begin{scope}[shift=($(m-1-1)!.65!(m-2-2)$)]
% \draw +(.25,0) -- +(0,0)  -- +(0,.-.25);
% \end{scope}
% %end pushout symbol%
\end{tikzpicture}
\]
Denote by $\Match(\alpha,F)$ the $\cUp$ set of all matching families
for the cover $\alpha$.

The \textbf{plus construction} applied to a presheaf $F$ produces a
new presheaf $F^+$, which assigns to an object $X \in \Cat$, the
$\cUp$-set
\begin{equation} \label{eq:plus}
F^+(X):= \colim_{\alpha \in \covers(X)} \Match(\alpha,F).
\end{equation}
An element of $F(X)^+$ is an equivalence class of matching
families $\overline{x_\alpha}$. Matching families $x_\alpha$
and $x_\beta$ for covers $\alpha \maps U \to X$, $\beta \maps V \to X$,
respectively, are equivalent iff there exists a cover $\gamma \maps W
\to X$ refining $\alpha$ and $\beta$ such that $x_\alpha \circ f =
x_\beta \circ g$, where $f \maps W \to U$, $g \maps W \to V$ are
morphisms such that $\gamma=\alpha \circ f=\beta \circ g$.
The sheafification functor $\ell \maps \PSh(\Cat) \to \Sh(\Cat)$
is then defined as
\[
\ell(F):=\bigl(F^+\bigr)^+.
\]

\subsection{Independence of choice of ambient universe}
Let $F \maps \Cat^\op \to \widetilde{\Set}$ be
a presheaf.  We say $F$ is a \textbf{$\cU$-presheaf} iff
$F(X)$ is a $\cU$-set for all objects $X \in \Cat$.  The analogous definition
for sheaves is clear: A sheaf $F$ is a \textbf{$\cU$-sheaf} iff the
presheaf $i(F)$ is a $\cU$-presheaf.  Since $\Cat$ is locally $\cU$-small, the
representable sheaves are $\cU$-sheaves. 

We want  $\cU$-small colimits and limits involving
$\cU$-(pre)sheaves to output an object which ``stays'' in our universe
$\cU$, and does not depend on the non-canonical choice of the ambient
universe $\cUp$. This is true for $\cU$-presheaves, since the inclusion
functor
\[
\Set \hookrightarrow \widetilde{\Set}
\]
reflects colimits and limits for all $\cU$-small diagrams \cite[Cor.\ 1.19]{Low:2013}.
And, indeed, this will also be true for $\cU$-sheaves, provided we require
our pretopology to satisfy a certain smallness condition.
 
First, we deal with limits of $\cU$-sheaves.
\begin{proposition} \label{prop:small_lim}
Let $J$ be a $\cU$-small category and $D \maps J \to \Sh(\Cat)$ a
diagram such that $D(j)$ is a $\cU$-sheaf for all objects $j \in
J$. Then $\lim D$ is a $\cU$-sheaf
which can be constructed independently from the choice of ambient universe $\cUp$. 
\end{proposition}
\begin{proof}
Since the inclusion $i \maps \Sh(C) \to \PSh(C)$ is a right adjoint,
it preserves all $\cUp$-small limits. Hence, $i \lim D \cong \lim i
\circ D$. Now $i \circ D$ is a $\cU$-small limit of 
$\cU$-presheaves. Such a limit is computed point-wise, and since $\Set$ is
complete with respect to $\cU$-small limits, it follows that 
$\lim i \circ D$ is a small presheaf.
\end{proof}

\subsubsection{A smallness condition for pretopologies}

\begin{definition}\label{def:small}
We say that a pretopology $\covers$ on a locally
$\cU$-small category $\Cat$ admits a \textbf{($\cU$-)small refinement} $\cO$
iff for every object $X \in \Cat$ there is a $\cU$-small subset
$\cO(X) \subseteq \covers(X)$ such that every cover in $\covers(X)$
is refined by a cover in $\cO(X)$.
\end{definition}

An example of pretopology which admits a $\cU$-small refinement is the
surjective submersion topology on the category of Banach
manifolds. 
\begin{example} \label{ex:banach_small}
Let $(\Mfd,\covers_{\subm})$ denote the category of ($\cU$-small) Banach manifolds
equipped with the surjective submersion pretopology. Let $M \in
\Mfd$. If $\{U_{i}\}_{i \in I}$ is an open cover of $M$, then the
morphism $\coprod_{i \in I} U_{i} \xto{\iota} M$ is a surjective
submersion, where $\iota$ is the unique map induced by the inclusions
$U_i \subseteq M$. Consider the subset of $\covers(M)$ 
\[
\cO(M):= \Bigl \{ \coprod_{i \in I} U_{i} \xto{\iota} M ~ \vert~
\text{$\{U_{i}\}_{i \in I}$ is an open cover of $M$} \Bigr \}.
\] 

Since power sets of $\cU$-sets are $\cU$-sets, 
$\cO(M)$ is a $\cU$-set for each $M \in
\Mfd$. If $f \maps N \to M$ is a surjective
submersion, then for each $x\in N$ there exists an open neighborhood
$U_x \subseteq M$ of $f(x)$ and a map $\sigma_x \maps U_x \to N$
such that $\sigma_x(f(x))=x$ and  $f \circ \sigma_x =
\id_{U_x}$. Therefore via the universal property of the coproduct, we
have a commuting diagram in $\Mfd$
\[
\begin{tikzpicture}[descr/.style={fill=white,inner sep=2.5pt},baseline=(current  bounding  box.center)]
\matrix (m) [matrix of math nodes, row sep=1em,column sep=1em,
  ampersand replacement=\&]
  {  
\coprod_{x \in N} U_{x} \& \& N \\
\& M \&\\
}; 
\path[->,font=\scriptsize] 
 (m-1-1) edge node[auto] {$\sigma$} (m-1-3)
 (m-1-1) edge node[auto,swap] {$\iota$} (m-2-2)
 (m-1-3) edge node[auto] {$f$} (m-2-2)
;
\end{tikzpicture}
\]
Hence, every cover in $\covers_{\subm}(M)$ is refined by a cover in $\cO(M)$.
\end{example}

Note we do not require $\cO$ in Def.\ \ref{def:small} to be a
pretopology. Nevertheless, taking refinements induces a directed
preorder structure on $\cO$.  If $F$ is a $\cU$-presheaf on
$(\Cat,\covers)$, and $\covers$ admits a $\cU$-small refinement, then
for each $\alpha \in \cO(X)$, we have $\Match(\alpha,F) \in \Set$ and
moreover, we have
\[
\colim_{\alpha \in \cO(X)} \Match (\alpha, F) \in \Set
\]
since this is $\cU$-small colimit of $\cU$-small sets.
Then we have the following:

\begin{theorem} \label{thm:small}
Let $(\Cat,\covers)$ be a locally $\cU$-small category equipped with
a pretopology which admits a $\cU$-small refinement. If $F$ is a
$\cU$-presheaf on $\Cat$, then its sheafification $\ell(F)$ is a $\cU$-sheaf
which can be constructed independently from the choice of ambient universe $\cUp$. 
\end{theorem}

\begin{proof}
Let $X \in \Cat$ and define a new $\cU$-presheaf $F^+_{\cO}$ via the colimit
\[
F^{+}_{\cO}(X):=  \colim_{\alpha \in \cO(X)} \Match (\alpha, F).
\]
There is the obvious function from $F^{+}_{\cO}$ to the plus
construction \eqref{eq:plus}
applied to $F$:
\[
\begin{array}{c}
F^{+}_{\cO}(X) \to F^{+}(X)\\
 \overline{x_\alpha} \mapsto \overline{x_\alpha},\\
 \end{array}
 \]
where $\overline{x_\alpha}$ on the right hand side above is the
equivalence class in
$\colim_{\alpha \in \covers(X)} \Match (\alpha, F)$ represented by the
matching family $x_\alpha$.
The theorem is proved if this assignment is a bijection. But this follows
directly from the fact that for  every cover $\alpha \in \covers(X)$, there
exists a cover in $\cO(X)$ refining $\alpha$. 
\end{proof}

Since the colimit of a diagram of sheaves is constructed by
sheafifying the point-wise colimit of the underlying diagram of
presheaves, the above theorem gives, as a corollary, the dual of Prop.\ \ref{prop:small_lim}
\begin{cor} \label{cor:small_colim}
Let $J$ be a $\cU$-small category and $D \maps J \to \Sh(\Cat)$ a
diagram such that $D(j)$ is a $\cU$-sheaf for all objects $j \in
J$. Then $\colim D$ is a $\cU$-sheaf 
which can be constructed independently of the choice of ambient universe $\cUp$. 
\end{cor}

 %%%%%%%%%


\begin{thebibliography}{10}

\bibitem{Allocca:2014} 
M. P. Allocca, Homomorphisms of $L_\infty$
  modules, {\sl J. Homotopy Relat. Struct}. {\bf 9} (2014) 285--298.

\bibitem{AGV:1972} 
M. Artin, A. Grothendieck, and J. L. Verdier,
  \textit{Th\'{e}orie des topos et cohomologie \'{e}tale des sch\'{e}mas. Tome 2}, Lecture
  Notes in Mathematics, Vol. 270, Springer-Verlag, Berlin,
  1972. S\'{e}minaire de G\'{e}om\'{e}trie Alg\'{e}brique du Bois-Marie 1963–1964 (SGA
  4).

% \bibitem{Artin-Mazur:1966} 
% M. Artin\ and\ B. Mazur, On the van Kampen
%   theorem, \textsl{Topology} {\bf 5} (1966), 179--189. 
 
\bibitem{Behrend-Getzler:2015}
K.\  Behrend and E.\ Getzler, Geometric higher groupoids and
categories. Available as \href{http://arxiv.org/abs/1508.02069}{ arXiv:1508.02069}.

\bibitem{blohmann-zhu}
C.\ Blohmann and C.\ Zhu, Higher Morita equivalence for $L_\infty$ groupoids. In preparation.

\bibitem{Brown:1973}
K. S. Brown, Abstract homotopy theory and generalized sheaf cohomology, \textsl{Trans. Amer. Math. Soc.} {\bf 186} (1973), 419--458.

% \bibitem{Cisinski:2010}
% D.-C. Cisinski, Invariance de la $K$-th\'eorie par \'equivalences d\'eriv\'ees, \textsl{J. K-Theory} {\bf 6} (2010), no.~3, 505--546. 

\bibitem{Baez-Crans:2004} 
J. C. Baez\ and\ A. S. Crans,
  Higher-dimensional algebra. VI. Lie 2-algebras, \textsl{Theory
  Appl. Categ.} {\bf 12} (2004), 492--538. %\href{http://arxiv.org/abs/math/0307263}{arXiv:math/0307263}.

\bibitem{BHH:2017} 
I.\ Barnea, Y.\ Harpaz, and G.\ Horel, Pro-categories in homotopy theory, \textsl{Algebr. Geom. Topol.}
{\bf 17} (2017), 179--189.

\bibitem{Borceux:1994}
F.\ Borceux, \textit{Handbook of categorical algebra 1: Basic Category Theory}. 3, Encyclopedia of Mathematics and its Applications, 50, Cambridge University Press, Cambridge(1994). 


\bibitem{DHR:2015}
V. A. Dolgushev, A. E. Hoffnung\ and\ C. L. Rogers, What do homotopy
algebras form?, \textsl{Adv. Math.} {\bf 274} (2015), 562--605. %\href{http://arxiv.org/abs/1406.1751}{arXiv:1406.1751}.

% \bibitem{DR:2017} 
% V. A. Dolgushev\ and\ C. L. Rogers, On an enhancement of the category of shifted $L_\infty$-algebras, \textsl{Appl. Categ. Structures} {\bf 25} (2017), no.~4, 489--503.


\bibitem{Dubuc:1981}
E. J. Dubuc, $C\sp{\infty }$-schemes, \textsl{Amer. J. Math.} {\bf 103} (1981), no.~4, 683--690.

 \bibitem{Dugger:1999}
 D.\ Dugger, Sheaves and homotopy theory. Available as \href{http://pages.uoregon.edu/ddugger/cech.html}
 {pages.uoregon.edu/ddugger/cech.html}

\bibitem{Duskin:1977} 
J. Duskin, Higher-dimensional torsors and the
  cohomology of topoi: the abelian theory, in {\it Applications of
    sheaves (Proc. Res. Sympos. Appl. Sheaf Theory to Logic, Algebra
    and Anal., Univ. Durham, Durham, 1977)}, 255--279, Lecture Notes
  in Math., 753, Springer, Berlin.

\bibitem{Dwyer-Kan:1980} W. G. Dwyer\ and\ D. M. Kan, Calculating
  simplicial localizations, \textsl{J. Pure Appl. Algebra} {\bf 18} (1980),
  no.~1, 17--35.

\bibitem{Freed-Hopkins:2013} 
D. S. Freed\ and\ M. J. Hopkins,
  Chern-Weil forms and abstract homotopy theory,
  \textsl{Bull. Amer. Math. Soc. (N.S.)} {\bf 50} (2013), no.~3,
  431--468. %\href{http://arxiv.org/abs/1301.5959}{arXiv:1301.5959}.

\bibitem{Ezra-infty} E. Getzler, Lie theory for 
nilpotent $L_\infty$-algebras, Ann. of Math. (2) {\bf 170}, 1 (2009) 271--301; 

\bibitem{Glenn:1982}
P. G. Glenn, Realization of cohomology classes in arbitrary exact categories, \textsl{J. Pure Appl. Algebra} {\bf 25} (1982), no.~1, 33--105.


% \bibitem{Goerss-Schemm}
% P. Goerss\ and\ K. Schemmerhorn, Model categories and simplicial methods, in {\it Interactions between homotopy theory and algebra}, 3--49, Contemp. Math., 436, Amer. Math. Soc., Providence, RI.

\bibitem{Henriques:2008} A. Henriques, Integrating $L\sb \infty$-algebras,
  \textsl{Compos. Math.} {\bf 144} (2008), no.~4, 1017--1045. \href{http://arxiv.org/abs/math/0603563}{ arXiv:math/0603563}.

\bibitem{Horel:2016}
G. Horel, Brown categories and bicategories,
\textsl{Homology Homotopy Appl.} {\bf 18} (2016), 217–-232. 

\bibitem{Hovey:1999}
M. Hovey, {\it Model categories}, Mathematical Surveys and Monographs, 63, Amer. Math. Soc., Providence, RI, 1999. 

\bibitem{Jardine:1983} J. F. Jardine, Simplicial objects in a
  Grothendieck topos, in {\it Applications of algebraic $K$-theory to
    algebraic geometry and number theory, Part I, II (Boulder, Colo.,
    1983)}, 193--239, Contemp. Math., 55, Amer. Math. Soc.,
  Providence, RI.

\bibitem{Joyal:1984} A.\ Joyal, \textit{Lettre d'Andr\'{e} Joyal \`{a} Alexandre Grothendieck}. Edited by G.\ Maltsiniotis, 1984. Available at \href{webusers.imj-prg.fr/~georges.maltsiniotis/ps/lettreJoyal.ps}{webusers.imj-prg.fr/$\sim$georges.maltsiniotis/ps/lettreJoyal.ps}

\bibitem{Johnstone:2002} 
P. T. Johnstone, {\it Sketches of an
    elephant: a topos theory compendium. Vol. 2}, Oxford Logic Guides,
  44, Oxford Univ. Press, Oxford, 2002.

\bibitem{Jurco:2012} B.\ Jur\v{c}o, From simplicial Lie algebras and hypercrossed complexes to differential graded Lie algebras via 1-jets, \textsl{J. Geom. Phys.} {\bf 62} (2012), no.~12, 2389--2400.

200L\bibitem{Lada-Markl:1995} 
T. Lada\ and\ M. Markl, Strongly homotopy
  Lie algebras, \textsl{Comm. Algebra} {\bf 23} (1995), no.~6,
  2147--2161. %\href{http://arxiv.org/abs/hep-th/9406095}{arXiv:hep-th/9406095}.


\bibitem{Lang:95} S. Lang, {\it Differential and Riemannian
    manifolds}, third edition, Graduate Texts in Mathematics, 160,
  Springer, New York, 1995.

\bibitem{Li:2015}
D.\ Li. \textit{Higher Groupoid Actions, Bibundles, and
  Differentiation}. Ph.D.\ thesis, Georg-August University, G\"ottingen, 2014.
Available as \href{http://arxiv.org/abs/1512.04209}{arXiv:1512.04209}.
% Available as \href{http://ediss.uni-goettingen.de/handle/11858/00-1735-0000-0022-5F4F-A}
% {ediss.uni-goettingen.de/handle/11858/00-1735-0000-0022-5F4F-A}.

\bibitem{Low:2013}
Z.\ L.\  Low, Universes for category theory. Available as
\href{http://arxiv.org/abs/1304.5227}{arXiv:1304.5227v2}.


\bibitem{MM:1994} 
S. Mac Lane\ and\ I. Moerdijk, {\it Sheaves in
    geometry and logic}, corrected reprint of the 1992 edition,
  Universitext, Springer, New York, 1994.

\bibitem{NSS:2015} T. Nikolaus, U. Schreiber\ and\ D. Stevenson,
  Principal $\infty$-bundles: presentations, \textsl{J. Homotopy
  Relat. Struct.} {\bf 10} (2015), no.~3, 565--622. %\href{http://arxiv.org/abs/1207.0249}{arXiv:1207.0249}.

\bibitem{NS:2011}
T. Nikolaus\ and\ C. Schweigert, Equivariance in higher geometry,
\textsl{Adv. Math.} {\bf 226} (2011), no.~4, 3367--3408.  
%\href{http://arxiv.org/abs/1004.4558}{arXiv:1004.4558}.

\bibitem{Noohi:2013} B. Noohi, Integrating morphisms of Lie
  2-algebras, \textsl{Compos. Math.} {\bf 149} (2013), no.~2,
  264--294. %\href{http://arxiv.org/abs/0910.1818}{arXiv:0910.1818}.

\bibitem{Pridham:2013} 
J. P. Pridham, Presenting higher stacks as
  simplicial schemes, \textsl{Adv. Math.} {\bf 238} (2013), 184--245.
%\href{http://arxiv.org/abs/0905.4044}{arXiv:0905.4044}.

\bibitem{Quillen:1969}
D.~Quillen.
\newblock Rational homotopy theory.
\newblock {\em Ann. of Math. (2)}, 90:205--295, 1969.

 \bibitem{Riehl:2014}
E. Riehl, \textit{Categorical homotopy theory}, New Mathematical Monographs, 24, Cambridge University Press, Cambridge, 2014.

\bibitem{Rogers:2017}
C.\ Rogers, Homotopical properties of the simplicial Maurer--Cartan functor.
\textsl{MATRIX Annals}, MATRIX Book Series {\bf 1} Springer, Berlin, 2018. 
\href{https://arxiv.org/abs/1612.07868}{arXiv:1612.07868}.

\bibitem{Rogers:2018} C.\ Rogers, An explicit model for the homotopy theory of finite type Lie $n$-algebras. Available as \href{https://arxiv.org/abs/1809.05999}{arXiv:1809.05999}.

\bibitem{Severa:2007}
P.\ \v{S}evera, $L_\infty$-algebras as first approximations, in {\it XXVI Workshop on Geometrical Methods in Physics}, 199--204, AIP Conf. Proc., 956, Amer. Inst. Phys., Melville, NY. 2007.

\bibitem{Severa-Siran}
P.\ \v{S}evera and M.\ \v{S}ira\v{n}, Integration of differential graded manifolds. Available as
\href{https://arxiv.org/abs/1506.04898v3}{arXiv:1506.04898}.

% \bibitem{Sweedler:1969}
% M. E. Sweedler, {\it Hopf algebras}, Mathematics Lecture Note Series, W. A. Benjamin, Inc., New York, 1969.


% \bibitem{Stevenson:2012} D. Stevenson, D\'ecalage and Kan's simplicial
%   loop group functor, \textsl{Theory Appl. Categ.} {\bf 26} (2012),
%   768--787. %\href{http://arxiv.org/abs/1112.0474}{arXiv:1112.0474}.

% \bibitem{Stevenson:2016}
% D.\ Stevenson, Classifying theory for simplicial parameterized
% groups. %Available as \href{http://arxiv.org/abs/1203.2461}{arXiv:1203.2461}.

\bibitem{Vallette:2014}
B.\ Vallette, Homotopy theory of homotopy algebras. Available as
\href{http://arxiv.org/abs/1411.5533}{ arXiv:1411.5533}.

\bibitem{Waterhouse:1975}
W. C. Waterhouse, Basically bounded functors and flat sheaves, \textsl{Pacific J. Math.} {\bf 57} (1975), no.~2, 597--610.

\bibitem{Weiss:1999}
M. Weiss, Hammock localization in Waldhausen categories, \textsl{J.\ Pure Appl.\ Algebra} {\bf 138} (1999), 185--195.

\bibitem{Wolfson:2016}
J. Wolfson, Descent for $n$-bundles, \textsl{Adv. Math.} {\bf 288} (2016), 527--575.
\href{http://arxiv.org/abs/1308.1113}{arXiv:1308.1113}.

\bibitem{Zhu:2009a}
C. Zhu, $n$-groupoids and stacky groupoids,
\textsl{Int. Math. Res. Not.} {\bf 2009}, no.~21, 4087--4141.
\href{http://arxiv.org/abs/0801.2057}{arXiv:0801.2057}.
\end{thebibliography}
\end{document}